\documentclass[reqno, 11pt]{amsart}
\pdfoutput=1
\usepackage{graphicx} 

\usepackage[french,english]{babel}
\usepackage{graphicx}
\usepackage{subcaption}
\usepackage{tabularx}
\usepackage{booktabs}
\usepackage{array}
\usepackage{amsmath}
\usepackage{amsfonts}
\usepackage{amssymb}
\usepackage{amsthm}
\usepackage{bbm}
\usepackage{float}
\usepackage[scr]{rsfso}
\usepackage{enumitem}
\usepackage{permute}
\usepackage{slashed}
\usepackage[usenames,dvipsnames]{xcolor}
\usepackage[pagebackref = false, colorlinks, linkcolor = Red, citecolor = Green, bookmarksdepth=2, linktocpage=true]{hyperref}
\usepackage[capitalize]{cleveref}
\usepackage{verbatim}

\crefformat{footnote}{#2\footnotemark[#1]#3}

\usepackage[T1]{fontenc}

\usepackage{tikz}
\usetikzlibrary{calc}
\usetikzlibrary {arrows.meta}

\newcommand{\op}{\operatorname}

\newcommand{\be}{\begin{equation}}
	\newcommand{\ee}{\end{equation}}
\newcommand{\Ga}{\Gamma}

\newcommand{\R}{\mathbb R}

\newcommand{\N}{\mathbb N}
\newcommand{\ga}{\gamma}

\newcommand{\La}{\Lambda}
\newcommand{\inte}{\op{int}}
\newcommand{\ba}{\backslash}

\newcommand{\cal}{\mathcal}
\newcommand{\br}{\mathbb R}

\newcommand{\Isom}{\op{Isom}}
\newcommand{\PSL}{\op{PSL}}
\newcommand{\F}{\cal F}

\newcommand{\bH}{\mathbb H}

\newcommand{\vol}{\op{Vol}}

\newcommand{\G}{\Gamma}
\newcommand{\m}{\mathsf{m}}

\newcommand{\h}{\mathsf h}

\newcommand{\T}{\op{T}}
\renewcommand{\frak}{\mathfrak}
\renewcommand{\v}{\mathsf v}
\newcommand{\e}{\varepsilon}

\newcommand{\BMS}{\mathsf{m}_\v}
\renewcommand{\L}{\mathcal L}
\newcommand{\fa}{\mathfrak a}

\renewcommand{\i}{\op{i}}

\renewcommand{\S}{\mathbb S}

\newcommand{\C}{\cal C}

\renewcommand{\fg}{\frak g}
\renewcommand{\c}{\mathbb C}

\newcommand{\rank}{r}

\newcommand{\supp}{\op{supp}}

\DeclareMathOperator{\interior}{int}

\DeclareMathOperator{\SL}{SL}

\DeclareMathOperator{\Ad}{Ad}

\newcommand{\LieG}{\mathfrak{g}}
\newcommand{\LieA}{\mathfrak{a}}
\newcommand{\LieN}{\mathfrak{n}}
\newcommand{\LieK}{\mathfrak{k}}

\newcommand{\LieP}{\mathfrak{p}}

\newcommand{\Fboundary}{\mathcal{F}}

\newcommand{\involution}{\mathsf{i}}
\newcommand{\growthindicator}{\psi_\Gamma}

\newcommand{\limitset}{\Lambda}

\newcommand{\limitcone}{\mathcal{L}}

\newcommand{\BR}{\mathsf{m}^{\mathrm{BR}}_\v}
\newcommand{\BRstar}{\mathsf{m}^{\mathrm{BR}_\star}_\v}
\renewcommand{\T}{\mathbb T}
\newcommand{\rankG}{\mathrm{rank}(G)}
\newcommand{\primGamma}{\Gamma_{\mathrm{prim}}}

\newtheorem{theorem}{Theorem}[section]
\newtheorem{thm}[theorem]{Theorem}

\newtheorem*{claim*}{Claim}

\newtheorem{lemma}[theorem]{Lemma}

\newtheorem{corollary}[theorem]{Corollary}

\newtheorem{cor}[theorem]{Corollary}

\newtheorem{proposition}[theorem]{Proposition}

\theoremstyle{definition}
\newtheorem{definition}[theorem]{Definition}
\newtheorem{Def}[theorem]{Definition}

\newtheorem{Q}[theorem]{Question}

\theoremstyle{remark}
\newtheorem{remark}[theorem]{Remark}
\newtheorem{rmk}[theorem]{Remark}
\newtheorem{Rmk}[theorem]{Remark}

\numberwithin{equation}{section}

\title[Jordan and Cartan spectra]{Jordan and Cartan spectra in higher rank with applications to correlations}

\author{Michael Chow}
\address{Department of Mathematics, Yale University, New Haven, CT 06511}
\email{mikey.chow@yale.edu}

\author{Hee Oh}
\address{Department of Mathematics, Yale University, New Haven, CT 06511}
\email{hee.oh@yale.edu}
\thanks{Oh is partially supported by the NSF grant No. DMS-1900101 and 2450703.}

\begin{document}
	\begin{abstract} 
		 For a given  $d$-tuple  $\rho=(\rho_1,\dots,\rho_d):\Gamma \to G$ of faithful Zariski dense convex cocompact representations of a finitely generated group $\Gamma$, we study the correlations of length spectra $\{\ell_{\rho_i(\gamma)}\}_{[\gamma]\in[\Gamma]}$ and correlations of displacement spectra $\{\mathsf{d}(\rho_i(\gamma)o,o)\}_{\gamma\in\Gamma}$. 
  We prove that for any interior vector $\mathsf v=(v_1,\dots,v_d)$ in the {\it{spectrum cone}}, there exists $\delta_\rho(\mathsf v) > 0$ such that for any $\varepsilon_1, \dots, \varepsilon_d>0$, there exist $c_1,c_2> 0$ such that		\begin{align*}		&\#\{[\gamma]\in [\Gamma]:  v_iT \le \ell_{\rho_i(\ga)} \le v_i T+\e_i, \;1 \le i \le d		\} \sim  c_1 \frac{e^{\delta_\rho (\v)T}}{ T^{{(d+1)}/{2}}};\\			&\#\{\gamma\in \Gamma:  v_iT \le \mathsf{d}(\rho_i(\gamma)o,o) \le v_i T+\varepsilon_i, \;1 \le i \le d		\} \sim  c_2 \frac{e^{\delta_\rho (\mathsf v)T}}{ T^{{(d-1)}/{2}}}.
		\end{align*}

  We deduce this result as a special case of our main theorem
on  the distribution of Jordan projections with holonomies and Cartan projections {\it{in tubes}} of an Anosov subgroup $\Gamma$ of a semisimple real algebraic group $G$.
		We also show that the growth indicator of $\Ga$ remains the same when we use Jordan projections instead of Cartan projections and tubes instead of cones, except possibly on the boundary of the limit cone. 
       
	\end{abstract}

	\maketitle
	
	\section{Introduction}
	\label{sec:Introduction}
	
	Let $G$ be a connected simple real algebraic group of rank one
	and $(X, \mathsf{d})$ the associated Riemannian symmetric space so that $G=\op{Isom}^+(X)$. 
	Let $\Gamma $ be a non-elementary  convex cocompact subgroup of $G$. Denote by $[\Gamma]$ the conjugacy classes of elements of $\Ga$.
	To each non-trivial element $[\ga]\in  [\Ga]$ corresponds a unique closed geodesic in the  locally symmetric manifold $\Ga\ba X$ whose length we denote by $\ell_\gamma$. Let $\delta_\Ga>0$ denote the critical exponent of $\Ga$. 
	The prime geodesic theorem and orbital counting theorem say that as $T \to \infty$,
	\begin{equation}\label{rone}
		\#\{[\gamma]\in [\Gamma]:  \ell_{\ga} \le T\} \sim \frac{e^{\delta_\Ga T}}{\delta_\Ga T} ;
	\end{equation}
	and\footnote{For real-valued functions $f_1,f_2$ of $T$, we write $f_1 \sim f_2 \iff \lim\limits_{T\to\infty} \frac{f_1(T)}{f_2(T)} = 1.$}  \begin{equation}\label{rtwo}
		\#\{ \gamma \in \Gamma: \mathsf{d}(\gamma o,o) \le T\} \sim \frac{e^{\delta_\Ga T}}{|m^{\op{BMS}}|\delta_\Ga}
	\end{equation}
	where $o\in X$ and $m^{\op{BMS}}$ denotes the suitably normalized Bowen-Margulis-Sullivan measure on the unit tangent bundle of $\Gamma \backslash X$. This is due to
  Margulis \cite{Mar04} and Roblin \cite{Rob03} in this generality (see also \cite{Hub59}, \cite{GW78}). We identify the unit tangent bundle of $\Gamma \backslash X$ with $\Gamma\backslash G/M$ where $M$ is a compact subgroup of $G$. For each non-trivial $[\gamma]\in [\Gamma]$, there exists a unique conjugacy class $[m_\gamma]$ of $M$, called the holonomy of $\gamma$.
 Equidistribution of holonomies was obtained for lattices by Sarnak-Wakayama \cite{SW99}, and 
 the following joint equidistribution was proved  by Margulis-Mohammadi-Oh \cite{MMO14} for all Zariski dense convex cocompact $\Ga$:
 for any conjugation invariant Borel subset $\Theta\subset M$ with smooth boundary, as $T\to \infty$,
 \begin{equation}\label{rthree}
		\#\{[\gamma]\in [\Gamma]:  \ell_{\ga} \le T, m_\ga\in \Theta\} \sim \frac{e^{\delta_\Ga T}}{\delta_\Ga T}\vol_M(\Theta)
	\end{equation}
 where $\vol_M$ denotes the Haar probability measure on $M$.

	\subsection*{Correlations of spectra.}	For a  finitely generated group $\Ga$\footnote{Throughout the paper, we assume that $\Ga$ is torsion-free.} and a $d$-tuple of faithful Zariski dense convex cocompact 
	representations $$\rho=(\rho_1, \dots, \rho_d)
	:\Ga\to G,$$
	we are interested in understanding the correlations among the
	length spectra $\{\ell_{\rho_i(\ga)}: [\gamma]\in[\Gamma]\}$ and holonomies  and the correlations among the displacement spectra $\{\mathsf{d}(\rho_i(\gamma) o,o):\gamma \in \Gamma\}$, $i=1,\cdots, d$. 
	These correlations are restricted by the {\it{spectrum cone}} $\cal L_\rho$, which is the smallest closed cone in $\R^d$ containing all vectors $(\ell_{\rho_1(\gamma)},\cdots,\ell_{\rho_d(\gamma)})$, $[\gamma]\in[\Gamma]$.  The interior $\interior\L_\rho$  is non-empty if and only if $\rho_1,\dots,\rho_d$ are independent from each other in the sense that
	no $\rho_i\circ \rho_j^{-1}:\rho_j(\Ga)\to \rho_i(\Ga)$  extends to a Lie group automorphism of $G$ for all $i\ne j$ (cf. proof of \cref{Correlation}). 
	
	\begin{thm}\label{main2} Let $\Ga$ be a finitely generated group and
		$d\in \mathbb N$.
		Let $\rho=(\rho_1, \dots, \rho_d):\Ga\to G$ be a $d$-tuple of faithful Zariski dense convex cocompact representations.
		For any  vector $\v =(v_1, \dots, v_d) \in\inte \L_\rho$, there exists $\delta_\rho(\v)>0$ such that for any $\e_1,\dots,\e_d>0$ and for any conjugation-invariant Borel sets $\Theta_1, \dots, \Theta_d \subset M$ with  null boundaries, we have as $T\to\infty$,
		
		\begin{multline*} 
			\#\{[\gamma]\in [\Gamma]:  v_iT \le \ell_{\rho_i(\ga)} \le v_iT+\e_i,\, m_{\rho_i(\ga)}\in \Theta_i, \; 1\le i\le d\} 
			\\
			\sim  c \frac{e^{\delta_\rho(\v)T}}{ T^{(d+1)/{2}}} \prod_{i=1}^d\vol_{M}(\Theta_i) \end{multline*} and 
		\begin{align*}
			&\#\{\gamma\in \Gamma:  v_iT \le  {\mathsf d}
			({\rho_i(\ga)}o, o) \le v_iT+\e_i,  \; 1 \le i \le d
			\} \sim  c\cdot c' \cdot \frac{e^{{\delta}_\rho (\v)T}}{ T^{{(d-1)}/{2}}}\end{align*} 
		for some constants $c =c(\v, \e_1,\dots,\e_d) > 0$ and $c'=c'(\v)>0$.

		Moreover, 
   \be\label{ub00} \delta_\rho(\v) \le \min_i \delta_{\rho_i(\Gamma)}v_i.\ee  If $d\ge 2$, we also have  \be\label{ine} \delta_\rho (\v)< \frac{1}{d} \sum_{i=1}^d \delta_{\rho_i(\Ga)}   v_i .\ee
	\end{thm}

 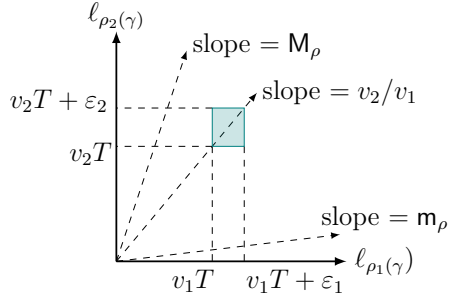
\begin{figure}[ht]
		\centering
		\scalebox{0.85}{\begin{tikzpicture}
				\begin{scope}[yscale=-0.6,xscale=1,rotate=270]
					\draw[thick,-latex] (0,0) -- (6,0);
					\draw (6.4,0) node{\large $\ell_{\rho_2(\gamma)}$};
					\draw[thick,-latex] (0,0) -- (0,3.6);
					\draw (0,4.2) node{\large $\ell_{\rho_1(\gamma})$};
					\draw[dashed,-latex] (0,0) -- (0.7,3.5);
					\draw (1.1,4.2) node{\large slope $=\mathsf{m}_\rho$};
					\draw[dashed,-latex] (0,0) -- (5.5,1.1);
					\draw (5.6,2.2) node{\large slope $=\mathsf{M}_\rho$};
					\draw[dashed,-latex] (0,0) -- (4.4,2.2);
					\draw (4.4,3.5) node{\large slope $=v_2/v_1$};
					\draw[dashed] (3,1.5) -- (3,0);
					\draw (2.8,-0.4) node{\large $v_2T$};
					\draw[dashed] (4,1.5) -- (4,0);
					\draw (4.2,-0.9) node{\large $v_2T+\varepsilon_2$};
					\draw[dashed] (3,1.5) -- (0,1.5);
					\draw (-0.5,1.2) node{\large $v_1T$};
					\draw[dashed] (3,2) -- (0,2);
					\draw (-0.52,2.8) node{\large $v_1T+\varepsilon_1$};
					\draw[teal, thin, fill = teal, fill opacity=0.2] (3,1.5) -- (4,1.5) -- (4,2) -- (3,2) -- (3,1.5);
				\end{scope}
		\end{tikzpicture}}
		\caption{The $d=2$ case: note that the size of the box is independent of $T$.} \label{fig:Correlation1}
	\end{figure} 

	\begin{Rmk}
		\begin{enumerate}
		    \item 
	Let $\Ga$ be a cocompact lattice of $\PSL_2(\br)=\Isom^+(\bH^2)$ and
		$\rho:\Ga\to\PSL_2\R$  a discrete faithful representation (whose image is necessarily a cocompact lattice).
Let
	$$ \mathsf m_\rho=\inf_{\ga\in \Ga-\{e\}} \frac{\ell_{\rho(\gamma)}}{\ell_\ga} \quad \text{ and } \quad \mathsf M_\rho= \sup_{\ga\in \Ga-\{e\}} \frac{\ell_{\rho(\gamma)}}{\ell_\ga} $$ be the minimal and maximal stretch constants of $\rho$ respectively. By a theorem of Thurston \cite[Theorem 3.1]{Thu86},  if $\rho$ is not a conjugation by a M\"obius transformation, then $\mathsf m_\rho<1<\mathsf M_\rho$. In this case, the first asymptotic on the correlations of length spectra in
		\cref{main2} for the pair $(\text{id}, \rho)$ was proved by Schwartz-Sharp \cite[Theorem 1]{SS93} for the specific direction $\v=(1,1)$ together with the bound $\delta_{(\text{id}, \rho)} (1,1)<1$. We also mention a related work of Dai-Martone \cite{DM22} which generalizes the result of Schwartz-Sharp to pairs of Hitchin representations and some specific direction. Their results do not overlap with our results.
 
  \item If we set $\ell_{\rho(\gamma)} = (\ell_{\rho_1(\gamma)}, \dots, \ell_{\rho_d(\gamma)})$, then the condition $v_iT \le \ell_{\rho_i(\gamma)} \le v_iT +\e_i$, $1 \le i \le d$, can be written as $\ell_{\rho(\gamma)} \in T\v + \prod_{i=1}^d[0,\e_i]$. We remark that if we replace the box $\prod_{i=1}^d[0,\e_i]$ with a general compact set $\cal K$ with null boundary, then we can approximate $\cal K$ with boxes and we obtain the same asymptotic in \cref{main2} with constant $c = c(\v, \cal K)$. Similarly for displacements. 
	\end{enumerate} \end{Rmk}

	\subsection*{Correlation of complex eigenvalues} When $G=\PSL_2(\c)=\Isom^+(\bH^3)$, Theorem \ref{main2} describes the correlations of complex eigenvalues of convex cocompact representations. 
	Denote by $\lambda^{\c}_g$  the complex eigenvalue of $g\in \PSL_2(\c)$ so that $|\lambda_g^\c|\ge 1$.  The argument
	of $\lambda^{\c}_g$ is well-defined as an element of $[0, \pi)$, so
	that $\lambda^\c_g =|\lambda^{\c}_g| \op{Arg}(\lambda^{\c}_g)$.
	  Let $\Gamma< \PSL_2\c$ be a  convex cocompact subgroup.
 We have $\ell_{\gamma} = 2\log|\lambda^{\c}_\gamma|$ for each non-trivial $\gamma\in \Gamma$. Since $\delta_{\Ga}$ is equal to the Hausdorff dimension $\op{dim_H}\La_{\Ga}$ of the limit set of $\Ga$ by Patterson and Sullivan (\cite{Pat76}, \cite{Sul79}), the following is a special case of \cref{main2}:
	
 \begin{cor}  Let $\Ga<\PSL_2(\c)$ be a Zariski dense convex cocompact subgroup and $\rho:\Ga\to \PSL_2(\c)$ a faithful Zariski dense convex cocompact representation. For any $ \mathsf m_\rho  <s < \mathsf M_\rho$,  there exists $$0< 2 \delta_s <  \op{dim_H}\La_\Ga + s \cdot \op{dim_H}\La_{\rho(\Ga)}  $$ such that for any $\e_1,\e_2 > 0$, there exists  a constant $c=c(s, \e_1, \e_2)>0$ such that for any
		$0<\theta_1<\theta_2<\pi$ and $0<\theta_1'<\theta_2'<\pi$, we have as $t\to\infty$,
		\begin{multline*}
  \#\left\lbrace [\gamma]\in [\Gamma] :
			\begin{tabular}{@{}l@{}}
				 $t \le |\lambda^{\c}_\ga| \le (1+\e_1) t, \,\mathrm{Arg} (\lambda^{\c}_\ga)\in [\theta_1, \theta_2]$ ,\\ $t^s\le |\lambda^{\c}_{\rho(\ga)}| \le (1+\e_2) t^s,\, \mathrm{Arg} (\lambda^{\c}_{\rho(\ga)}) \in [\theta_1', \theta_2']$
     \end{tabular}
				\right\rbrace
			\\ \sim  c \frac{t^{ 2\delta_s}}{(\log t)^{3/2}} (\theta_2-\theta_1)(\theta_2'-\theta_1') ;
		\end{multline*}
  
	\end{cor}
		Using the relation $\cosh \mathsf{d}(go,o) = \|g\|^2$ where $\|g\|$ denotes the Frobenius norm of $g \in \PSL_2\c$, we can also obtain correlations for Frobenius norms from \cref{main2}. 
	
	\subsection*{Jordan and Cartan projections in tubes.} In fact, we prove much more general results for Anosov subgroups on (i) joint equidistribution of Jordan projections {\it in tubes} and their holonomies and (ii)
	equidistribution of Cartan projections {\it in tubes},  of which Theorem \ref{main2} is a special case.
	
	Let $G$ be a connected semisimple real algebraic group. Let $P < G$ be a minimal parabolic subgroup with Langlands decomposition $P=MAN$ where $N$ is the unipotent radical of $P$, $A$ is a maximal real split torus and $M$ is a maximal compact subgroup of $P$ commuting with $A$. Let $\fg$ and $\fa$ denote the Lie algebras of $G$ and $A$ respectively, and choose a positive Weyl chamber $\fa^+$. Let $K$ be a maximal compact subgroup of $G$ such that the Cartan decomposition $G=K (\exp\LieA^+) K$ holds. Let $\mu:G\to \fa^+$ denote the Cartan projection, that is,
	$\mu(g)$ is the unique element of $\fa^+$ such that $g\in K \exp (\mu(g)) K$ for all $g\in G$.
	
	A finitely generated subgroup $\Ga<G$ is called an Anosov subgroup with respect to $P$ if there exists $C>0$ such that for all $\ga\in \Ga$,
	$$ \alpha (\mu(\ga)) \ge  C|\ga| -C^{-1}$$
	for all simple roots $\alpha$ of $(\frak g, \fa^+)$ where $|\gamma |$ denotes the word length of $\ga$ with respect to a fixed finite set of generators of $\Ga$. If we only require this condition for a subset $\theta$ of simple roots, then we get a more general definition of $\theta$-Anosov subgroup, but this will not be considered in this paper.
	Anosov subgroups of $G$ were first introduced by Labourie \cite{Lab06} for surface groups who showed that the image of a Hitchin representation is Anosov. They were later generalized by Guichard-Wienhard \cite{GW12} for Gromov hyperbolic groups. There are several equivalent characterizations of Anosov subgroups due to Kapovich-Leeb-Porti \cite{KLP17} and to Gu\'{e}ritaud-Guichard-Kassel-Wienhard \cite{GGKW17} one of which is given as above. Anosov subgroups are regarded as the higher rank generalization of convex-cocompact subgroups. 
	Schottky subgroups are Anosov and every Zariski dense subgroup of $G$ contains a Schottky, and hence Anosov subgroup \cite{Ben97}. Another important class of Anosov subgroups which is particularly relevant to this paper is the class of self-joining groups defined in \eqref{self} associated to a $d$-tuple of convex cocompact representations $\rho_1, \dots, \rho_d:
	\Ga\to G$ of a finitely generated group $\Ga$.

	In the rest of the introduction,
	let $\Ga<G$ be a Zariski dense Anosov subgroup with respect to $P$. Every nontrivial element $\gamma \in \Gamma$ is loxodromic \cite[Lemma 3.1]{GW12} and hence conjugate to an element $\exp(\lambda(\gamma))m(\gamma) $ where $\lambda(\gamma) \in \interior\LieA^+$ is the Jordan projection of $\gamma$ and $m(\gamma) \in M$.
 The conjugacy class $[m(\gamma)]\in [M]$ is uniquely determined and called the holonomy of $\gamma$. We note that $\lambda(\gamma)$ and $[m(\gamma)]$ depend only on the conjugacy class of $\gamma$. 
	The limit cone $\L=\L_\Ga\subset \fa^+$ of $\Ga$ is the smallest closed cone containing $\lambda(\Ga)$; this is a convex cone with non-empty interior \cite{Ben97}.

	By a \emph{tube} $\mathbb T$ in $\fa^+$, we mean a subset 
	of the form
	$$\T =\T(\v, \e, w)=\{u+w\in \fa^+: \|u-\br \v\| \le \e\}$$
	for some unit vector $\v\in \fa^+$, $\e>0$, and  $w\in \fa$ where
	$\|\cdot\|$ is a Euclidean norm on $\fa$. The unit vector $\v$ will be called the \emph{direction} of $\T$. We say a tube $\T$ is \emph{essential} for $\Gamma$ if its direction $\v$ belongs to the interior $\inte \L$.
	
 \begin{figure}[ht]
		\centering
		\scalebox{0.75}{\begin{tikzpicture}
				\draw[thick,-latex] (0,0) -- (1,5);
				\draw[thick,-latex] (0,0) -- (5,1);
				\draw (1.5,5) node{\large $\fa^+$};
				\draw[thick, -latex] (0,0) -- (1,1);
				\draw (1.1,1.1) node{\large $\v$};
				\draw[dashed] (-0.5,0.5) -- (0.5,-0.5);
				\draw (0.2,-0.5) node{\large $\epsilon$};
				\draw[thin] (0.25,1.25) -- (3.5,4.5);
				\draw[thin] (1.25,0.25) -- (4.5,3.5);
				\draw (4,4) node{\large $\T(\v,\varepsilon, 0)$};
				\draw[dashed] (-0.5,0.5) -- (0.25,1.25);
				\draw[dashed] (0.5,-0.5) -- (1.25,0.25);
				\draw[dashed] (0,4.243) arc (90:0:4.243);
				\draw[teal, thin, fill = teal, fill opacity=0.2] (2.458,3.458) arc (54.594:35.406:4.243) -- (1.25,0.25) -- (0,0) -- (0.25,1.25) -- (2.458,3.458);
		\end{tikzpicture}}
		\caption{} \label{fig:Tube1}
	\end{figure}

	The {\it{growth indicator}} $\growthindicator : \LieA^+ \to [0,\infty) \cup \{-\infty\}$ of $\Gamma$ is defined by $\growthindicator(0)=0$ and
	\begin{equation*}
		\growthindicator(w) = \|w\| \inf_{\text{open cones }\mathcal{C}\ni w} \tau_\mathcal{C} \qquad \text{for all non-zero $w \in \LieA^+$}
	\end{equation*}
	where $\tau_\mathcal{C}$ is the abscissa of convergence of the series $t \mapsto \sum_{\gamma \in \Gamma, \mu(\gamma) \in \mathcal{C}} e^{-t\|\mu(\gamma)\|}$.  It was first introduced by Quint and it is regarded as the higher rank generalization of the critical exponent in rank one. Quint showed that $\growthindicator|_{\LieA^+-\L} = -\infty$, $\growthindicator|_{\L}\ge0$ and $\growthindicator|_{\interior\L}>0$ \cite[Theorem 4.2.2]{Qui02a}. 
	
	The holonomy group $M_\Gamma$ of $\Ga$ is defined as the smallest closed subgroup of $M$ containing the holonomies of $\Gamma$. This is a finite index normal subgroup of $M$ \cite[Corollary 1.10]{GR07}.  We denote by  $r=r(G)$  the real rank of $G$.

	\begin{thm}[Jordan spectrum]
		\label{main1} \label{main0} 
		Let $\Gamma<G$ be a Zariski dense Anosov subgroup. 
		For any essential tube $\T$ of direction $\v$,  there exists $\kappa_{\T}>0$ such that
		for any conjugation-invariant Borel subset $\Theta\subset M$ with smooth boundary, we have as $T\to \infty$,
		$$\#\{[\gamma] \in [\Gamma]:\|\lambda(\gamma) \|\le T,\; 
		\lambda(\gamma) \in \T,\; m(\gamma)\in \Theta\} \sim \kappa_{\T} \frac{e^{\psi_\Ga(\v) T}}{T^{(r+ 1)/2}} \op{Vol}_M(\Theta\cap M_\Ga) .$$
		
		In particular,
		$$\#\{[\gamma] \in [\Gamma]:\|\lambda(\gamma) \|\le T,\; 
		\lambda(\gamma) \in \T\} \sim \frac{\kappa_{\T}}{[M:M_\Gamma]} \cdot \frac{e^{\psi_\Ga(\v) T}}{T^{(r + 1)/2}} . $$
	\end{thm}

	We also obtain a similar counting result for the Cartan projections $\mu(\Gamma)$:
	\begin{thm} [Cartan spectrum] \label{main00}
		Let $\Gamma<G$ be a Zariski dense Anosov subgroup. 
		For any essential tube $\T$ of direction $\v$, we have
		as $T\to \infty$,
		$$\#\{\gamma \in \Gamma:\|\mu(\gamma) \|\le T,\; 
		\mu(\gamma) \in \T\} \sim \frac{\kappa_\T}{|m_{\mathcal{X}_\v}|} \cdot \frac{e^{\psi_\Ga(\v) T}}{T^{(r -1)/2}}  $$
		where  $\kappa_{\T}$ is as in Theorem \ref{main0} and $m_{{\mathcal X}_\v}$
		is the finite measure defined in \eqref{eqn:BMSMeasureAndXMeasure}.
	\end{thm}

 \begin{rmk} \label{mmm} We refer to \eqref{kt} for a formula for $\kappa_\T$. For example,
 for tubes $\T$ in the maximal growth direction $\v_\Ga\in \inte\L$ such that $\psi_\Ga(\v_\G)=\max_{\|u\|=1}\psi_\Ga(u)$,
  the constant $\kappa_\T$ is  proportional to the volume of the cross section of $\T$ orthogonal to $\v_\Ga$ with the multiplicative constant depending only on $\Ga$ (see Remark \ref{ug}).
\end{rmk}

 \medskip 
  
	As an immediate consequence of Theorems \ref{main1} and \ref{main00}, we obtain a higher rank extension of  the asymptotic ratio of the number of Cartan projections to Jordan projections in rank one given by \eqref{rone} and \eqref{rtwo}.

	\begin{corollary}[Asymptotic ratio of Jordan vs. Cartan] \label{compp}
		For any essential tube $\T$ of direction $\v$, we have as $T \to \infty$,
		\be\label{comp}\frac{\#\{\gamma \in \Gamma:\|\mu(\gamma) \|\le T,\; 
			\mu(\gamma) \in \T\} }{\#\{[\gamma] \in [\Gamma]:\|\lambda(\gamma) \|\le T,\; 
			\lambda(\gamma) \in \T\} }\sim 
		\frac{[M:M_\Gamma]}{|m_{\mathcal{X}_\v}|} T. \ee 
		\end{corollary}
	
	Note that  the multiplicative constant $\tfrac{[M:M_\Gamma]}{|m_{\mathcal{X}_\v}|}$ is independent of tubes $\T$, depending only on $\Gamma$ and $\v$. We mention a related work \cite{BS} where  Jordan and Cartan spectra have the same asymptotic limit for random products.

	\begin{rmk}
		 Without the restriction to tubes,
			the asymptotic  $\#\{[\gamma] \in [\Gamma]:\|\lambda(\gamma) \|\le T,\; m(\gamma)\in \Theta\} \sim c\cdot  \frac{e^{\delta_\Ga T}}{\delta_\Ga T} \op{Vol}_M(\Theta\cap M_\Ga) $ for some $0<c<1$ was obtained by Chow-Fromm \cite[Theorem 7.3]{CF23} where $\delta_\Ga= \psi_\Ga(\v_\Ga)$.
			Similarly, the asymptotic
			$\#\{\gamma \in \Gamma:\|\mu(\gamma) \|\le T\} \sim c' e^{\delta_\Ga T}$ for some constant $c'>0$ was obtained by Sambarino \cite{Sam15};
  see also \cite[Corollary 9.21]{ELO20} for a precise description of the multiplicative constant. By comparing the constants from these results, we find that as $T\to \infty$,
  	\be\label{comp2}\frac{\#\{\gamma \in \Gamma:\|\mu(\gamma) \|\le T\} }{\#\{[\gamma] \in [\Gamma]:\|\lambda(\gamma) \|\le T\} }\sim 
		\frac{[M:M_\Gamma]}{|m_{\mathcal{X}_{\v_\Gamma}}|} T \ee 
   where $\v_\Ga$ is the maximal growth direction as defined in Remark \ref{mmm}.
	\end{rmk}
	\medskip

	\subsection*{Growth indicators using Jordan projections and tubes}
	It is natural to ask whether the growth indicator $\psi_\Ga:\fa^+\to \br\cup \{-\infty\} $ can also be defined using the Jordan projections rather than Cartan projections, or using tubes rather than cones.
	For a subset $S\subset \fa^+$, denote by
	$\tau_S$ and $\cal T_S$    respectively the abscissa of convergence of the series
	$$ t \mapsto \sum_{\gamma \in \Gamma,\; \mu(\gamma) \in S} e^{-t\|\mu(\gamma)\|}  \;\; \text{ and } \;\; t \mapsto \sum_{[\gamma] \in [\Gamma],\; \lambda(\gamma) \in S} e^{-t\|\lambda(\gamma)\|}.$$
	
	Define the following degree one homogeneous functions $\fa^+\to [0, \infty)  \cup \{-\infty\}$: for all non-zero $w\in \fa^+$, set
	\begin{equation}
		\label{growthfunctions}
		\begin{aligned}[b]
			\psi_\Ga^{\op{cones}}(w) &=\|w\|\inf_{\text{open cones }\cal C \ni w} 
			\tau_{\cal C} ;\\
   \psi_\Ga^{\op{tubes}}(w) &=\|w\|\inf_{\text{open tubes }\T \ni w} 
			\tau_{\T}; \\
			\h_\Ga^{\op{cones}}(w)&=\|w\| \inf_{\text{open cones }\mathcal{C}\ni w} \cal T_\mathcal{C}  ;  \\
			\h_\Ga^{\op{tubes}}(w)&=\|w\| \inf_{\text{open tubes }\mathcal{\T}\ni w} \cal T_\mathcal{\T} 
		\end{aligned} 
	\end{equation}
	and define $\psi_\Ga^{\op{cones}}(0)=\psi_\Ga^{\op{tubes}}(0)= \h_\Ga^{\op{cones}}(0)=\h_\Ga^{\op{tubes}}(0)=0$.
	Note that by definition, $\psi_\Ga=\psi_\Ga^{\op{cones}}$, and that
	these definitions are independent of the choice of the norm $\|\cdot\|$. All of these functions are $-\infty$ outside the limit cone $\L$. In fact, they coincide with each other on $\interior\limitcone$  as well:
	\begin{corollary}
		\label{c1}
 For any Zariski dense Anosov subgroup $\Gamma<G$, we have  $$\psi_\Gamma = \psi_\Gamma^{\op{tubes}}=\h_\Gamma^{\op{cones}}  =\h_\Gamma^{\op{tubes}}\quad\text{on $\inte  \L$} .$$ 
 We also have $\psi_\Ga=\psi_\Gamma^{\op{tubes}}$ on $\fa^+$.
 \end{corollary}
This corollary implies that
  for all unit vector $\v\in \inte\L$,
		we have
			$$\psi_\Ga^{\op{tubes}}(\v)= \inf_{\text{open tubes $\T\ni \v$}} \limsup_{T\to \infty} \tfrac{1}{T}\log \#\{\gamma \in \Gamma:\mu(\gamma)\in \T, \|\mu(\gamma) \|\le T\};$$
			$$\h_\Ga^{\op{cones}}(\v)=\inf_{\text{open cones $\cal C\ni \v$}} \limsup_{T\to \infty} \tfrac{1}{T}\log \#\{[\gamma] \in [\Gamma]:\lambda(\gamma)\in \mathcal C, \|\lambda(\gamma) \|\le T\};$$
   	$$\h_\Ga^{\op{tubes}}(\v)= \inf_{\text{open tubes $\T\ni \v$}} \limsup_{T\to \infty} \tfrac{1}{T}\log \#\{[\gamma] \in [\Gamma]:\lambda(\gamma)\in \T, \|\lambda (\gamma) \|\le T\},$$
 and   $\psi_\Ga(\v)$ is equal to any of the above.

Recall that  $\v_\Ga\in \fa^+$ denotes the unique unit vector of maximal growth for Cartan projections in cones, i.e., $\psi_\Ga(\v_\Ga)=\max_{\|u\|=1}\psi_\Ga(u)$.
\begin{cor} Jordan and Cartan projections of $\Ga$ in cones and tubes all have the same direction of maximal growth: $\psi_\Ga^{\op{tubes}}(\v_\Ga)=\max_{\|u\|=1} \psi_\Ga^{\op{tubes}}(u)$;
 $\h_\Ga^{\op{cones}}(\v_\Ga)=\max_{\|u\|=1} \h_\Ga^{\op{cones}}(u)$, and
  $\h_\Ga^{\op{tubes}}(\v_\Ga)=\max_{\|u\|=1} \h_\Ga^{\op{tubes}}(u)$.
   \end{cor}

	\subsection*{On the proofs.} We now give an outline of the
	proofs of main theorems in the introduction. We will focus on the correlations of length spectra and Jordan projections in tubes for Anosov subgroups. We first explain how to deduce the correlations of length spectra in \cref{main2} from \cref{main1}. Given $\rho=(\rho_1, \dots, \rho_d)$ as in Theorem \ref{main2}, we consider the self-joining of $\Gamma$ via $\rho$ defined by 
	\be\label{self} \Gamma_\rho = \{(\rho_1(\gamma),\dots,\rho_d(\gamma)):\gamma \in \Gamma\} .\ee 
	By the hypothesis that $\rho_1,\dots,\rho_d$ are convex cocompact, $\Gamma_\rho$ is an Anosov subgroup and
	when $\rho_1,\dots,\rho_d$ are independent from each other, $\Gamma_\rho$ is Zariski dense in the semisimple real algebraic group $\prod_{i=1}^d\Isom^+(X)$. 
	
	Indeed, the novelty of our proof of Theorem \ref{main2} is to relate it with the problem on understanding the Jordan spectrum of $\Ga_\rho$. The vector $(\ell_{\rho_1(\gamma)},\dots,\ell_{\rho_d(\gamma)})$ is the Jordan projection $\lambda(\rho(\gamma))$ of $\rho(\gamma)=(\rho_1(\ga), \cdots, \rho_d(\ga))$ and the spectrum cone $\limitcone_\rho$ coincides with the limit cone $\limitcone_{\Gamma_\rho}$. Hence given $\v = (v_1,\dots,v_d) \in \interior\limitcone_{\Gamma_\rho}$ and $\e_1,\dots,\e_d>0$ we are interested in the asymptotic behavior of
	$$\#\{[\gamma]\in [\Gamma]: \lambda(\rho(\gamma)) \in [v_1T,v_1T+\e_1]\times\cdots\times[v_dT,v_dT+\e_d]\}.$$
	As $T$ tends to $\infty$, the box $ \prod_{i=1}^d[v_iT,v_iT+\e_i]\subset\R^d$ sweeps out the following essential tube for $\Gamma_\rho$:
	$$\T(\v,\mathsf{K}) = \R_+\v + \mathsf{K} = \bigcup_{T\ge 0}\prod_{i=1}^d[v_iT,v_iT+\e_i]$$ where $\mathsf{K} \subset \R^d$ is a $(d-1)$-dimensional compact subset transverse to $\R\v$.  In fact, we count in {\it{truncated tubes}}: 
	$$\T_{T,b}=\T_T(\v,\mathsf{K},b) = \{t\v+u: u \in \mathsf{K},0\le  t \le T + b(u)\}$$
	for a continuous function $b\in C(\mathsf{K})$. By considering different functions $b$, we will be counting in truncated tubes of different shapes. This is crucial since we can realize the box as a difference of two truncated tubes of different shapes (see \cref{fig:Correlation2}). Hence \cref{main2} can be deduced from a more refined version of  Theorems \ref{main0} and \ref{main00} with $\delta_\rho=\psi_{\Ga_\rho}$. The upper bound \eqref{ine} is a direct consequence of  a result of Kim-Minsky-Oh \cite[Corollary 1.6]{KMO21} on the growth indicator $\psi_{\Ga_\rho}$.
	
	We now give an outline of the proofs of Theorems \ref{main1}. For a Zariski dense Anosov subgroup $\Ga<G$, joint equidistribution of nontrivial closed $A$-orbits in $\Gamma\backslash G/M$ and their holonomies were obtained by Chow-Fromm
	\cite{CF23} following the rank one approach of  Margulis-Mohammadi-Oh \cite{MMO14}. Their theorem gives the counting result for Jordan projections in certain types of cones with respect to an ordering given by a certain linear form of $\fa$. We follow the approach of \cite{CF23}. For truncated tubes, there is essentially only one ordering possible since tubes are associated to unique directions. One of the important features of an Anosov subgroup is that, denoting $\primGamma$ the set of all primitive elements of $\Ga$,
	each conjugacy class $[\gamma] \in [\primGamma]$ bijectively corresponds to a closed $A$-orbit $ C({\gamma}) \subset \Gamma \backslash G/M$ which is homeomorphic to a cylinder $\S^1 \times \R^{r -1}$ \cite[Lemma 4.14]{CF23}. We equip $C(\gamma)$ with the measure induced by the Lebesgue measure on $\LieA$.  For each $T>0$, we define a Radon measure $\eta_T$ on $\Gamma \backslash G/M \times [M]$ by the following: for $f \in \mathrm{C}_{\mathrm{c}}(\Gamma \backslash G/M)$\footnote{\label{continuous}For a topological space $X$, we denote by $C(X)$ (resp. $C_{\mathrm{c}}(X)$) the space of (resp. compactly supported) continuous real-valued functions on $X$.} and a conjugation-invariant Borel subset $\Theta$ of $M$, let 
	
	\begin{equation}\label{eta}
		\eta_{T}(f\otimes\mathbbm{1}_\Theta)=\sum_{
			[\gamma] \in [\primGamma],\,  \lambda(\gamma) \in \T_{T,b}} \int_{C(\gamma)} f \cdot \mathbbm{1}_\Theta(m(\gamma)).
	\end{equation}
	
	Theorem \ref{main1} follows once we 
	find an asymptotic for $\eta_T(f\otimes\mathbbm{1}_\Theta)$ whose proof we now outline.
	For $g_0 \in G$ and $\varepsilon>0$, the {\it{$\varepsilon$-flow box centered at $g_0$}} is defined as
	$$\mathcal{B}(g_0,\varepsilon):=g_0(\check{N}_\varepsilon N \cap N_\varepsilon \check{N} AM)M_\varepsilon A_\varepsilon$$
	where $\check{N}$ is the horospherical subgroup opposite to $N$. Flow boxes form a basis for the topology on $G$, so it suffices to understand the asymptotic behavior of $\eta_T(\tilde{\mathcal{B}}(g_0,\varepsilon)\otimes\Theta)$ where $\tilde{\mathcal{B}}(g_0,\varepsilon)$ is the image of $\mathcal{B}(g_0,\varepsilon)$ under the projection $G \to \Gamma \backslash G/M$. 
	We describe the main steps to find an asymptotic for $\eta_T(\tilde{\mathcal{B}}(g_0,\varepsilon)\otimes\Theta)$.
	First, we find an asymptotic for $\#\Gamma \cap g_0S_{T,b}g_0^{-1}$ where
	$$S_{T,b}= \check{N}_\varepsilon\exp(\T_{T,b})\Theta N_\varepsilon^{-1}.$$ 
	When $r=\rankG= 1$, a well-known approach for this (\cite{EM93}, \cite{OS13}, \cite{MO15}, etc; see also \cite{DRS93}) is to show that the sets 
	$ S_{T,b,\varepsilon}^- =\bigcap_{g_1,g_2 \in G_\varepsilon}g_1S_{T,b}g_2$ and $ S_{T,b,\varepsilon}^+= \bigcup_{g_1,g_2 \in G_\varepsilon}g_1S_{T,b}g_2 $ can be well-approximated by product subsets of $\check{N}\exp(\LieA^+)MN$ and to obtain an asymptotic for $\#\Gamma \cap g_0S_{T,b}g_0^{-1}$, using strong mixing of the $A$-action on $\Gamma\backslash G$ for the finite Bowen-Margulis-Sullivan (BMS) measure.  When $r \ge 2$, the BMS measure $\BMS$ associated to $\v$ is infinite and the $A$-action is not strongly mixing on $\Gamma \backslash G$. Instead, we use \emph{local mixing} (\cref{thm:DecayofMatrixCoefficients}) for the action of  one parameter family $\exp(t\mathsf{v}+\sqrt{t}u)$ for certain $u\in\LieA$, obtained in (\cite{CS23}, \cite{ELO22b}).
	The availability of the local mixing is one of the main reasons why our theorems are proved for Anosov subgroups.
	Let $$L^*(\T_{T,b}) =  \frac{\kappa_{\mathsf{v}}}{\delta_\v|m_{\mathcal{X}_\v}|}\int_{\mathsf{K}} e^{\delta_\v b(u)} \, du\cdot\frac{e^{\delta_\v T}}{T^{(\rank-1)/2}}$$
	where $\kappa_\v>0$ is as in the local mixing \cref{thm:DecayofMatrixCoefficients}. Using local mixing for $\exp(t\mathsf{v}+\sqrt{t}u)$ along with an accompanying uniformity statement, we prove that  
	$$\#\Gamma \cap g_0S_{T,b}g_0^{-1}= L^*(\T_{T,b})\left( \frac{\BMS(\tilde{\mathcal{B}}(g_0,\varepsilon))}{b_{r}(\varepsilon)}\vol_M(\Theta) (1+O(\varepsilon)) +o_T(1)\right)$$
	where $b_r(\varepsilon)$ is the volume of the $r$-dimensional Euclidean ball of radius $\varepsilon$.
	We emphasize that the fact that
	we are using the tubes in the definition of $S_{T,b}$
	is quite crucial 
	in this step of the proof, as $L^*(\T_{T,b})$ is the asymptotic of a certain integral over $\T_{T,b}$ and it is unclear how to compute this for general subsets of $\LieA^+$ other than tubes. By a wavefront-type argument, the family of subsets $g_0S_{T,b}g_0^{-1}$ and $\mathcal{V}_{T,b}$ are approximately equal to each other
	where 
	$$\mathcal{V}_{T,b} := \mathcal{B}(g_0,\varepsilon)\exp(\T_{T,b})\Theta\mathcal{B}(g_0,\varepsilon)^{-1}.$$
	Moreover, since the direction of $\T$ lies in the interior of $\fa^+$, we can apply a closing lemma for 
	regular directions to elements of $\Ga$ and approximate
	$\#\Gamma \cap \mathcal{W}_{T,b}$  using
	$\#\Gamma \cap \mathcal{V}_{T,b}$ where 
	$$\mathcal{W}_{T,b}  := \{gamg^{-1}:g\in\mathcal{B}(g_0,\varepsilon),\; a \in \exp(\T_{T,b}),\; m \in \Theta\}.$$
	Since
	$$\eta_{T}(\tilde{\mathcal{B}}(g_0,\varepsilon)\otimes\Theta)=b_{\rank}(\varepsilon)\cdot\#(\primGamma\cap\mathcal{W}_{T,b}),$$ 
	we then get the asymptotic for $\eta_{T}(\tilde{\mathcal{B}}(g_0,\varepsilon)\otimes\Theta)$, which yields the asymptotic for
	\eqref{eta} using a standard partition of unity argument:
	$$\eta_T(f\otimes\mathbbm{1}_\Theta) \sim L^*(\T_{T,b})\cdot\BMS(f) \cdot \vol_M(\Theta\cap M_\Ga).$$
	
	This process yields an asymptotic for counting Jordan projections with weights $\int_{C(\gamma)}f$. Removing these weights is a difficult problem for a general discrete subgroup; for instance, this is related to the difficulty of counting Jordan projections for lattices (and also a reason for the presence of weight in a related  work of Dang-Li \cite{DL22}).
	However for $\Gamma$ Anosov, there is a convenient choice of $f$ based on the vector bundle structure of the support of $\mathsf m_\v$ (see the proof of \cref{Counting}) so that the weight $\int_{C(\gamma)}f$ is equal to $\psi_\v(\lambda(\gamma))$ where $\psi_\v:\LieA\to\R$ is the unique linear form tangent to $\psi_\Ga$ at $\v$ and this weight can be removed as in the rank one case.
	
	To study the correlations of displacement spectra, we observe that the vector $(\mathsf{d}(\rho_1(\gamma)o,o), \dots, \mathsf{d}(\rho_d(\gamma)o,o))$ is the Cartan projection $\mu(\rho(\gamma))$ of $\rho(\gamma) \in \Gamma_\rho$. Hence we are led to count Cartan projections in tubes for Anosov subgroups. Counting Cartan projections for cones with respect to certain orderings was done by Edwards-Lee-Oh \cite{ELO20}. In proving Theorem \ref{main00}, the main technical difficulty is again to estimate certain integrals over tubes.

	\subsection*{Organization.} 
	\begin{itemize}
		\item 
		In \cref{sec:Preliminaries}, we recall the definition of $\Gamma$-conformal measures on the Furstenberg boundary and generalized BMS measures for general Zariski dense subgroups $\Gamma$ of $G$. 
		\item In \cref{sec:AnosovSubgroups}, we recall the local mixing result on Anosov homogeneous spaces.
		\item In \cref{sec:Jordan}, we  deduce the Jordan projection counting in tubes  (\cref{Counting}) from
		\cref{JointEquidistribution}. 
		\item In \cref{sec:j},  we  prove the joint equidistribution in essential tubes  of nontrivial closed $A$-orbits  and their holonomies (\cref{JointEquidistribution}).
		\item In \cref{sec:Cartan}, we prove equidistribution and counting for Cartan projections in tubes (\cref{Cartan}, \cref{CartanCount}).
		\item In \cref{sec:Correlation}, we give an application of Jordan projection counting in tubes to deduce the correlations of length spectra (\cref{Correlation}). The correlations of displacements (\cref{Correlation2}) is a similar application of Cartan projection counting in tubes.
		
		\item In \cref{sec:last}, we  deduce that the growth indicator for Anosov subgroups can be equivalently defined using Jordan projections rather than Cartan projections or using tubes rather than cones at least in the interior of the limit cone (\cref{growthrates}).  
	\end{itemize}

\medskip 
	\subsection*{Acknowledgement.}
	We would like to thank Giuseppe Martone for interesting conversations 
	about his paper with Dai \cite{DM22}. We also thank Dongryul Kim for useful comments on the preliminary version.
	\section{Preliminaries}
	\label{sec:Preliminaries}
	
	Throughout the paper, we let $G$ be a connected semisimple real algebraic group. Fixing a Cartan involution of the Lie algebra $\LieG$ of $G$,
	let $\LieG = \LieK \oplus \LieP$ be the  eigenspace decomposition corresponding to the eigenvalues $+1$ and $-1$ respectively. Let $K < G$ be the maximal compact subgroup whose Lie algebra is $\mathfrak{k}$. Let $\LieA \subset \LieP$ be a maximal abelian subalgebra and choose a closed positive Weyl chamber $\LieA^+ \subset \LieA$. We denote by $\Phi^+$ the set of positive roots for $(\LieG, \LieA^+)$. 
	
	Let $A = \exp \LieA$, $A^+ = \exp \LieA^+$, and denote $a_w = \exp(w)$ for all $w \in \LieA$. Let $M = C_K(A)$ be the centralizer of $A$ in $K$. We set
	\begin{align*}
		N&= \left\{n \in G: \lim_{t \to \infty} a_{-tw} n a_{tw} = e \text{ for all } w \in \interior\LieA^+\right\};
		\\
		\check{N} &= \left\{h \in G: \lim_{t \to \infty} a_{tw} h a_{-tw} = e \text{ for all } w \in \interior\LieA^+\right\},
	\end{align*}
	and  $\LieN = \log N$ and $\check{\LieN}=\log\check{N}$. Let 
	$P=MAN$ and
	$$\mathcal{F}: = G/P \cong K/M$$
	denote the \emph{Furstenburg boundary} of $G$ where the isomorphism $G/P \cong K/M$ is given by the Iwasawa decomposition $G \cong K \times A \times N$. Let $\cal W=N_K(A)/M$ denote the Weyl group. Let $w_0 \in K$ be a representative of the element in $\cal W$ such that $\Ad_{w_0}(\LieA^+) = -\LieA^+$.	The map $\involution: \LieA^+ \to \LieA^+$ defined by $\involution(w) = -\Ad_{w_0}(w)$ is called the \emph{opposition involution}.

	For all $g \in G$, let
\be  \label{eqn:FurstenbergBoundaryNotation}
		 g^+ = gP \in \Fboundary \quad \text{ and }\quad 
		 g^- = gw_0P \in \Fboundary.\ee 
	
	Fix a left $G$-invariant and right $K$-invariant Riemannian metric $d_G$ on $G$ and denote the corresponding inner product and norm on $\frak g$ by $\langle \cdot, \cdot \rangle$ and $\|\cdot\|$ respectively. Using the inner product on $\LieA$, we identify $\LieA$ with $\R^{\rank}$ and equip it with the Lebesgue measure which induces a Haar measure on $A$. 
	For $\varepsilon >0$ and a subset $S\subset G$, we set $S_\e=S\cap G_\e$ where $G_\varepsilon= \{g \in G: d_G(e,g) < \varepsilon\}$. For all $w\in \LieA^+$, we have
	\begin{equation}
		\label{eqn:Ad}
		\begin{aligned}[b]
			\| \Ad_{a_{-w}}x \| \le \|x\| e^{-\min_{\alpha\in\Phi^+}\alpha (w)} \quad & \text{for all } x \in \LieN;
			\\
			\| \Ad_{a_{w}}x \| \le \|x\| e^{-\min_{\alpha\in\Phi^+}\alpha (w)} \quad & \text{for all } x \in \check{\LieN}.
		\end{aligned}
	\end{equation}
	
	\subsection*{$\check{N}AMN$-coordinates} 
	The product map $\check{N} \times A \times M \times N \to G$ is a diffeomorphism onto a Zariski open neighborhood of $e$. The same is true if we permute $\check{N},A,M,N $. This fact is used to prove the next two lemmas.
	
	\begin{lemma}
		\label{lem:transversality}
		\begin{enumerate}
			\item \label{itm:transversality1}
			For all sufficiently small $\varepsilon>0$, if $0<\varepsilon_1,\varepsilon_2 <\varepsilon$, $h \in \check{N}_{\varepsilon_1}$ and $n \in N_{\varepsilon_2}$, then\footnote{\label{bigO}We write $O(\varepsilon)$ for a function which is in absolute value at most $C\varepsilon$ for some constant $C>0$ independent of $\e$.} 
			$$hn = n_1h_1a_1m_1 \in N_{O(\varepsilon_2)}\check{N}_{O(\varepsilon_1)}A_{O(\varepsilon)}M_{O(\varepsilon)}.$$
			
			\item \label{itm:transversality2}
			Fix a bounded subset $\check{U} \subset \check{N}$. For sufficiently small $\varepsilon>0$ we have for all $h_0 \in \check{U}$,
			$$h_0N_{O(\varepsilon)} \subset M_{O(\varepsilon)}A_{O(\varepsilon)}N_{O(\varepsilon)}h_0\check{N}_{O(\varepsilon)}.$$
		\end{enumerate}
		(\ref{itm:transversality1}) and (\ref{itm:transversality2}) still hold if the roles of $N$ and $\check{N}$ are swapped.
	\end{lemma}
	
	\begin{proof}  Let $\varepsilon>0$ be sufficiently small so that the image of the diffeomorphism given by the product map $N \times \check{N}\times A \times M \to G$ contains $\check{N}_\varepsilon N_\varepsilon$. Then by the implicit function theorem, there are smooth functions $x,y,a,m$ defined on $\check{N}_\varepsilon \times N_\varepsilon$ such that
		$$hn=x(h,n)y(h,n)a(h,n)m(h,n) \in N_{O(\varepsilon)}\check{N}_{O(\varepsilon)}A_{O(\varepsilon)}M_{O(\varepsilon)}.$$
		Note that $y(e,n)=e$ and if $h \in \check{N}_{\varepsilon_1}$, then $y(h,n)\in y(e,n)N_{O(\varepsilon_1)} = \check{N}_{O(\varepsilon_1)}$. Similarly, if $n \in N_{\varepsilon_2}$, then $x(h,n) \in N_{O(\varepsilon_2)}$. This proves (\ref{itm:transversality1}).
		
		For (\ref{itm:transversality2}), since $\check{U}$ is bounded, for $\varepsilon>0$ sufficiently small the image of the product map $M \times A \times N \times \check{N} \to G$ contains $\check{U}N_{O(\varepsilon)}$. A similar application of the implicit function theorem finishes the proof.
	\end{proof}
	
	\begin{lemma}
		\label{lem:wavefront}
		Fix bounded subsets $\check{U}\subset \check{N}$ and $U\subset N$. For all sufficiently small $\varepsilon >0$, the following holds. 
		\begin{enumerate}
			\item 
			\label{itm:wavefront1}
			If $g_1\in G_\varepsilon$ and $g = manh \in MAN\check{U}$, then 
			\[gg_1 \in mM_{O(\varepsilon)} aA_{O(\varepsilon)}nN_{O(\varepsilon)}h\check{N}_{O(\varepsilon)}.\] 
			
			\item 
			\label{itm:wavefront2}
			If $g_2\in G_\varepsilon$ and $g = ahmn \in A\check{N}MU$, then 
			\[gg_2 \in aA_{O(\varepsilon)}h\check{N}_{O(\varepsilon)} mM_{O(\varepsilon)}nN_{O(\varepsilon)}.\]
			
			\item 
			\label{itm:wavefront3}
			If $g_1,g_2\in G_\varepsilon$ and $g = hamn \in \check{U}AMU$, then 
			$$g_1gg_2 \in \check{N}_{O(\varepsilon)}  hA_{O(\varepsilon)}aM_{O(\varepsilon)}mnN_{O(\varepsilon)}.$$
		\end{enumerate}
	\end{lemma}
	
	\begin{proof}
 Let $\varepsilon>0$, $g_1 \in G_\varepsilon$ and $g=manh\in MAN\check{U}$. For $\varepsilon$ sufficiently small, we may write $g_1 = m_1a_1n_1h_1 \in M_{O(\varepsilon)}A_{O(\varepsilon)}N_{O(\varepsilon)}\check{N}_{O(\varepsilon)}$ since the product map $M \times A \times N \times \check{N} \to G$ is a diffeomorphism onto an open neighborhood of $e$. Let $m'=mm_1\in mM_{O(\varepsilon)}$, $a'=aa_1 \in aA_{O(\varepsilon)}$, $n'=(m_1a_1)^{-1}n(m_1a_1) \in nN_{O(\varepsilon)}$ and $h'=(m_1a_1)^{-1}h(m_1a_1) \in h\check{N}_{O(\varepsilon)}$ so that $gg_1 = m'a'n'h'n_1h_1$. Note that $h'n_1 \in \check{U}\check{N}_{O(\varepsilon)}N_{O(\varepsilon)}$ and $\check{U}\check{N}_{O(\varepsilon)} \subset \check{N}$ is bounded. Then by \cref{lem:transversality}(\ref{itm:transversality2}), if $\varepsilon$ is sufficiently small, then we can write $h'n_1 = m_2a_2n_2h''$ where $m_2a_2n_2 \in M_{O(\varepsilon)}A_{O(\varepsilon)}N_{O(\varepsilon)}$ and $h'' \in h\check{N}_{O(\varepsilon)}$. Then
		\begin{align*}
			gg_1 = m'a'n'h'n_1h_1 = m'a'n'm_2a_2n_2h''h_1= 
			m''a''n''n_2h''h_1
		\end{align*}
		where $m''=m'm_2 \in mM_{O(\varepsilon)}$, $a''=a'a_2\in aA_{O(\varepsilon)}$ and $n'' = (m_2a_2)^{-1}n'(m_2a_2) \in nN_{O(\varepsilon)}$. This completes the proof of (\ref{itm:wavefront1}).
		
		The proof of (\ref{itm:wavefront2}) is similar and (\ref{itm:wavefront3}) can be deduced from (\ref{itm:wavefront1}) and (\ref{itm:wavefront2}) in a similar manner.   
	\end{proof}
	
	Henceforth, let $\Gamma < G$ be a Zariski dense discrete subgroup.

	\subsection*{Limit set, limit cone and holonomy group.}
	\label{subsec:LimitSetAndLimitCone}
	Let $m_\Fboundary$ denote the unique $K$-invariant probability measure on $\Fboundary$.
	The \emph{limit set} $\limitset \subset \Fboundary$ of $\Gamma$ is defined by
	\begin{equation*}
		\limitset = \{\xi \in \Fboundary : \exists \{\gamma_n\}_{n \in \N} \subset \Gamma \text{ such that } (\gamma_n)_*m_\Fboundary \xrightarrow{n \to \infty} D_\xi\}
	\end{equation*}
	where $D_\xi$ denotes the Dirac measure at $\xi$. It is the unique $\Gamma$-minimal subset of $\Fboundary$ \cite{Ben97}.

	Any $g\in G$ can be written as the commuting product $g=g_hg_e g_u$ where $g_h$ is hyperbolic, $g_e$ is elliptic and $g_u$ is unipotent. 
	The hyperbolic component $g_h$ is conjugate to a unique element $\exp \lambda(g) \in A^+$ and $\lambda(g)$ is called 
	the Jordan projection of $g$.
	When $\lambda(g)\in \inte \fa^+$, $g\in G$ is called {\it{loxodromic}} in which case $g_u$ is necessarily trivial and $g_e$ is conjugate to
	an element $m(g)\in M$ which is unique up to conjugation in $M$. We call its conjugacy class
 $[m(g)]\in [M]$ the {\it holonomy} of $g$.

	The \emph{limit cone} $\limitcone = \limitcone_\Gamma $ of $\Gamma$ is the smallest closed cone containing the  Jordan projection $\lambda(\Gamma)$. It is convex and with non-empty interior {\cite{Ben97}} which we denote by $\interior\limitcone$. We note that $\lambda(g^{-1}) = \involution(\lambda(g))$ for all $g \in G$ and hence $\limitcone = \involution(\limitcone)$.

	The \emph{holonomy group} of $\Gamma$ is the closed subgroup $M_\Gamma < M$ generated by all of the holonomies in $\Gamma$.
	By \cite[Corollary 1.10]{GR07}, $M_\Gamma$ is a normal subgroup of $M$ of finite index. In particular if $M$ is connected, then $M_\Gamma = M$.  If $G$ is of rank one, we always have  $M = M_\Gamma$ for any Zariski dense $\Ga$,
	since either $M$ is connected or $G=\SL_2\R$ and $M_\Gamma = M = \{\pm\big(\begin{smallmatrix}1 & 0\\0 & 1 \end{smallmatrix}\big)\}$ \cite[Lemma 2]{CG93}. In general, there are examples of Zariski dense subgroups with $M_\Ga\ne M$ (e.g., Hitchin representations \cite[Theorem 1.5]{Lab06}).

	For $g \in G$, let  $\mu(g)$ denote the \emph{Cartan projection} of $g$, that is, $\mu(g) \in \LieA^+$ is the unique element in $\LieA^+$ such that 
	$$g \in K\exp(\mu(g))K.$$ We note that $\mu(g^{-1}) = \involution(\mu(g))$ for all $g \in G$. 	
	\subsection*{Conformal measures.} 
	The {Iwasawa cocycle} $\sigma: G \times \Fboundary \to \LieA$ is the map which assigns to each $(g,kM) \in G \times \Fboundary$ the unique element $\sigma(g,kM) \in \LieA$ such that $gk \in Ka_{\sigma(g,\xi)}N$. The {$\LieA$-valued Busemann function} $\beta: \Fboundary \times G \times G \to \LieA$ is defined by
	\begin{equation*}
		\beta_\xi(g_1,g_2) = \sigma(g_1^{-1}, \xi) - \sigma(g_2^{-1}, \xi)
	\end{equation*}
	for all $g_1,g_2 \in G$ and $\xi \in \Fboundary$. 
	
	\begin{definition}[Growth indicator]\label{Dg}
		\label{def:GrowthIndicatorFunction}
		The \it{growth indicator} $\growthindicator : \LieA^+ \to \R \cup \{-\infty\}$ of $\Gamma$ is defined by
		\begin{equation*}
			\growthindicator(w) = \|w\| \inf_{\text{open cones }\mathcal{C}\ni w} \tau_\mathcal{C} \qquad \text{for all non-zero $w \in \LieA^+$}
		\end{equation*}
		where $\tau_\mathcal{C}$ is the abscissa of convergence of the series $t \mapsto \sum_{\gamma \in \Gamma, \mu(\gamma) \in \mathcal{C}} e^{-t\|\mu(\gamma)\|}$. We set $\psi_\Ga(0)=0$.
	\end{definition}
	
	It is concave, upper semicontinuous, and satisfies $\growthindicator|_{\interior\L}>0$ \cite[Theorem 4.2.2]{Qui02a}. By \cite[Lemma 3.1.1]{Qui02a}, when $\psi_\Ga(w)>0$ (i.e., $w \in \interior\limitcone$), we have
	$$\psi_\Ga(w) = \|w\|\inf_{\text{open cones }\mathcal{C} \ni w}\limsup_{T\to\infty}\tfrac{1}{T} \log
	\#\{\gamma \in \Gamma: \mu(\gamma) \in \mathcal{C}, \; \; \|\mu(\gamma)\| \le T\}.$$

	Given a closed subgroup $\Delta < G$, a Borel probability measure $\nu$ on $\Fboundary$ is called a \emph{$\Delta$-conformal measure} if there exists $\psi \in \LieA^*$ such that for any $\gamma \in \Delta$ and $\xi \in \Fboundary$,
	$$\frac{d\gamma_*\nu}{d\nu}(\xi) = e^{\psi(\beta_\xi(e,\gamma))}$$
	where $\gamma_*\nu(Q) = \nu(\gamma^{-1}Q)$ for any Borel subset $Q \subset \Fboundary$. In that case, we call $\nu$ a \emph{$(\Delta,\psi)$-conformal measure}.

	\subsection*{Generalized BMS measures.}
	\label{subsec:GeometricMeasures}
	We recall the definitions of generalized Bowen-Margulis-Sullivan measures using the Hopf parametrization of $G/M$. There is a unique open $G$-orbit in $\Fboundary \times \Fboundary$ given by
	\begin{equation}
		\label{eqn:FurstenbergBoundary2}
		\Fboundary^{(2)} = G.(e^+,e^-) \subset \Fboundary\times\Fboundary.
	\end{equation} 
	If $(x,y) \in \Fboundary^{(2)}$, then we say that $x$ and $y$ are in \emph{general position}. The \emph{Hopf parametrization} is a diffeomorphism $G/M \to \Fboundary^{(2)} \times \LieA$ defined by
	\begin{equation*}
		gM \mapsto (g^+, g^-, \beta_{g^+}(e, g)) \quad\text{for all $g\in G$}.
	\end{equation*}

	For a pair $(\nu_{\psi_1}, \nu_{\psi_2})$ of $(\Ga, \psi_1)$- and $(\Ga, \psi_2)$-conformal measures on $\F$, the generalized Bowen-Margulis-Sullivan measure $\mathsf m=\mathsf{m}_{\nu_{\psi_1},\nu_{\psi_2}}$ is defined on $G/M \cong \Fboundary^{(2)} \times \LieA$ by 
	\begin{equation}
		\label{eqn:BMSMeasureDefinition}
		{\mathsf{m}}_{\nu_{\psi_1},\nu_{\psi_2}}(gM) = e^{\psi_1\left(\beta_{g^+}(e,g)\right) + \psi_2\left(\beta_{g^-}(e,g)\right)} \, d\nu_{\psi_1}(g^+) \, d\nu_{\psi_2}(g^-) \, dw.
	\end{equation}
	The measure $\mathsf m$ is left $\Gamma$-invariant and right $A$-quasi-invariant. It is $A$-invariant if and only if $\psi_2=\psi_1\circ \i$.
	The measure $\mathsf m$ descends to a measure on $\Gamma \backslash G/M$ and by lifting it using the Haar probability measure on $M$, we also obtain a measure on $\Gamma \backslash G$. Abusing notation, we also denote it by $\m =\mathsf{m}_{\nu_{\psi_1},\nu_{\psi_2}}$ as well.

	\section{Local mixing for Anosov subgroups}
	\label{sec:AnosovSubgroups}
	There are several equivalent characterizations of Anosov subgroups. We use the following definition \cite{KLP17}:
	a finitely generated subgroup $\Ga<G$ is called an Anosov subgroup (with respect to $P$) if there exists $C>0$ such that for all $\ga\in \Ga$,
	$$ \alpha (\mu(\ga)) \ge  C|\ga| -C^{-1}$$
	for all simple root $\alpha$ of $(\frak g, \fa^+)$ where $|\gamma|$ denotes the word length of $\ga$ with respect to a fixed finite set of generators of $\Ga$.

	In this section, let $\Gamma<G$ be a Zariski dense Anosov subgroup. The next theorem summarizes some facts about conformal measures for Anosov subgroups. We say that a linear form $\psi:\LieA\to\R$ is tangent to $\growthindicator$ at $\v\in \fa^+$ if
	$$\growthindicator \le \psi \;\; \text{and}\;\; \growthindicator(\v)=\psi(\v).$$

	\begin{thm}
		\label{thm:Anosov}
		Let $\v\in \inte\L$ be a unit vector. Then 
		there exists a unique linear form $\psi_\v$ tangent to $\growthindicator$ at $\mathsf{v}$. There also exists a unique $(\Ga, \psi_\v)$-conformal measure $\nu_\v$ on $\F$. Moreover, $\nu_\v$ is supported on $\Lambda$. 
	\end{thm}
	
	The three claims in \cref{thm:Anosov} can be found in  \cite[Proposition 4.11]{PS17},
	\cite[Theorem 1.3]{LO22} and \cite[Theorem 7.9]{ELO20}, respectively.
	In the following, fix a unit vector $$\v\in \inte\L.$$
	Noting that $m_\F$ is a $(\Ga, 2\rho)$-conformal measure where  $$2\rho=\sum_{\alpha\in \Phi^+} \alpha $$ is the sum of all positive roots with multiplicity, we set 
	$${\mathsf m}_\v= {\mathsf m}_{\nu_{\v}, \nu_{\i(\v)}},\;\;
	{\mathsf m}^{\op{BR}}_\v= {\mathsf m}_{m_{\F}, \nu_{\i(\v)}} \;\;\text{ and } \;\; {\mathsf m}^{\op{BR}_*}
	_\v= {\mathsf m}_{\nu_{\v}, m_{\F}} .$$
	The measures $\BR$ and $\BRstar$  are respectively $\check N$ and $ N$-invariant and called Burger-Roblin measures.

	The Anosov property of $\Ga$ implies that any two distinct points in $\La$ are in general position (\cite{GW12}, \cite{KLP17}). 	Let $\limitset^{(2)} = (\limitset \times \limitset) \cap \Fboundary^{(2)} = \{(x, y) \in \limitset \times \limitset: x \neq y\}$.
	The map $\pi_\v: \limitset^{(2)} \times \LieA \to \limitset^{(2)} \times \R$ defined by 
	\begin{equation*}
		\pi_\v(x, y, w) = (x, y, \psi_\v(w)) \quad \text{for all $(x, y, w) \in \limitset^{(2)} \times \LieA$}
	\end{equation*}
	is a vector bundle with typical fiber $\ker\psi_\v$.
	Note that $\Gamma$ acts on $\limitset^{(2)} \times \LieA$ and on $\limitset^{(2)} \times \R$ on the left respectively by
	\begin{equation*}
		\gamma \cdot (x, y, w) = (\gamma x, \gamma y, w + \beta_x(\gamma^{-1}, e)),\;\; 	\gamma \cdot (x, y, t) = (\gamma x, \gamma y, t + \psi_\v(\beta_x(\gamma^{-1}, e)))
	\end{equation*}
	for all $\gamma \in \Gamma$,  $(x, y) \in \limitset^{(2)}$, $w\in \fa$ and $t\in \br$. 
	\begin{theorem}[{\cite[Proposition A.1]{Car21}}, see also {\cite[Theorem 4.15]{CS23}}]
		\label{thm:GammaActionIsNice}
		The left $\Gamma$-action on $\limitset^{(2)} \times \R$ is properly discontinuous and cocompact. 
	\end{theorem}
	
	By \cref{thm:GammaActionIsNice}, the space
	$${\mathcal{X}_\v} = \Gamma \backslash (\limitset^{(2)} \times \R)$$
	is a compact Hausdorff topological space. 
	Define the locally finite Borel measure $\tilde{m}_{{\mathcal{X}_\v}}$ on $\limitset^{(2)} \times \R$ by
	$$d\tilde{m}_{{\mathcal{X}_\v}}(\xi, \eta, t) = e^{\psi_\v(\beta_{\xi}(e, g))+\psi_\v(\involution(\beta_{\eta}(e, g)))} \, d\nu_\v(\xi) \, d\nu_{\involution(\v)}(\eta) \, dt$$
	where $g \in G$ is any element with $g^+ = \xi$ and $g^- = \eta$ and $dt$ denotes the Lebesgue measure on $\R$ \cite[Definition 3.8]{LO20b}. Note that $\tilde{m}_{{\mathcal{X}_\v}}$ is left $\Gamma$-invariant, so $\tilde{m}_{{\mathcal{X}_\v}}$ descends to a finite measure $m_{{\mathcal{X}_\v}}$ on ${\mathcal{X}_\v}$. 
	
	Set $$\Omega= \Gamma \backslash \limitset^{(2)}\times\LieA\subset \Ga\ba G/M$$ which is the support of $\BMS$. The map $\pi_\v$ is $\Gamma$-equivariant and descends to a map $\pi_\v:\Omega \to {\mathcal{X}_\v}$ which is in fact a trivial $\ker\psi_\v$-vector bundle (\cite[Proposition 3.5]{Sam15}, \cite[Corollary 4.9]{LO20b}). Hence $\Omega$ is homeomorphic to ${\mathcal{X}_\v} \times \ker\psi_\v$ and
	\begin{equation}
		\label{eqn:BMSMeasureAndXMeasure}
		d\BMS\bigr|_\Omega = dm_{{\mathcal{X}_\v}} \, du
	\end{equation}
	where $du$ denotes the appropriately normalized Lebesgue measure on $\ker\psi_\v$.
	
	\subsection*{Local mixing.}
	We recall the local mixing theorem for the Haar measure on $\Gamma \backslash G$ which will be used in \cref{sec:Jordan}. Let $dx$ denote the right $G$-invariant measure on $\Gamma \backslash G$ induced by the Haar measure on $G$. Given an inner product $\langle\cdot,\cdot\rangle_*$ on $\LieA$, let $I: \ker\psi_\v \to \R$ be defined by 
	\begin{equation}
		\label{eqn:IDefinition}
		I(u) = \langle u, u\rangle_* - \frac{\langle u, \mathsf{v} \rangle_*^2}{\langle \mathsf{v}, \mathsf{v}\rangle_*}\quad  \text{ for all }u \in \ker\psi_\v.
	\end{equation} 
	
	\begin{theorem}[{\cite[Theorem 1.3]{CS23}}, {\cite[Theorem 3.4]{ELO22b}}]
		\label{thm:DecayofMatrixCoefficients}
		There exist $\kappa_{\mathsf{v}} >0$ and an inner product $\langle \cdot, \cdot \rangle_*$ on $\LieA$ such that for any $u \in \ker\psi_\v$ and $\phi_1, \phi_2 \in C_{\mathrm{c}}(\Gamma \backslash G)$, we have
		\begin{multline*}
			\lim_{t \to +\infty} t^{\frac{\rank - 1}{2}}e^{(2\rho - \psi_\v)(t\mathsf{v} + \sqrt{t}u)} \int_{\Gamma \backslash G} 	\phi_1(xa_{t\mathsf{v} + \sqrt{t}u}) \phi_2(x) \, dx \\
			=\frac{\kappa_{\mathsf{v}}e^{-I(u)}}{|m_{\mathcal{X}_\v}|}  \sum_{Z} 	\BR\bigr|_{Z\check{N}}(\phi_1)\cdot 	\BRstar\bigr|_{ZN}(\phi_2)
		\end{multline*}
		where the sum is taken over all $A$-ergodic components $Z$ of $\BMS$.
		
		Moreover, there exist $\eta_{\mathsf{v}}>0$ and $s_{\mathsf{v}}>0$ such that for all $\phi_1, \phi_2 \in C_{\mathrm{c}}(\Gamma \backslash G)$, there exists $D_\v>0$ depending continuously on $\phi_1$ and $\phi_2$ such that for all $(t,u) \in (s_{\mathsf{v}},\infty) \times \ker\psi_\v$ such that $t\mathsf{v}+\sqrt{t}u \in \LieA^+$, we have
		$$\left|t^{\frac{\rank - 1}{2}}e^{(2\rho - \psi_\v)(t\mathsf{v} + \sqrt{t}u)} \int_{\Gamma \backslash G} \phi_1(xa_{t\mathsf{v} + \sqrt{t}u}) \phi_2(x) \, dx\right| \le D_\mathsf{v}  e^{-\eta_{\mathsf{v}} I(u)}.
		$$
	\end{theorem}
	
	\section{Joint equidistribution of cylinders and holonomies}
	\label{sec:Jordan}
	Let $\Gamma < G$ be a Zariski dense Anosov subgroup. We give a definition of essential tube that is slightly more general than the one given in the introduction.
	\begin{definition}
		An essential tube $\T$ for $\Gamma$ is given by
		$$\T=\T(\v, \mathsf K)=(\R\v + \mathsf{K}) \cap\LieA^+ $$
		where $\v\in \inte\L$ is a unit vector and 
		and $\mathsf K\subset  \ker\psi_\v $  is a compact subset with non-empty interior and Lebesgue null boundary.
		We call $\v$ the direction and $\mathsf K$ the cross-section of $\T$.
	\end{definition}
	For this entire section,
	fix $$ \text{an essential tube }\T=\T(\v,\mathsf{K})\quad\text{and}\quad b\in C(\mathsf{K}).$$ 
	For $T>0$, set 
	$$\T_{T,b}=\T_T(\mathsf{v},\mathsf{K},b)=\{t\mathsf{v}+u \in \LieA^+ : t \le T+b(u), \, u \in \mathsf{K}\}.$$ 
	A subset of $\LieA^+$ of this form will be called a \emph{truncated tube}.
	Note that for $b\equiv 0$, we have
	$$\T_{T,0}=\{t\mathsf{v}+u \in \LieA^+ : t \le T, \, u \in \mathsf{K}\}.$$

	\begin{figure}[H]
		\centering
		\scalebox{0.6}{\begin{tikzpicture}
				\draw[thick,-latex] (0,0) -- (1,5);
				\draw[thick,-latex] (0,0) -- (5,1);
				\draw (1.5,4) node{\large $\LieA^+$};
				\draw[thick, -latex] (0,0) -- (1,2);
				\draw (1.1,2.2) node{\large $\v$};
				\draw[dashed,latex-latex] (-2,2) -- (2,-2); 
				\draw (-1,2) node{\large $\ker\psi_\v$};
				\draw[thin] (1,0.2) -- (3,4.2);
				\draw[thin] (3,0.6) -- (5,4.6);
				\draw (4,4.5) node{\large $\T(\v,\mathsf{K})$};
				\draw[dashed] (1,0.2) -- (0.6,-0.6);
				\draw[dashed] (3,0.6) -- (1.8,-1.8);
				\draw[thin] (0.6,-0.6) -- (1.8,-1.8);
				\draw (1,-1.4) node{\large $\mathsf{K}$};
				\draw[thin] (2.6,3.4) -- (3.8,2.2);
				\draw[teal, thin, fill = teal, fill opacity=0.2] (2.6, 3.4) -- (3.8, 2.2) -- (3.8,2.2) -- (3,0.6) -- (1,0.2) -- (2.6,3.4);
				\draw (3,2) node{\large $\T_{T,0}$};
				
				\begin{scope}[shift = {(8,0)}]
					\draw[thick,-latex] (0,0) -- (1,5);
					\draw[thick,-latex] (0,0) -- (5,1);
					\draw (1.5,4) node{\large $\LieA^+$};
					\draw[thick, -latex] (0,0) -- (1,2);
					\draw (1.1,2.2) node{\large $\v$};
					\draw[dashed,latex-latex] (-2,2) -- (2,-2); 
					\draw (-1,2) node{\large $\ker\psi_\v$};
					\draw[thin] (1,0.2) -- (3,4.2);
					\draw[thin] (3,0.6) -- (5,4.6);
					\draw (4,4.5) node{\large $\T(\v,\mathsf{K})$};
					\draw[dashed] (1,0.2) -- (0.6,-0.6);
					\draw[dashed] (3,0.6) -- (1.8,-1.8);
					\draw[thin] (0.6,-0.6) -- (1.8,-1.8);
					\draw (1,-1.4) node{\large $\mathsf{K}$};
					\draw (0.6, -0.6) .. controls (1, -1) and (2.5, 0) .. (1.8, -1.8);
					\draw (1.6,-0.4) node{\large $b$};
					\draw (2.6, 3.4) .. controls (3, 3) and (4.5, 4) .. (3.8, 2.2);
					\draw[dashed] (2.6,3.4) -- (3.8,2.2);
					\draw[teal, thin, fill = teal, fill opacity=0.2] (2.6, 3.4) .. controls (3, 3) and (4.5, 4) .. (3.8, 2.2) -- (3.8,2.2) -- (3,0.6) -- (1,0.2) -- (2.6,3.4);
					\draw (3,2) node{\large $\T_{T,b}$};
				\end{scope}
		\end{tikzpicture}}
		\caption{} \label{fig:Tube2}
	\end{figure} 
	
	Let $$\delta_\v = \growthindicator(\v).
	$$ 
	The main goal of this section is to prove the following equidistribution of Jordan projections in tubes and their holonomies:
	\begin{theorem}\label{Counting}
		For any $\varphi \in \mathrm{Cl}(M)$, we have as $T \to \infty$,
		\begin{equation}
			\label{eqn:Equidistribution}
			\sum_{
				[\gamma] \in [\Gamma],\, \lambda(\gamma) \in \T_{T,b}} \varphi(m(\gamma)) 
			\sim \frac{\kappa_{\mathsf{v}}}{\delta_\v }  \int_{\mathsf{K}} e^{\delta_\v b(u)} \, du \cdot \int_{M_\Gamma}\varphi \, dm \cdot \frac{e^{\delta_\v T}}{T^{(\rank+1)/2}} 
		\end{equation}
		where $\kappa_\mathsf{v}$ is as in \cref{thm:DecayofMatrixCoefficients}.
		In particular, 
		we have 
		$$\#\{[\gamma]\in[\Gamma]:\lambda(\gamma) \in \T_{T,b}\} \sim \frac{\kappa_{\mathsf{v}}}{\delta_\v [M:M_\Gamma]}\int_{\mathsf{K}} e^{\delta_\v b(u)} \, du \cdot\frac{e^{\delta_\v T}} {T^{(\rank+1)/2}} \quad \text{as } T \to \infty.$$
	\end{theorem} 
	
	When $b\equiv 0$ and $\varphi \equiv 1$, we get the following:
	\begin{cor}
		We have
		$$\#\{[\gamma]\in[\Gamma]:\lambda(\gamma) \in [0,T]\v + \mathsf{K}\} \sim \frac{\kappa_{\mathsf{v}}\vol_{\ker\psi_\v}(\mathsf{K})}{\delta_\v [M:M_\Gamma]} \cdot\frac{e^{\delta_\v T}} {T^{(\rank+1)/2}} \quad \text{as } T \to \infty$$
		where $\vol_{\ker\psi_\v}(\mathsf{K}) = \int_\mathsf{K} \, du$.
	\end{cor}
	
	In the rank one case, the prime geodesic theorem is deduced from equidistribution of closed geodesics in the unit tangent bundle (\cite{Mar04}, \cite{Rob03}, \cite{MMO14}). In the same spirit, we first prove joint equidistribution theorems (Theorems \ref{thm:MuTJointEquidistribution}, \ref{JointEquidistribution}) of closed $A$-orbits $C(\gamma)$ and their holonomies $m(\gamma)$ with $\lambda(\gamma)$ in tubes.
	
	By the Anosov hypothesis on $\Ga$, every non-trivial element of $\Ga$ is loxodromic (\cite[Proposition 3.4]{Lab06}, \cite[Corollary 3.2]{GW12}).
	Let $\primGamma$ denote the set of primitive elements in $\Gamma$, that is,
	it consists of all elements $\gamma\in \Ga$ such that $\ga\ne \ga_0^k$ for any $\ga_0\in \Ga$ and $k\ge 2$. For each conjugacy class $[\gamma]\in[\primGamma]$ with $\gamma \in g (\inte A^+) Mg^{-1}$ for some $g \in G$, consider the closed $A$-orbit 
	$$ C({\gamma})=\Gamma gAM \subset \Omega  $$ which is homeomorphic to a cylinder $\S^1 \times \R^{r -1}$ \cite[Lemma 4.14]{CF23}. 
	For $f\in C_\mathrm{c}(\Omega)$,
	the integral $\int_{C(\ga)} f $  is computed with respect to the measure on $C(\gamma)$ induced by the Lebesgue measure on $\LieA$.
	Set \be\label{cv}  c(\v, b) =\frac{\kappa_{\mathsf{v}}}{\delta_\v|m_{\mathcal{X}_\v}|} \int_{\mathsf{K}} e^{\delta_\v b(u)} \, du \ee with  $\kappa_\v$ and $m_{\mathcal{X}_\v}$ as in \cref{thm:DecayofMatrixCoefficients} and  \eqref{eqn:BMSMeasureAndXMeasure} respectively.
	\begin{theorem}[Joint Equidistribution I]
		\label{thm:MuTJointEquidistribution}
		For any $f \in C_{\mathrm{c}}(\Omega)$ and $\varphi \in \mathrm{Cl}(M)$, we have as $T \to \infty$,
		$$\sum_{[\gamma] \in [\primGamma],\, \lambda(\gamma) \in \T_{T,b}} \int_{C_\ga} f \; \cdot \varphi(m(\gamma)) \sim c(\v, b)\cdot \BMS(f)\cdot\int_{M_\Ga}\varphi \, dm \cdot \frac{e^{\delta_\v T}}{T^{(\rank-1)/2}}. $$
	\end{theorem}	
	
	\begin{thm}[Joint equidistribution II]
		\label{JointEquidistribution}
		For any $f \in C_{\mathrm{c}}(\Omega)$ and $\varphi \in \mathrm{Cl}(M)$, we have as $T \to \infty$,
		\begin{equation*}
			\label{eqn:JointEquidistribution}
			\sum_{[\gamma] \in [\primGamma] ,\, \lambda(\gamma) \in \T_{T,b}}\frac{1}{\psi_\v(\lambda(\gamma))}
			{\int_{C(\ga)} f} \cdot \varphi(m(\gamma)) \sim  c(\v, b) \cdot  \BMS(f)\cdot  \int_{M_\Gamma}\varphi \, dm  \cdot \frac{e^{\delta_\v T}}{ T^{(\rank+1)/2}}.
		\end{equation*}
	\end{thm}

	We will prove these two theorems in the next section.
	In the rest of this section,
	we will explain how to  deduce \cref{Counting} from \cref{JointEquidistribution} using the following structure of $\Omega$:
	\be\label{ommm} \Omega =\supp\BMS \cong {\mathcal{X}_\v} \times \ker\psi_\v \quad \text{ and }\quad  d\BMS\bigr|_\Omega = dm_{{\mathcal{X}_\v}} \, du \ee from \eqref{eqn:BMSMeasureAndXMeasure}.
	We emphasize that this fact we are relying on is a special feature of an Anosov subgroup which is not available for a general discrete subgroup such as a higher rank lattice.
	
	\subsection*{Proof of \cref{Counting} assuming Theorem \ref{JointEquidistribution}}  
	Choose  $f_1 \in C_{\mathrm{c}}(\ker\psi_\v)$ with $\int f_1(u)\, du=1$. In view of \eqref{ommm},
 the function
	$$f= \mathbbm{1}_{{\mathcal{X}_\v}} \otimes f_1$$
	can be considered as a function in  $C_{\mathrm{c}}(\Omega)$ and 
	$$\BMS(f) = |m_{\mathcal{X}_\v}|\int f_1(u)\, du = |m_{\mathcal{X}_\v}|.$$
	Therefore for every $[\ga] \in [\primGamma]$, 
	$$\int_{C(\ga)} f = \psi_\v(\lambda(\ga))\int f_1(u)\, du = \psi_\v(\lambda(\ga)).$$
	By applying \cref{JointEquidistribution} to this function $f$,  we obtain \eqref{eqn:Equidistribution} with $[\Gamma]$ replaced with $[\primGamma]$:
	\be\label{haha} \sum_{
		[\gamma] \in [\primGamma],\, \lambda(\gamma) \in \T_{T,b}} \varphi(m(\gamma)) 
	\sim \frac{\kappa_{\mathsf{v}}}{\delta_\v }  \int_{\mathsf{K}} e^{\delta_\v b(u)} \, du \cdot \int_{M_\Gamma}\varphi \, dm \cdot \frac{e^{\delta_\v T}}{T^{(\rank+1)/2}}. \ee 
	
	We claim that this asymptotic remains true when $[\primGamma]$ is replaced with $[\Gamma]$. To see that, we first take $\varphi \equiv 1$ and $b \equiv 0$ and obtain
	\begin{equation}
		\label{Prim}
		\#\{[\gamma]\in[\primGamma]:\lambda(\gamma) \in \T_T\} \sim \frac{\kappa_{\mathsf{v}}}{\delta_\v [M:M_\Gamma]}\int_\mathsf{K}e^{\delta_\v b(u)}\,du \cdot\frac{e^{\delta_\v T}}{T^{(\rank+1)/2}}.
	\end{equation}
	Suppose that $\mathsf{K}$ is a convex set containing $0$. 		
	Observe that if $\lambda(\gamma_0^j) \in \T_{T}(\v,\mathsf{K},b)$, then $\lambda(\gamma_0) = \frac{1}{j}\lambda(\gamma_0^j) \in \T_{T/j}(\v,\frac{1}{j}\mathsf{K},b) \subset \T(\v,\mathsf{K},b)$ by the convexity of $\mathsf K$. Therefore
	\begin{multline*}
		\#\{[\gamma]\in[\Gamma]:\lambda(\gamma) \in \T_T\} - \#\{[\gamma]\in[\primGamma]:\lambda(\gamma) \in \T_T\} 
		\\
		\ll\footnote{For functions $f$ and $g$ of $T$, we write $f\ll g$ if $f$ is bounded above by a constant multiple of $g$.} T \#\{[\gamma]\in[\primGamma]:\lambda(\gamma) \in \T_{T/2}\} \ll e^{\delta_\v T/2}.
	\end{multline*}

    Identifying $\ker\psi_\v$ with $\R^{r-1}$, we next consider the case when $\mathsf{K}$ is a box $\prod_{i=1}^{r-1} [a_i,b_i]$. Set $B_j = [0,b_1] \times \cdots\times [0,a_j] \times \cdots \times [0,b_{r-1}]$ and note that for any $f\in C(\ker\psi_\v)$, we have
    \begin{equation}
    \label{box}
        \int_{\prod_{i=1}^{r-1} [a_i,b_i]}f = \int_{\prod_{i=1}^{r-1} [0,b_i]}f - \sum_{j=1}^{r-1}\int_{B_j}f + (r-2)\int_{\prod_{i=1}^{r-1} [0,a_i]}f.
    \end{equation}
    \cref{Counting} now follows for $\mathsf{K}=\prod_{i=1}^{r-1} [a_i,b_i]$ by applying \cref{Counting} to each box containing $0$ on the right hand side of \eqref{box}.
    
    For general $\mathsf{K}$, using the hypothesis that the boundary of $\mathsf{K}$ has zero Lebesgue measure, we can approximate $\mathsf{K}$ above and below by boxes and apply \cref{Counting} to the boxes.

	\subsection*{Proof of \cref{main0}}
	We now deduce \cref{main0} from \cref{Counting}. Fix $\varepsilon > 0$, $w \in \LieA$ and conjugation invariant Borel subset $\Theta \subset M$ with smooth boundary. Recall that $\T=\T(\v, \e, w)=\{u+w\in \fa^+: \|u-\br \v\| \le \e\}$. Let 
	$$B_T=\{v \in \T: \|v\| \le T\}.$$
	Then we want an asymptotic for
	$$\#\{[\gamma]\in[\Gamma]: \lambda(\gamma) \in B_T, \, m(\gamma) \in \Theta\}.$$
	Decomposing $w$ as an element in the direct sum $\LieA=\ker\psi_\v\oplus\R\v$, we may assume without loss of generality that $w \in \ker\psi_\v$. Note that the set $B_T$ 
	is not a truncated tube as defined at the beginning of \cref{sec:Jordan}. Let
	$$\mathsf{K} = \{u \in \ker\psi_\v: \|u-w-\R\v\| \le \varepsilon\}$$ 
	so that 
	$$\T=\T(\v,\mathsf{K}) = \{t\v+u \in \LieA^+: t\in \R, \,  u \in \mathsf{K}\}.$$ 
	For $u\in \mathsf{K}$, define $b(u)\in\R$ as the unique number such that $u + b(u)\v$ is orthogonal to $\v$ with respect to the inner product inducing the norm $\|\cdot\|$. We claim that for fixed $\varepsilon'>0$, when $T$ is sufficiently large, we have
	\begin{equation}
		\label{eqn:normtube}
		\T_{T,b-\varepsilon'} \subset B_T \subset \T_{T,b}.
	\end{equation}
	
	Suppose that $v \in \T_{T,b-\varepsilon'}$. Then $v=t\v+u$ for some $u \in \mathsf{K}$ and $t \le T+b(u)-\varepsilon'$. We have 
	\begin{align*}
		&\|t\v+u\|  = \|(t-b(u))\v + (u+b(u)\v)\|
		\\
		&= \sqrt{(t-b(u))^2+\|u+b(u)\v\|^2} \le \sqrt{(T-\varepsilon')^2 + \|u+b(u)\v\|^2}
		\le T
	\end{align*}
	where the last inequality holds when $T$ is sufficiently large since $u+b(u)\v$ is bounded. This shows that $v \in B_T$.
	
	Now suppose that $v \in B_T$. Then $v=t\v+u$ for some $u \in \mathsf{K}$ and $t \in \R$ with $\|t\v+u\| \le T$. We have
	$$t-b(u)\le\|t\v+u\| = \sqrt{(t-b(u))^2 + \|u+b(u)\v\|^2} \le T$$
	and hence $t \le T+ b(u)$ and $v \in \T_{T,b}$. This proves \eqref{eqn:normtube}.
	
	It follows from \eqref{eqn:normtube} that
	\begin{multline*}
		\#\{[\gamma]\in[\Gamma]:\lambda(\gamma) \in \T_{T,b}\} 
		\le \#\{[\gamma]\in[\Gamma]: \lambda(\gamma) \in B_T, \, m(\gamma) \in \Theta\}
		\\
		\le  \#\{[\gamma]\in[\Gamma]:\lambda(\gamma) \in \T_{T,b}\}.
	\end{multline*}
	Applying \cref{Counting} to $\T_{T,b-\varepsilon'}$ and $\T_{T,b}$ and then taking $\varepsilon' \to 0$, we conclude that 
	\begin{multline*}
		\lim_{T\to \infty} \frac{T^{(r-1)/2}}{e^{\delta_\v T}}\#\{[\gamma]\in[\Gamma]: \lambda(\gamma) \in B_T, \, m(\gamma) \in \Theta\} 
		= \kappa_\T  \cdot\vol_M(\Theta\cap M_\Gamma).
	\end{multline*}
	where \be\label{kt} \kappa_\T= \frac{\kappa_{\mathsf{v}}}{\delta_\v}\int_{\mathsf{K}}e^{\delta_\v b(u)} \, du .\ee
	
	\begin{remark}\label{ug}
		There is a unique maximal growth direction $\v_\Ga$ at which $\psi_\Ga$ attains its maximum on $\{w\in\LieA^+: \|w\| =1\}$. When $\Ga$ is Anosov, $\v_\Ga \in \interior\L$ \cite[Proposition 4.11]{PS17} and  $\ker\psi_{\v_\Ga}$ is orthogonal to $\v$. The function $b \in C(\mathsf{K})$ in the proof of \cref{main0} was defined by the condition that $u+b(u)\v$ is orthogonal to $\v$ for every $u \in \ker\psi_\v$. Hence for essential tubes $\T(\v,\mathsf K) $ of direction $\v=\v_\Ga$, we have $b\equiv 0$ and $\kappa_\T= \tfrac{\kappa_{\mathsf{v_\Ga}}}{\delta_{\v_\Ga}}\vol(\mathsf{K})$.
	\end{remark}
	\section{Proofs of joint equidistribution theorems}\label{sec:j}
	This section is devoted to proving Theorems \ref{thm:MuTJointEquidistribution} and \ref{JointEquidistribution}. We keep the notations for $\T=\T(\v, \mathsf K), b$ etc from Section  \ref{sec:Jordan}.
	To simplify notation, we write the proof for the case that $M=M_\Ga$ (e.g., $M$ is connected). The case when $M\ne M_\Ga$ is complicated by the fact that the BMS measure $\mathsf m_\v$ on $\Gamma \backslash G/M$ has more than one $A$-ergodic components, more precisely, the number of its ergodic components is equal to $[M:M_\Ga]$ \cite{LO20a}, but this can be handled in exactly the same way as in \cite[Section 5]{CF23}.

	\subsection*{Counting in $\check{N}AMN$-coordinates.}
	\label{subsec:CountinginNAMNCoordinates}
	Let $\nu$ and $\nu_{\i}$ denote the unique $(\Ga, \psi_{\v})$- and $(\Ga,\psi_{\i(\v)})$-conformal measures on $\cal F$ respectively.    Let $\vol_M$ denote the Haar probability measure on $M$.
	Fix bounded Borel sets 
	$$\check{\Xi} \subset\check{N},\;\; \Xi \subset N,\;\; \Theta \subset M$$ with non-empty relative interiors and null boundaries: $$\nu(\partial \check{\Xi} e^+)=\nu_{\involution}(\partial \Xi^{-1}e^-)=\vol_M(\partial\Theta)=0.$$
	For $T > 0$, consider the following subset of $\check{N}A^+MN$:  
	\begin{equation}
		\label{eqn:STDefinition}
		S_{T,b}=S_{T,b}(\check{\Xi},\Xi,\T,\Theta) = \check{\Xi}\exp(\T_{T,b}) 	\Theta \Xi.
	\end{equation}

	Define the measures $\tilde{\nu}$ on $\check{N}$ and $\tilde{\nu}_{\involution}$ on $N$ by
	\begin{align*}
		& d\tilde{\nu}(h) = e^{\psi_\v(\beta_{h^+}(e,h))}  \, d\nu(h^+) ;
		& d\tilde{\nu}_{\involution}(n) = e^{(\psi_\v\circ\involution)(\beta_{n^-}(e,n))}\, d\nu_{\involution}(n^-).
	\end{align*} 
	
	In this subsection, we prove an asymptotic for 
	$\# \Gamma \cap S_{T,b}$.
	\begin{proposition} \label{prop:STAsymptotic}
		We have, as  $T\to\infty$,
		\begin{equation}
			\label{eqn:STAsymptotic}
			\#(\Gamma \cap S_{T,b}) \sim c(\v, b) \tilde{\nu}(\check{\Xi})\tilde{\nu}_{\involution}(\Xi^{-1})\vol_M(\Theta)\frac{e^{\delta_\v T}}{T^{(\rank-1)/2}} .
		\end{equation}
	\end{proposition}

	For a sufficiently small $\e>0$, consider the following $\e$-approximations of $S_{T,b}$:
	\begin{align*}
		& S_{T,b,\varepsilon}^- = \bigcap_{g_1,g_2 \in G_\varepsilon}g_1S_{T,b}g_2; & S_{T,b,\varepsilon}^+ = \bigcup_{g_1,g_2 \in G_\varepsilon}g_1S_{T,b}g_2.
	\end{align*}
	
	In the next lemma, we state a property of the $\check{N}AMN$-coordinates that we will use to prove \cref{lem:STpmBounds} which bounds the sets $S_{T,b,\varepsilon}^\pm$ by product subsets of $\check{N}A^+MN$ that approximate $S_{T,b}$.
	
	\begin{lemma}
		\label{lem:STpmBounds} 
		For all sufficiently small $\varepsilon > 0$,
		there exist truncated tubes
		$\T_{T,\e}^{-}\subset \T_{T, b}\subset 
		\T_{T,\e}^{+} $ 
		and
		Borel subsets $\check{\Xi}_{\varepsilon}^- \subset \check{\Xi} \subset \check{\Xi}_{\varepsilon}^+$ of $\check N$, $\Xi_{\varepsilon} ^- \subset \Xi \subset \Xi_{\varepsilon} ^+$  of $N$ and $\Theta_{\varepsilon}^- \subset \Theta \subset \Theta_{\varepsilon}^+$  of $M$ satisfying the following:
		\begin{enumerate}	\item \label{itm:4.3.1} for all $T>0$,
			\begin{multline}
				\check N_{O(\varepsilon)}\check{\Xi}_{\varepsilon}^- \exp(\T_{T,\varepsilon}^-) M_{O(\varepsilon)} \Theta^-_{\varepsilon} \Xi_{\varepsilon} ^-N_{O(\varepsilon)}\subset S_{T,b,\varepsilon}^-  \\ \subset S_{T,b,\varepsilon}^+
				\subset \check N_{O(\varepsilon)}\check{\Xi}_{\varepsilon}^+ \exp(\T_{T,\varepsilon}^+) M_{O(\varepsilon)} \Theta^+_{\varepsilon} \Xi_{\varepsilon} ^+N_{O(\varepsilon)}
			\end{multline}
			
			\item \label{itm:4.3.2} an $O(\varepsilon)$-neighborhood of $\T_{T,\varepsilon}^-$ 
			(resp. $\T_{T, b}$) contains $\T_{T,b}$ (resp.  $\T_{T,\varepsilon}^+$);
			
			\item \label{itm:4.3.3} $\nu((\check{\Xi}_{\varepsilon}^+ -\check{\Xi}_{\varepsilon}^-) e^+ ) \to 0$, $\nu_{\involution}((\Xi_{\varepsilon} ^+ -\Xi_{\varepsilon} ^-) e^- ) \to 0$ and $\vol_M(\Theta_\varepsilon^+ - \Theta_\e^-) \to 0$ as $\varepsilon \to 0$.
		\end{enumerate}
	\end{lemma}
	
	\begin{proof}
		We prove the half of the lemma involving $-$ signs; the second half involving the $+$ signs is similar.
		
		For $T>0$, let $\T_{T,\varepsilon}^-$ (resp. $\check{\Xi}_{\varepsilon}^-,\Xi_{\varepsilon} ^-,\Theta_{\varepsilon}^-$) be the intersection of $\T_{T,b}$ (resp. $\check{\Xi}_{\varepsilon},\Xi_{\varepsilon} ,\Theta_{\varepsilon}$) and the complement of an $O(\varepsilon)$-neighborhood of its exterior, or more explicitly, 
		$$\T_{T,\varepsilon}^- = \T_{T,b} - (\partial \T_{T,b}+\LieA_{O(\varepsilon)}).$$
		To check (\labelcref{itm:4.3.1}), let $g\in \check N_{O(\varepsilon)}\check{\Xi}_{\varepsilon}^- \exp(\T_{T,\varepsilon}^-) M_{O(\varepsilon)} \Theta^-_{\varepsilon} \Xi_{\varepsilon}^-N_{O(\varepsilon)}$.  By \cref{lem:wavefront}(\ref{itm:wavefront3}), it follows that  for all $g_1,g_2 \in G_\varepsilon$, 
		$$g_1^{-1}gg_2^{-1} \in \check N_{O(\varepsilon)}\check{\Xi}_{\varepsilon}^- A_{O(\varepsilon)}\exp(\T_{T,\varepsilon}^-) M_{O(\varepsilon)} \Theta^-_{\varepsilon} \Xi_{\varepsilon}^-N_{O(\varepsilon)} \subset S_{T,b}.$$ This implies that $g \in S_{T,b,\varepsilon}^-$,
		which establishes (\ref{itm:4.3.1}).
		It is clear that (\labelcref{itm:4.3.2}) is satisfied. Moreover, since $\nu(\partial \check{\Xi}e^+) = 0$, we have
		$$|\nu(\check{\Xi} e^+) - \nu(\check{\Xi}_{\varepsilon}^-e^+)| \le \nu(\check{N}_{O(\varepsilon)}\partial \check{\Xi}e^+) \xrightarrow{\varepsilon \to 0} 0.$$
		Similarly for $\Xi_{\varepsilon} ^-$ and $\Theta_{\varepsilon}^-$ and hence (\labelcref{itm:4.3.3}) is satisfied. 
	\end{proof}
	
	\subsection*{An integral over tubes.} It is possible to generalize \cref{lem:STpmBounds} by replacing the truncated tubes $\T_{T,b}$ with a sequence of sufficiently nice compact subsets of $\LieA$ and then attempt to obtain the counting as in \cref{prop:STAsymptotic} following the well-known approaches of (\cite{EM93}, \cite{OS13}, \cite{MO15}, etc). However, in carrying out this in this higher rank and infinite volume setting, 
	putting \be\label{lto} L(\T_0)=\frac{\kappa_{\mathsf{v}}}{|m_{\mathcal{X}_\v}|}\int_{t\mathsf{v}+\sqrt{t}u \in \T_{0}}e^{\delta_\v t}e^{-I(u)} \, dt \, du\ee 
	where $\kappa_\v$ and $I(u)$ are as in \cref{thm:DecayofMatrixCoefficients}, we need to control the asymptotic of $L(\T_{T,b})$ and
	$L (\T_{T,\varepsilon}^\pm)$ as $T\to \infty$ as in Lemma \ref{lem:MainTermAsymptotic} where the shape of the tubes plays an important role. We will need  the the following technical lemma in the proof of Lemma \ref{lem:MainTermAsymptotic}.
	\begin{lemma}\label{ft} 
		Define $f_T:\ker \psi_\v \to \br$ by
		$$f_T(u) = \frac{1}{e^{\delta_\v T}}e^{- I(u)/T} \int_{R_T(u/\sqrt{T})} e^{\delta_\v t}\, dt \quad\text{ for $u\in \ker \psi_\v$} $$
		where $R_T(u) =  \{t\ge0 : t \le T +b(\sqrt{t}u), \sqrt{t}u \in \mathsf{K} \}$.
		We have
		\be \lim_{T\to\infty}f_T(u)= \begin{cases}  \frac{1}{\delta_\v}e^{\delta_\v b(u)} &\text{if $ u\in \inte \mathsf K$}  \\ 0 & \text{if $u\in\ker \psi_\v- \mathsf{K}$} .\end{cases} \ee 
	\end{lemma}
	\begin{proof}
		To compute the desired limit, we will first describe the values in $R_T(u)$. For convenience, let
		\begin{align*}
            & M_b=\max b; \qquad m_b=\min b;
            \\
			& J_{u}= \{s\ge 0 : su \in \mathsf{K}\} \quad \text{and} \quad J_{u}^2  = \{s^2: s \in J_{u}\} \quad \text{for } u\in \ker \psi_\v.
		\end{align*}
		Then
		\begin{equation*}
			\begin{aligned}
				R_T(u/\sqrt{T}) & = \{t \ge 0: t \le T+b(\sqrt{\tfrac{t}T}u), \, \sqrt{\tfrac{t}T}u \in \mathsf{K}\}
				\\
				& = \{t\ge 0: t \le T+b(\sqrt{\tfrac{t}T}u), \, t \in T J_{u}^2\}
				\\
				& \subset \left[0,\min\left\{T+M_b, \left(\sup J_{u}^2\right)T\right\}\right]			
			\end{aligned}
		\end{equation*} 
		
		First, fix $u \in \interior \mathsf{K}$. Then 
		$$s_u:=  \sup\{0\le s\le 1: s u \notin \mathsf K\}<1.$$ 
		Assume that $T$ is sufficiently large so that $$s_u T < T + m_b<T+M_b< \left(\sup J_{u}^2\right)T.$$
		Observe that 
		$[ s_u T , T + m_b] \subset R_T(u/\sqrt{T})$
		and 
		$R_T(u/\sqrt{T}) \cap (T+M_b, \infty) = \emptyset.$
		Note that since $s_u < 1$, we have
		\begin{equation}
			\label{eqn:ignore1}
			\frac{1}{e^{\delta_\v T}}e^{- I(u)/T} \int_0^{s_uT} e^{\delta_\v t}\, dt \le\frac{1}{e^{\delta_\v T}}e^{- I(u)/T} e^{\delta_\v s_uT}
		\end{equation}
		and hence the quantity on the left hand-side goes to $0$ as $T \to \infty$.
		
		Next we investigate which values in $[T+m_b, T+M_b]$ are in $R_T(u/\sqrt{T})$.
		Fix $\varepsilon > 0$. Using continuity of $b$, we note that for all $T$ sufficiently large and for all $\eta \in (\varepsilon, b(u) - m_b)$, we have
		$$T+b(u)-\eta < T+ b\left(\sqrt{\tfrac{T+b(u)-\eta}T}u\right)$$
		and hence\footnote{We write $o_T(1)$ for a function of $T$ that converges to $0$ as $T$ tends to $\infty$.}
		\begin{equation}
			\label{eqn:RT1}
			[s_u T , T + b(u) - o_T(1)] \subset R_T(u/\sqrt{T}).
		\end{equation}
		Similarly, we note that for all $T$ sufficiently large and for all $\eta \in (\varepsilon, M_b - b(u))$, we have
		$$T+b(u)+\eta > T+ b\left(\sqrt{\tfrac{T+b(u)+\eta}T}u\right)$$
		and hence 
		\begin{equation}
			\label{eqn:RT2}
			R_T(u/\sqrt{T})\cap [T+b(u)+o_T(1) , \infty)  = \emptyset.
		\end{equation}
		Also note that
		\begin{equation}
			\label{eqn:ignore2}
			\frac{e^{- I(u)/T}}{e^{\delta_\v T}} \int_{T+b(u)-o_T(1)}^{T+b(u)+o_T(1)} e^{\delta_\v t}\, dt = \frac{1}{e^{\delta_\v T}}e^{- I(u)/T} e^{\delta_\v b(u)}(e^{o_T(1)}-e^{-o_T(1)})
		\end{equation}
		goes to $0$ as $T \to \infty$.
		Summarizing \eqref{eqn:ignore1}-\eqref{eqn:ignore2}, we have shown that to compute $\lim_{T\to\infty}f_T(u)$, we can replace $R_T(u/\sqrt{T})$ with $[s_T, T+b(u)-o_T(1)]$ to get
		\begin{align*}
			\lim_{T\to\infty}f_T(u) & = \lim_{T\to\infty}\frac{1}{e^{\delta_\v T}}e^{- I(u)/T} \int_{R_T(u/\sqrt{T})} e^{\delta_\v t}\, dt
			\\
			& = \lim_{T\to\infty}\frac{1}{\delta_\v e^{\delta_\v T}}e^{- I(u)/T} (e^{\delta_\v(T+b(u)-o_T(1))} - o_T(e^{\delta_\v T}))
			\\
			& = \frac{1}{\delta_\v}e^{\delta_\v b(u)}.
		\end{align*}
		
		Next, fix $u \notin \mathsf{K}$. Since $\mathsf{K}$ is closed, there exists $\eta >0$ such that $(1-\eta,1+\eta)\cap J_u = \emptyset$. Then for all sufficiently large $T$, we have
		\begin{equation*}
			R_T(u/\sqrt{T}) 
			= \{t\ge 0: t \le T+b(\sqrt{\tfrac{t}T}u), \, t \in T J_{u}^2\} \subset [0,(1-\eta)^2T].
		\end{equation*}
		Then
		$$f_T(u) \le \frac{1}{e^{\delta_\v T}}e^{- I(u)/T} \int_0^{(1-\eta)^2T} e^{\delta_\v t}\, dt \le  \frac{1}{e^{\delta_\v T}}e^{- I(u)/T} e^{\delta_\v (1-\eta)^2T}$$
		and hence $\lim_{T\to\infty}f_T(u) = 0$.
	\end{proof}

	\begin{lemma}
		\label{lem:MainTermAsymptotic}
		We have 
		\begin{equation}
			\label{eqn:MainTermAsymptotic}
			\lim\limits_{T\to\infty}\frac{1}{e^{\delta_\v T} T^{(1-\rank)/2}}\int_{t\mathsf{v}+\sqrt{t}u \in \T_{T,b}}e^{\delta_\v t} e^{- I(u)} \, dt \, du =\frac{1}{\delta_\v}\int_{\mathsf{K}} e^{\delta_\v b(u)} \, du.
		\end{equation}
	\end{lemma}
	\begin{proof}
		We may assume $\T_{T,b} = \{t\v+u:0\le t\le T+b(u), \, u \in \mathsf{K}\}$ as $\T_{T,b}- \LieA^+$ is contained in a fixed compact set independent of $T$.  Then
		\[R_T(u) = \{t\ge0 : t \le T +b(\sqrt{t}u), \sqrt{t}u \in \mathsf{K} \}.\] 
		By changing the variable $\sqrt{T}u$ to $u$, we have
		\begin{equation*}
			\begin{aligned}[b]
				& \frac{1}{ e^{\delta_\v T}T^{(1-\rank)/2}}\int_{t\mathsf{v}+\sqrt{t}u \in \T_{T,b}}e^{\delta_\v t} e^{- I(u)} \, dt \, du 
				\\
				& = \frac{1}{ e^{\delta_\v T}T^{(1-\rank)/2}}\int_{\ker\psi_\v}e^{- I(u)} \int_{R_T(u)} e^{\delta_\v t}\, dt \, du
				\\
				& = \int_{\ker\psi_\v}\frac{1}{e^{\delta_\v T}}e^{- I(u)/T} \int_{R_T(u/\sqrt{T})} e^{\delta_\v t}\, dt \, du\\ & =\int_{\ker\psi_\v} f_T.
			\end{aligned}
		\end{equation*} Set
		$$\mathsf A_T= \int_{u \in \interior \mathsf{K}} f_T, \;\; \mathsf B_T= \int_{u\in \mathrm{hull} (\mathsf{K}\cup \{0\}) -\mathsf K} f_T \;\; \text{ and }\;\; \mathsf C_T= \int_{u\notin \mathrm{hull}(\mathsf{K} \cup \{0\})} f_T $$ where $\mathrm{hull}\mathsf K\cup\{0\}$ means the convex hull of $\mathsf K \cup \{0\}$.
		By  hypothesis that $\partial \mathsf K$ has measure zero, we have
		$\int_{\ker\psi_\v} f_T= \mathsf A_T + \mathsf B_T + \mathsf C_T$.

		\noindent{\bf Asymptotic of $\mathsf A_T$.}
		Since
		\begin{equation*}
			\begin{aligned}
				R_T(u/\sqrt{T}) & = \{t \ge 0: t \le T+b(\sqrt{\tfrac{t}T}u), \, \sqrt{\tfrac{t}T}u \in \mathsf{K}\}
				\\
				& = \{t\ge 0: t \le T+b(\sqrt{\tfrac{t}T}u), \, t \in T J_{u}^2\}
				\\
				& \subset \left[0,\min\left\{T+M_b, \left(\sup J_{u}^2\right)T\right\}\right]			
			\end{aligned}
		\end{equation*} 
		where $J_{u}= \{s\ge 0 : su \in \mathsf{K}\}$, we have
		for all $u \in \ker\psi_\v$,
		\begin{equation*}
			\label{eqn:constantbound}
			f_T(u) \le  \frac{1}{e^{\delta_\v T}} \int_0^{T+M_b} e^{\delta_\v t}\, dt \le\frac{1}{\delta_\v}e^{\delta_\v M_b}. 
		\end{equation*}
		
		Since $\mathsf K$ is bounded and the integral in \eqref{eqn:constantbound} is uniformly bounded for all $T$
		we can apply the Lebesgue dominated convergence theorem for the sequence $f_T|_{\inte \mathsf K}$  to compute the asymptotic of $\mathsf A_T$ using \cref{ft}:
		\begin{equation*}
			\lim_{T\to\infty} {\mathsf A_T}=
			 \int_{u \in \interior \mathsf{K}} \lim_{T\to\infty}f_T( u) du = \frac{1}{\delta_\v}\int_{u \in \interior \mathsf{K}}e^{\delta_\v b(u)} \, du.
		\end{equation*}
		
		Hence it suffices to show that $\mathsf B_T+\mathsf C_T$ converges to $0$ as $T\to \infty$.
		
		\noindent{\bf Asymptotic of $\mathsf B_T$.} 
		For $u\in \mathrm{hull} (\mathsf{K} \cup \{0\})$, we still have the constant upper bound in \eqref{eqn:constantbound}.
		Hence we can also apply the Lebesgue dominated convergence theorem and Lemma \ref{ft} to compute 
		\begin{equation*}
			\lim_{T\to\infty} \mathsf{B}_T = \int_{u\in \mathrm{hull} (\mathsf{K}\cup \{0\}) -\mathsf K}\lim_{T\to\infty}f_T(u) \, du = \int_{u\in \mathrm{hull} (\mathsf{K}\cup \{0\}) -\mathsf K}0 \, du = 0
		\end{equation*} 
		
		\noindent{\bf Asymptotic of $\mathsf C_T$.} 
		Fix $u\notin \mathrm{hull}(\mathsf K\cup \{0\})$. Then for all $t \ge 1$, $tu\notin \mathsf{K}$ otherwise $u$ would lie on the segment with endpoints $0$ and $tu$ and be in $\mathrm{hull}(\mathsf{K} \cup \{0\})$. In particular, we have 
		$$j_{u}: = \sup( J_{u})=
		\sup \{0\le s\le 1 : su \in \mathsf{K}\} <1. $$ 
		Define $$g(u):={ 2\sqrt{\delta_\v I(u)(1-j_{u}^2)}} .$$
		Using
		$R_T(u/\sqrt{T}) \subset TJ_{u}^2 \subset T[0,j_{u}^2]$,
		we get for all $T>0$,
		\begin{multline*}
			f_T(u) =\frac{1}{e^{\delta_\v T}}e^{- I(u)/T} \int_{R_T(u/\sqrt{T})} e^{\delta_\v t}\, dt \le \tfrac{1}{\delta_\v}e^{- I(u)/T}e^{-\delta_\v T(1-j_{u}^2)} \le  \tfrac{1}{\delta_\v}e^{- g(u)}\end{multline*}
		where the last inequality uses the fact that $a+b \ge 2\sqrt{ab}$ for all $a,b \ge 0$.
		Note that the function  $g(u)$ is radially increasing with at least a linear rate: for all $r>1$, we have
		$$ g(ru)={2 \sqrt{\delta_\v I(r u)(1-j_{r u}^2)}} =2r\sqrt{\delta_\v I(u)(1-j_{u}^2/r^2)} \ge r g(u),$$
		so the function $u\mapsto  e^{-g(u)}$ is $L^1$-integrable on $\ker\psi_\v -\mathrm{hull}(\mathsf{K} \cup \{0\})$.
		Hence we apply the Lebesgue dominated convergence theorem to compute the asymptotic of $\mathsf C_T$ using \cref{ft}:
		\begin{equation*}
			\lim_{T\to\infty} \mathsf{C}_T = \int_{u\notin \mathrm{hull} (\mathsf{K}\cup \{0\})}\lim_{T\to\infty}f_T(u) \, du = \int_{u\notin \mathrm{hull} (\mathsf{K}\cup \{0\})}0 \, du = 0
		\end{equation*} 
		This finishes the proof of Lemma \ref{lem:MainTermAsymptotic}.
	\end{proof}
	
	We are now ready to give the proof of \cref{prop:STAsymptotic}.
	
	\subsection*{Proof of \cref{prop:STAsymptotic}.}
	We first introduce some notation. Given a bounded Borel subset $B$ of $G$, define the counting function $F_B: G \times G \to \N$ by
	$$F_B(g,h) = \sum_{\gamma \in \Gamma} 1_B(g^{-1}\gamma h).$$ 
	Note that $ F_{S_{T,b}}(e,e) = \#\Gamma \cap S_{T,b}$.
	The function $F_B$ is $\Gamma$-invariant in both arguments so it descends to a function on $\Gamma \backslash G \times \Gamma\backslash G$ which we still denote by $F_B$. For $F_1,F_2 :\Gamma \backslash G \times \Gamma\backslash G \to \R$, let 
	$$\langle F_1,F_2 \rangle = \int_{\Gamma \backslash G \times \Gamma\backslash G}F_1(x_1,x_2)F_2(x_1,x_2) \, dx_1 \, dx_2$$
	when the integral makes sense where $dx_1,dx_2$ are both the Haar measure on $\Gamma \backslash G$. For $\varepsilon>0$ smaller than the injectivity radius of $\Gamma$ at $e$, we fix a nonnegative function $\psi_\varepsilon \in C^\infty(G)$ with $\supp \psi_\varepsilon \subset G_\varepsilon$ and $\int_G \psi_\varepsilon \, dg = 1$. Let $\Psi_\varepsilon \in C^\infty(\Gamma \backslash G)$ be defined by $\Psi_\varepsilon(\Gamma g) = \sum_{\gamma \in \Gamma} \psi_\varepsilon(\gamma g)$ for all $g \in G$. 
	
	Observe that we have
	\begin{equation}
    \label{Ineq}
		\langle F_{S_{T,b,\varepsilon}^-}, \Psi_\varepsilon \otimes \Psi_\varepsilon \rangle \le F_{S_{T,b}}(e,e)  \le \langle F_{S_{T,b,\varepsilon}^+}, \Psi_\varepsilon \otimes \Psi_\varepsilon \rangle.
	\end{equation}

	Given a bounded Borel subset $B \subset G$, we define $f_{B}:\check{N}\times MN \to \R$ by
	\begin{equation*}
		f_{B}(h,mn) = \tfrac{\kappa_{\mathsf{v}}}{|m_{\mathcal{X}_\v}|}\int_{h^{-1}a_{t\mathsf{v}+\sqrt{t}u}mn \in B}e^{\delta_\v t}e^{-I(u)} \, dt \, du.
	\end{equation*}
	Continuing the notation \eqref{lto}, set $$L(\T_{T,b})=\tfrac{\kappa_{\mathsf{v}}}{|m_{\mathcal{X}_\v}|}\int_{t\mathsf{v}+\sqrt{t}u \in \T_{T,b}}e^{\delta_\v t}e^{-I(u)} \, dt \, du$$ which is the value of $f_{S_{T,b}}(h,mn)$ when it is nonzero.

 We denote by   
	\begin{align*} 
		&I_1: MAN\check{N} \to \LieA, &J_1: MAN\check{N} \to \check{N},
		\\ 
		&I_2: A\check{N}MN \to \LieA, &J_2: A\check{N}MN \to MN
	\end{align*}
	the natural projection maps. 
	For a bounded subset $B \subset G$, define
	\begin{equation*}
		\tilde f_{B}(g_1,g_2) =
		f_{B} {(J_1(g_1^{-1}),J_2(g_2^{-1})) } e^{\psi_\v(I_1(g_1^{-1})-I_2(g_2^{-1}))}
	\end{equation*} for any $g_i\in G$ such that $J(g_i^{-1})$ and
	$I(g_i^{-1})$ are defined for $i=1,2$.
	
	The first step is to rewrite $\langle F_{S_{T,b,\varepsilon}^\pm}, \Psi_\varepsilon \otimes \Psi_\varepsilon \rangle$ by decomposing the Haar measure on $G$ using $\check{N}AMN$ coordinates and separating the expected main term from local mixing from the error term. 
	Set $Q_{S_{T,b,\varepsilon}^\pm} $ to be 
	\begin{multline}
		\label{eqn:Q}
		\int_{K/M\times K}\int_{G_\varepsilon \times G_\varepsilon} \tilde f_{S_{T,b,\varepsilon}^\pm}(g_1^{-1} k_1,g_2^{-1} k_2) 
		\cdot  \psi_\varepsilon(g_1) \psi_\varepsilon(g_2)     \, dg_1  \, dg_2 \, d\nu(k_1^+) \, d\nu_{\involution}(k_2^+).
	\end{multline}    	
    
	Then as in \cite[Lemmas 5.4-5.8]{CF23}, we have 
	\begin{align*}
		& \langle F_{S_{T,b,\varepsilon}^\pm}, \Psi_\varepsilon \otimes \Psi_\varepsilon \rangle
		\\ 
		& = Q_{S_{T,b,\varepsilon}^\pm} +\int_{h^{-1}a_{t\mathsf{v}+\sqrt{t}u}mn \in S_{T,b,\varepsilon}^\pm}e^{\delta_\v t}E_\varepsilon(t,u,h,mn) \, dt \, du \, dh \, dm  \, dn
	\end{align*}
	where
	\begin{multline*}
		E_\varepsilon(t,u,h,mn) = t^{(\rank-1)/2}e^{2\rho(t\mathsf{v}+\sqrt{t}u)-\delta_\v t}\int_{\Gamma \backslash G} \Psi_\varepsilon(xh)\Psi_\varepsilon(xa_{t\mathsf{v}+\sqrt{t}u}mn) \, dx
		\\
		- \tfrac{\kappa_{\mathsf{v}}e^{-I(u)}}{|m_{\mathcal{X}_\v}|}\BRstar\bigr(h^{-1}.\Psi_\varepsilon)\BR\bigr((mn)^{-1}.\Psi_\varepsilon). 
	\end{multline*}

	In view of \eqref{eqn:Q} and \cref{lem:STpmBounds}, we only need to consider $k_1 \in K$ such that $J_1(k_1^{-1}g_1)$ is in a bounded set. Similarly for $k_2$. Then by \cref{lem:wavefront}(\ref{itm:wavefront1})-(\ref{itm:wavefront2}), we get
	\begin{align*}
		&\psi_\v(I_1(k_1^{-1}g_1))=\psi_\v(I_1(k_1^{-1}))+O(\varepsilon);
		&\psi_\v(I_2(k_2^{-1}g_2))=\psi_\v(I_2(k_2^{-1}))+O(\varepsilon)
	\end{align*} 
	and
	\begin{equation}
		\label{eqn:J}
		\begin{aligned}[b]
			&J_1(k_1^{-1}g_1)^{-1} \in \check{N}_{O(\varepsilon)} J_1(k_1^{-1})^{-1};
			&J_2(k_2^{-1} g_2) \in M_{O(\varepsilon)}J_2(k_2^{-1})N_{O(\varepsilon)}.
		\end{aligned}
	\end{equation}
	
	Let $\T_{T,\varepsilon}^- = \T(\v,\mathsf{K}_\varepsilon^-,T,b_\varepsilon^-) \subset  \T_{T,b}$, $\check{\Xi}_{\varepsilon}^- \subset \check{\Xi}$, $\Xi_{\varepsilon} ^- \subset \Xi$ and $\Theta_{\varepsilon}^- \subset \Theta$ be as in \cref{lem:STpmBounds}. For convenience, let 
	$$S_{T,b_\varepsilon^-}:=\check{\Xi}_{\varepsilon}^- \exp(\T_{T,\varepsilon}^-) \Theta_{\varepsilon}^- \Xi_{\varepsilon} ^-.$$
	By \cref{lem:STpmBounds}(\labelcref{itm:4.3.1}) and \eqref{eqn:J}, for all $g_1,g_2\in G_\varepsilon$ we have 
	$$f_{S_{T,b,\varepsilon}^-}(J_1(k_1^{-1}g_1),J_2(k_2^{-1}g_2)) \ge  f_{S_{T,b_\varepsilon^-}}(J_1(k_1^{-1}),J_2(k_2^{-1})).$$
	It now follows that
	\begin{equation}
		\label{eqn:STAsymptotic1}
		Q_{S_{T,b,\varepsilon}^-}
		\ge \int_{K/M\times K}  (1+O(\varepsilon)) \tilde  f_{S_{T,b_\varepsilon^-}}(k_1, k_2) \,  d\nu(k_1^+) \, d\nu_{\involution}(k_2^+).
	\end{equation}
	
	Using \cref{lem:STpmBounds}(\labelcref{itm:4.3.2})-(\labelcref{itm:4.3.3}), we observe that\footnote{\label{bigO2}For a function $f$ of $T,\varepsilon$, we write $O(f)$ for a function which is in absolute value at most $Cf$ for some constant $C>0$ independent of $T,\e$.}
	\begin{equation}
		\label{eqn:STAsymptotic2}
		\begin{aligned}[b] 
			&\int_{K/M\times K}  (f_{S_{T,b}} - f_{S_{T,b_\varepsilon^-}})(J_1(k_1^{-1}),J_2(k_2^{-1})) \, d\nu(k_1^+)\, d\nu_{\involution}(k_2^+)
			\\ 
			&=L(\T_{T,b})O\left(\nu\left((\check{\Xi}- \check{\Xi}_{\varepsilon}^-)e^+\right) + \nu_{\involution}\left((\Xi^{-1}- (\Xi_{\varepsilon} ^-)^{-1})e^- \right) + \vol_M\left(\Theta-\Theta_{\varepsilon}^-\right) \right)
			\\
			& \qquad \qquad \qquad + (L(\T_{T,b}) - L(\T_{T,\varepsilon}^-))\nu(\check{\Xi}_{\varepsilon}^-)\nu((\Xi_{\varepsilon} ^-)^{-1})\vol_M(\Theta^-)
			\\
			& = L(\T_{T,b})o_\varepsilon(1)+O(L(\T_{T,b}) - L(\T_{T,\varepsilon}^-))
			\\
			& = L(\T_{T,b})o_\varepsilon(1)
		\end{aligned}
	\end{equation}
	where we note that $L(\T_{T,b}) - L(\T_{T,\varepsilon}^-) = L(\T_{T,b})o_\varepsilon(1)$ by \cref{lem:MainTermAsymptotic}.\footnotemark 
	\footnotetext{For a real-valued function $f$ of $\varepsilon$, we write $f = o_\varepsilon(1) \iff \lim_{\varepsilon \to 0}f(\varepsilon)=0$}
	
	Combining \eqref{eqn:STAsymptotic1} and \eqref{eqn:STAsymptotic2} yields
	\be  Q_{S_{T,b,\varepsilon}^-}
	\ge\int_{K/M\times K} (1+O(\varepsilon)) \tilde f_{S_{T,b}}(k_1, k_2)  d\nu(k_1^+) \, d\nu_{\involution}(k_2^+)
	+ L(\T_{T,b})o_\varepsilon(1).
	\ee 
	A similar argument shows that
	\be Q_{S_{T,b,\varepsilon}^+}
	\le\int_{K/M\times K} (1+O(\varepsilon)) \tilde f_{S_{T,b}}(k_1, k_2)  d\nu(k_1^+) \, d\nu_{\involution}(k_2^+)
	+ L(\T_{T,b})o_\varepsilon(1)
	\ee 
	and we conclude that
	\be  Q_{S_{T,b,\varepsilon}^\pm}
	= \int_{K/M\times K} (1+O(\varepsilon)) \tilde f_{S_{T,b}} \, d\nu(k_1^+) \, d\nu_{\involution}(k_2^+)
	+ L(\T_{T,b})o_\varepsilon(1).
	\ee 
	
	Considering $k_1 \in K$ such that $k_1^{-1} = ma_wnh^{-1} \in MAN\check{N},$ we have $J_1(k_1^{-1}) = h^{-1}$, $h^+ = k_1M$, $I_1(k_1^{-1}) = w = \beta_{h^+}(e,h)$ and hence
	\begin{align*}
		\int_{{k_1^+\in K/M ,\,  (J_1(k_1^{-1}))^{-1} \in \check{\Xi}}} e^{\psi_\v(I_1(k_1^{-1}))}  \, d\nu(k_1^+)
		= \int_{h \in \check{\Xi}} e^{\psi_\v(\beta_{h^+}(e,h))}  \, 	d\nu(h^+) = \tilde{\nu}(\check{\Xi}).
	\end{align*}
	
	Similarly, considering $k_2 \in K$ such that $k_2^{-1} = a_whm^{-1}n^{-1} \in A\check{N}MN,$ we have $J_2(k_2^{-1}) = m^{-1}n^{-1} \in MN$, $k_2^- = n^-$, $I_2(k_2^{-1}) =w=-i(\beta_{n^-}(e,n))$ and hence
	\begin{align*}
		\int_{{k_2 \in K ,\, J_2(k_2^{-1}) \in \Theta\Xi }}  e^{-\psi_\v(I_2(k_2^{-1}))}d\nu_{\involution}(k_2^+) & = \int_{nm \in \Xi ^{-1}\Theta^{-1}}e^{(\psi_\v\circ\involution)(\beta_{n^-}(e,n))}\, d\nu_{\involution}(n^-) \, dm 
		\\
		& = \tilde{\nu}_{\involution}(\Xi ^{-1})\vol_M(\Theta).
	\end{align*}
	
	Thus, we obtain
	\begin{align*}
		& \langle F_{S_{T,b,\varepsilon}^\pm}, \Psi_\varepsilon \otimes \Psi_\varepsilon \rangle
		\\
		& = (1+O(\varepsilon))L(\T_{T,b})\tilde{\nu}(\check{\Xi})\tilde{\nu}_{\involution}(\Xi ^{-1})\vol_M(\Theta) 
		\\
		& \qquad +\int_{h^{-1}a_{t\mathsf{v}+\sqrt{t}u}mn \in S_{T,b,\varepsilon}^\pm}e^{\delta_\v t}E_\varepsilon(t,u,h,mn) \, dt \, du \, dh \, dm  \, dn +L(\T_{T,b})o_\varepsilon(1).
	\end{align*}
	
	For the error term $E_\varepsilon(t,u,h,mn)$, we claim that 
	\begin{equation}
		\label{eqn:ErrorTermAsymptotic}
		\lim\limits_{T\to\infty}\frac{1}{e^{\delta_\v T} T^{(1-\rank)/2}}\int_{h^{-1}a_{t\mathsf{v}+\sqrt{t}u}mn \in S_{T,b,\varepsilon}^\pm}e^{\delta_\v t}E_\varepsilon(t,u,h,mn) \, dt \, du \, dh \, dm  \, dn = 0.
	\end{equation}
	
	To prove \eqref{eqn:ErrorTermAsymptotic}, note that since $\check{\Xi} \subset \check{N}$ and $\Xi \subset N$ are bounded, by \cref{thm:DecayofMatrixCoefficients}, there exist positive constants $\eta_{\mathsf{v}}, D_{\mathsf{v}}$ and $s_{\mathsf{v}}$ such that  
	\[|E_\varepsilon(t,u,h,mn)| \le D_{\mathsf{v}} e^{-\eta_{\mathsf{v}} I(u)}\] 
	for all $(t,u) \in (s_{\mathsf{v}},\infty) \times \ker\psi_\mathsf{v}$ and $h,mn$ such that $h^{-1}a_{t\mathsf{v}+\sqrt{t}u}mn \in S_{T,b,\varepsilon}^\pm$. Then \eqref{eqn:ErrorTermAsymptotic} follows by using similar reasoning as in \cref{lem:MainTermAsymptotic} and the fact that for fixed $u \in \ker\psi_\v$, $\lim\limits_{t\to\infty}E_\varepsilon(t,u,h,mn) = 0$.
	
	Now using \cref{lem:MainTermAsymptotic} and \eqref{Ineq}, taking $T\to\infty$ and then $\varepsilon \to 0$, we conclude that
	\[\#(\Gamma \cap S_T) \sim \frac{\kappa_{\mathsf{v}}}{\delta_\v|m_{\mathcal{X}_\v}|}\int_{\mathsf{K}} e^{\delta_\v b(u)} \, du \cdot \tilde{\nu}(\check{\Xi})\tilde{\nu}_{\involution}(\Xi ^{-1})\vol_M(\Theta).\]    
	This finishes the proof of Proposition \ref{prop:STAsymptotic}.
	
	\subsection*{Counting via flow boxes.}
	We use \emph{flow boxes} as in \cite{Mar04} and \cite{MMO14}.
	
	\begin{definition}[$\varepsilon$-flow box at $g_0$]
		\label{def:FlowBox}
		Given $g_0 \in G$ and $\varepsilon > 0$, the \emph{$\varepsilon$-flow box at $g_0$} is defined by \[\mathcal{B}(g_0,\varepsilon)=g_0(\check{N}_\varepsilon N \cap N_\varepsilon \check{N} AM)M_\varepsilon A_\varepsilon.\]
		We denote the projection of $\mathcal{B}(g_0,\varepsilon)$ into $\Gamma \backslash G/M$ by $\tilde{\mathcal{B}}(g_0,\varepsilon)$. 
	\end{definition}
	
	For $g_0\in G$ and $T, \varepsilon >0$, we denote
	\begin{equation}
		\label{eqn:VTDefinition}
		\mathcal{V}_{T,b}(g_0,\varepsilon,\T,\Theta) = \mathcal{B}(g_0,\varepsilon)\T_{T,b}\Theta\mathcal{B}(g_0,\varepsilon)^{-1}.
	\end{equation}
	Our next goal is to obtain an asymptotic for $\# \Gamma \cap \mathcal{V}_{T,b}(g_0,\varepsilon,\T,\Theta)$.

	\begin{proposition}
		\label{prop:CountingInVT} 
		Let $g_0 \in G$.
		For all sufficiently small $\varepsilon>0$, we have  
		\begin{multline*}
			\#(\Gamma\cap\mathcal{V}_{T,b}(g_0,\varepsilon,\T,\Theta))
			\\
			=c(\v, b)\left( \frac{\BMS(\tilde{\mathcal{B}}(g_0,\varepsilon))}{b_{\rank}(\varepsilon)}\vol_M(\Theta) (1+O(\varepsilon)) +o_T(1)\right)\frac{e^{\delta_\v T}}{ T^{(\rank-1)/2}}
		\end{multline*}
		where $c(\v,b)$ is as in \eqref{cv} and $b_r(\varepsilon)$ denotes the volume of the Euclidean $\rank$-ball of radius $\varepsilon$.
	\end{proposition}
	
	The asymptotic in \cref{prop:CountingInVT} will be deduced from \cref{prop:STAsymptotic}. Using another wavefront-type lemma argument, we show the sets $S_{T,b}$ and $\mathcal{V}_{T,b}(e,\varepsilon,\T,\Theta)$ are approximately the same in the following precise sense:
	\begin{lemma}
		\label{lem:VTandST}
		There exists $C>0$ such that for all sufficiently small $\varepsilon >0$ and for all sufficiently large $T,T'$ with $T>T'$, we have
		\begin{align*}
			& \mathcal{V}_{T,b}(e,\varepsilon,\T,\Theta) - \mathcal{V}_{T',b}(e,\varepsilon,\T,\Theta)
			\\
			& \subset S_{T,b_\varepsilon^+}(\check{N}_{\varepsilon + O(\varepsilon e^{-C T'})},(N_{\varepsilon + O(\varepsilon e^{-C T'})})^{-1},\T^+_{\varepsilon},\Theta^+_{\varepsilon})
			\\
			& \qquad - S_{T'-O(\varepsilon),b_\varepsilon^+}(\check{N}_{\varepsilon + O(\varepsilon e^{-C T'})},(N_{\varepsilon + O(\varepsilon e^{-CT'})})^{-1},\T^+_{\varepsilon},\Theta^+_{\varepsilon})
		\end{align*}
		where $\mathsf{K}_\varepsilon^+$ and $b_\varepsilon^+$ are defined by the equation
		$\T_{T,b} +\LieA_{O(\varepsilon)}  = \T_T(\v,\mathsf{K}_\varepsilon^+,b_\varepsilon^+) $, 
		$\T_\varepsilon^+ := \T(\v,\mathsf{K}_\varepsilon^+)$, and $\Theta^+_\varepsilon = \bigcup_{m_1,m_2 \in M_{O(\varepsilon)}}m_1\Theta m_2$. \end{lemma}
	
	\begin{proof} 
		Since $\v\in \inte \fa^+$,  there exists $T_0>0$ sufficiently large  such that the linear forms $\psi_\v$ and $\alpha \in \Phi^+$  are all positive on $\T - \T_{T_0,b}$. 
		Therefore there exists a constant $C>0$ such that
		\begin{equation}
			\label{eqn:min}
			C\cdot \psi_\v  (w) \le \min_{\alpha\in\Phi^+}\alpha (w)
		\end{equation}
		for all $w\in\LieA^+$ in the $O(\varepsilon)$-neighborhood of $\T- \T_{T_0,b}$. By \eqref{eqn:Ad} and \eqref{eqn:min},
		for all $w\in\LieA^+$ in the $O(\varepsilon)$-neighborhood of $\T-\T_{T_0,b}$, we have
		\begin{equation}
			\label{eqn:N}
			a_{-w}N_\varepsilon a_w 
			\subset N_{\varepsilon e^{-C\psi_\v(w)}}. \end{equation} 
		Similarly, we have 
		\begin{equation} a_{w}\check{N}_\varepsilon a_{-w}  \subset \check{N}_{\varepsilon e^{-C\psi_\v(w)}}.
		\end{equation} 
		
		Let $T>T'>T_0$ and $\T_{T,\varepsilon}^+ = \T_T(\v,\mathsf{K}_\varepsilon^+,b_\varepsilon^+)$. Let 
		$$g \in \mathcal{V}_{T,b}(e,\varepsilon,\T,\Theta) -\mathcal{V}_{T',b}(e,\varepsilon,\T,\Theta).$$ 
		Write $g=g_1amg_2^{-1}$ where $g_1,g_2 \in \mathcal{B}(e,\varepsilon)$, $a \in \exp(\T_{T,b} - \T_{T',b})$ and $m \in \Theta$. Since $g_1, g_2 \in \mathcal{B}(e,\varepsilon)$, by \cref{lem:transversality}(\ref{itm:transversality1}) we can write $g_1 = h_1n_1m_1a_1 \in \check{N}_\varepsilon N_{O(\varepsilon)}M_\varepsilon A_\varepsilon$ and $g_2 = n_2h_2m_2a_2 \in N_\varepsilon \check{N}_{O(\varepsilon)}M_{O(\varepsilon)}A_{O(\varepsilon)}$. Then
		\begin{equation*}
			g = h_1n_1m_1a_1ma (m_2a_2)^{-1}h_2^{-1}n_2^{-1}.
		\end{equation*}
		
		Let $a' = aa_1a_2^{-1}$, $m'=m_1mm_2^{-1}$ and $n_3 = (m'a')^{-1}n_1m'a'$. Then $a' \in \T^+_{T,\varepsilon} - \T^+_{T'-O(\varepsilon),\varepsilon}$, $m' \in \Theta_\varepsilon^+$ and $n_3\in N_{O(\varepsilon e^{-C T'})}$ by \eqref{eqn:N}. By \cref{lem:transversality}(\ref{itm:transversality1}), we can write
		$n_3h_2^{-1} = m_3a_3h_4n_4 \in M_{O(\varepsilon)}A_{O(\varepsilon)}\check{N}_{O(\varepsilon)}N_{O(\varepsilon e^{-C T'})}$. 
		Then 
		\begin{multline*}
			g = h_1m'a'n_3h_2^{-1}n_2^{-1} = h_1m'a'm_3a_3h_4n_4n_2^{-1} \\ = h_5m''a''n_5 \in \check{N}_{\varepsilon+O(\varepsilon e^{-C T'})}\Theta^+_\varepsilon\exp(\T^+_{T,\varepsilon} - \T^+_{T'-O(\varepsilon),\varepsilon})N_{\varepsilon+O(\varepsilon e^{-C T'})}.
		\end{multline*}
		where $a'' = a'a_3$, $m''=m'm_3$, $n_5 =n_4n_2^{-1}\in N_{\varepsilon+O(\varepsilon e^{-C T'})}$ and $h_5 = h_1(m''a'')h_4(m''a'')^{-1} \in \check{N}_{\varepsilon+O(\varepsilon e^{-C T'})}$. This completes the proof.
	\end{proof}
	
	\subsection*{Proof of Proposition \ref{prop:CountingInVT}.}
	It suffices to consider the case $g_0=e$. Note that the boundaries  $\partial N_\varepsilon$ and $\partial \check{N}_\varepsilon$ are proper real algebraic subvarieties of $\Fboundary$ and hence $\nu(\partial \check{N}_\varepsilon) = \nu_{\involution}(\partial N_\varepsilon) = 0$ by \cite[Theorem 1.1]{KO23c}.  We have a trivial inclusion $S_{T,b}(\check{N}_\varepsilon,N_\varepsilon^{-1},\T,\Theta) \subset \mathcal{V}_{T,b}(e,\varepsilon,\T,\Theta)$. By \cref{prop:STAsymptotic}, we have
	\begin{multline}
		\label{eqn:CountinginVT1}
		\#(\Gamma\cap \mathcal{V}_{T,b}(e,\varepsilon,\T,\Theta)) 
		\ge c(\v,b)\left(\tilde{\nu}(\check{N}_\varepsilon)\tilde{\nu}_{\involution}(N_\varepsilon)\vol_M(\Theta) + o_T(1)\right)\frac{e^{\delta_\v T}}{T^{(\rank-1)/2}}.
	\end{multline}
	
	By \cite[Lemma 5.20]{CF23}, we have
	\begin{equation}
		\label{eqn:BMSofflowbox}
		\BMS(\tilde{\mathcal{B}}(e,\varepsilon)) = (1+O(\varepsilon)) b_{\rank}(\varepsilon)\tilde{\nu}(\check{N}_\varepsilon)\tilde{\nu}_{\involution}(N_\varepsilon).
	\end{equation}
	Then using  \eqref{eqn:CountinginVT1} and \eqref{eqn:BMSofflowbox}, we get
	\begin{multline*}
		\#(\Gamma\cap \mathcal{V}_{T,b}(e,\varepsilon,\T,\Theta))
		\\
		\ge c(\v,b) \left( \frac{\BMS(\tilde{\mathcal{B}}(e,\varepsilon))}{b_{\rank}(\varepsilon)} \vol_M(\Theta) (1+O(\varepsilon))+o_T(1)\right)\frac{e^{\delta_\v T}}{T^{(\rank-1)/2}}.
	\end{multline*}
	
	It remains to establish the reverse inequality. By \cref{lem:VTandST}, we have
	\begin{multline*}
		\mathcal{V}_{T,b}(e,\varepsilon,\T,\Theta) - \mathcal{V}_{T/2,b}(e,\varepsilon,\T,\Theta)
		\\
		\subset S_{T,b_\varepsilon^+}(\check{N}_{\varepsilon + O(\varepsilon e^{-C T/2})},(N_{\varepsilon + O(\varepsilon e^{-C T/2})}^-)^{-1},\T^+_\varepsilon,\Theta^+_\varepsilon)
	\end{multline*}
	and 
	\begin{equation*}
		\mathcal{V}_{T/2,b}(e,\varepsilon,\T,\Theta) \subset S_{T/2,b_\varepsilon^+}(\check{N}_{O(\varepsilon)},(N_{O(\varepsilon)})^{-1},\T^+_\varepsilon,\Theta^+_\varepsilon).
	\end{equation*}
	Then using \cref{prop:STAsymptotic} in the above inclusions, we get
	\begin{equation}
		\label{eqn:CountinginVT3}
		\begin{aligned}[b]
			&\#(\Gamma\cap\mathcal{V}_{T,b}(e,\varepsilon,\T,\Theta)) -\#(\Gamma\cap \mathcal{V}_{T/2,b}(e,\varepsilon,\T,\Theta))
			\\
			& \le c(\v,b_\varepsilon^+)  \left(\tilde{\nu}(\check{N}_{\varepsilon+ O(\varepsilon e^{-C T/2})})\tilde{\nu}_{\involution}(N_{\varepsilon+ O(\varepsilon e^{-C T/2})})\vol_M(\Theta_\varepsilon^+) + o_T(1)\right)\frac{e^{\delta_\v T}}{T^{(\rank-1)/2}}
			\\
			&  = c(\v,b)\bigg(\tilde{\nu}(\check{N}_{\varepsilon})\tilde{\nu}_{\involution}(N_{\varepsilon})\vol_M(\Theta)(1+O(\varepsilon)) + o_T(1)\bigg)\frac{e^{\delta_\v T}}{T^{(\rank-1)/2}}
			\\
			& =c(\v,b)\left(\frac{\BMS(\tilde{\mathcal{B}}(e,\varepsilon))}{b_{\rank}(\varepsilon)}\vol_M(\Theta) (1+O(\varepsilon))  + o_T(1)\right)\frac{e^{\delta_\v T}}{ T^{(\rank-1)/2}}
		\end{aligned}.
	\end{equation}
	Moreover,
	\begin{equation}
		\label{eqn:CountinginVT4}
		\begin{aligned}[b]
			&\#(\Gamma\cap \mathcal{V}_{T/2,b}(e,\varepsilon,\T,\Theta)) 
			\\
			& \qquad \qquad \le c(\v,b)\bigg( \tilde{\nu}(\check{N}_{\varepsilon})\tilde{\nu}_{\involution}(N_{\varepsilon})\vol_M(\Theta_\varepsilon^+) + o_T(1)\bigg)\frac{e^{\delta_\v T/2}}{(T/2)^{(\rank-1)/2}}.
		\end{aligned}
	\end{equation}
	Combining \eqref{eqn:CountinginVT3} and \eqref{eqn:CountinginVT4}, we obtain the desired inequality.

	\subsection*{Application of a closing lemma.}
	For $g_0\in G$ and $T, \varepsilon >0$, let
	\[\mathcal{W}_{T,b}(g_0,\varepsilon,\T,\Theta)=\{gamg^{-1}:g\in\mathcal{B}(g_0,\varepsilon),am\in \T_{T,b} \Theta\}.\]
	Note that the sets $\mathcal{V}_{T,b}(g_0,\varepsilon,\T,\Theta)$ and $\mathcal{W}_{T,b}(g_0,\varepsilon,\T,\Theta)$ are similar but the latter consists only of loxodromic elements and the former does not. We relate $\Gamma \cap \mathcal{W}_{T,b}(g_0,\varepsilon,\T,\Theta)$ to $\Gamma \cap \mathcal{V}_{T,b}(g_0,\varepsilon,\T,\Theta)$ by using the following closing lemma for regular directions.
	
	\begin{lemma} \cite[Lemma 2.7]{CF23}\label{lem:EffectiveClosingLemma}
		There exists $s_0>0$ for which the following holds. Let $\varepsilon>0$ be sufficiently small and $g_0\in G$. Suppose there exist $g_1,g_2\in\mathcal{B}(g_0,\varepsilon)$ and $\gamma\in G$ such that 
		\begin{equation}
			\label{eqn:EffectiveClosingLemmaHypothesis1}
			g_1 \tilde{a}_\gamma\tilde{m}_\gamma = \gamma g_2
		\end{equation}
		for some $\tilde{m}_\gamma\in M$ and $\tilde{a}_\gamma\in A$ with
		\begin{equation}
			\label{eqn:EffectiveClosingLemmaHypothesis2}
			s=\min_{\alpha \in \Phi^+} \alpha(\log \tilde{a}_\gamma)\ge s_0.
		\end{equation}
		Then there exist $g\in\mathcal{B}(g_0,\varepsilon + O(\varepsilon e^{-s}))$, $a_\gamma \in A$ and $m_\gamma \in M$ such that 
		\[\gamma=ga_\gamma m_\gamma g^{-1}.\]
		Moreover, $a_\gamma \in\tilde{a}_\gamma A_{O(\varepsilon)}$ and $m_\gamma\in\tilde{m}_\gamma M_{O(\varepsilon)}$.
	\end{lemma}
	
	The next \cref{lem:ComparisonLemmaForCountingInVTandWT} is a precise formulation of the property that the sets $\Gamma \cap \mathcal{W}_{T,b}(g_0,\varepsilon,\T,\Theta)$ and $\Gamma \cap \mathcal{V}_{T,b}(g_0,\varepsilon,\T,\Theta)$ are approximately the same. The proof of \cref{lem:ComparisonLemmaForCountingInVTandWT} uses the above Closing \cref{lem:EffectiveClosingLemma}. The essential reason why we are able to use this closing lemma is because only a finite volume part of the tube $\T$ is too close to the walls of the Weyl chamber to apply the closing lemma. Excluding this finite volume, we obtain a comparison between $\Gamma \cap \mathcal{W}_{T,b}(g_0,\varepsilon,\T,\Theta)$ and $\Gamma \cap \mathcal{V}_{T,b}(g_0,\varepsilon,\T,\Theta)$.
	
	\begin{lemma}
		\label{lem:ComparisonLemmaForCountingInVTandWT}
		There exists $C > 0$ such that for all sufficiently large $T,T'$ with $T>T'$, we have 
		\begin{multline*}
			\Gamma \cap \left(\mathcal{V}_{T,b_\varepsilon^-}(g_0,\varepsilon-O(\varepsilon e^{-CT'})),\T^-_\varepsilon,\Theta^-_\varepsilon)- \mathcal{V}_{T',b}(g_0,\varepsilon,\T,\Theta)\right)
			\\
			\subset \Gamma \cap \mathcal{W}_{T,b}(g_0,\varepsilon,\T,\Theta)
		\end{multline*}
		where $\mathsf{K}_\varepsilon^-, b_\varepsilon^-$ are defined by the equation $\bigcap_{w \in \LieA_{O(\varepsilon)}}(\T_{T,b}+w) = \T_T(\v,\mathsf{K}_\varepsilon^-,b_\varepsilon^-)$, $\T_\varepsilon^- = \T(\v,\mathsf{K}_\varepsilon^-)$ and $\Theta^-_\varepsilon = \bigcap_{g_1,g_2 \in M_{O(\varepsilon)}}g_1\Theta g_2$.
	\end{lemma}
	
	\begin{proof}
		As in \cref{lem:VTandST}, since $\v\in \inte \fa^+$,  there exists $T_0>0$ sufficiently large  such that the linear forms $\psi_\v$ and $\alpha \in \Phi^+$  are all positive on $\T - \T_{T_0,b}$. 
		Therefore there exists a constant $C>0$ such that
		$$C\psi_\v  (w) \le \min_{\alpha\in\Phi^+}\alpha (w)$$
		for all $w\in\LieA^+$ in some $\varepsilon_0$-neighborhood of $\T- \T_{T_0,b}$. 
		
		For $T>0$, let $\T_{T,\varepsilon}^- = \T_T(\v,\mathsf{K}_\varepsilon^-,b_\varepsilon^-)$. Fix $T>T'>T_0$ and assume $T'$ is sufficiently large so that if $w \in \T - \T_{T',b}$, then $$\min_{\alpha\in\Phi^+}\alpha(w) > CT'>s_0$$ 
		where $s_0$ is as in \cref{lem:EffectiveClosingLemma}. Suppose 
		\[\gamma\in \Gamma \cap \left(\mathcal{V}_{T,b_\varepsilon^-}(g_0,\varepsilon-O(\varepsilon e^{-CT'})),\T^-_\varepsilon,\Theta^-_\varepsilon)- \mathcal{V}_{T',b}(g_0,\varepsilon,\T,\Theta)\right).\]
		Then $\gamma = g_1\exp(w)mg_2$ where $g_1,g_2 \in \mathcal{B}(g_0,\varepsilon -O(\varepsilon e^{-CT'})))$, $w \in \T^-_{T,\varepsilon}-\T_{T',b}$, and $m \in \Theta^-_\varepsilon$. By \cref{lem:EffectiveClosingLemma}, we have $\gamma = g\exp(w')m'g^{-1}$ for some 
		$$g \in \mathcal{B}\left(g_0,\varepsilon - O(\varepsilon e^{-CT'})+O(\varepsilon e^{-\min_{\alpha\in\Phi^+}\alpha(w)})\right) \subset \mathcal{B}(g_0,\varepsilon),$$ $w' \in w A_{O(\varepsilon)}$ and $m'\in mM_{O(\varepsilon)}$. It follows that $w\in\T_{T,b}$ and $m' \in \Theta$, so $\gamma \in \Gamma \cap \mathcal{W}_{T,b}(g_0,\varepsilon,\T,\Theta)$. 
	\end{proof}
	
	In the next \cref{lem:BoundOnCountingPrimitiveHyperbolics}, we show that there are just as many primitive elements in $\Gamma \cap \mathcal{W}_{T,b}(g_0,\varepsilon,\Theta)$ as nonprimitive elements. This is what allows us to consider only primitive elements in $\Gamma$ as in the joint equidistribution \cref{JointEquidistribution}. The proof uses
	\cref{lem:ComparisonLemmaForCountingInVTandWT} and \cref{prop:CountingInVT} to get an estimate for $\#\primGamma \cap \mathcal{W}_{T,b}(g_0,\varepsilon,\Theta)$.
	
	\begin{lemma}
		\label{lem:BoundOnCountingPrimitiveHyperbolics}
		Suppose $g_0 \in G$ with $\Gamma g_0M \in \supp\BMS$ and $\vol_M(\Theta) > 0$. Then for all sufficiently small $\varepsilon>0$ and sufficiently large $T$, we have
		\[\#\Gamma\cap(\mathcal{W}_{T,b}(g_0,\varepsilon,\Theta)-\mathcal{W}_{2T/3,b}(g_0,\varepsilon,\Theta))\le\#\primGamma\cap\mathcal{W}_{T,b}(g_0,\varepsilon,\Theta).\]
	\end{lemma}
	
	\begin{proof}
		Let $\Gamma_{\mathrm{prim}^k}=\{\sigma^k:\sigma\in\primGamma\}$. We observe that
		\begin{align*}
			&\#\primGamma\cap\mathcal{W}_{T,b}(g_0,\varepsilon,\T,\Theta)&
			\\
			&  =\#\Gamma\cap\mathcal{W}_{T,b}(g_0,\varepsilon,\T,\Theta)\-\#\left(\bigcup_{k\ge 2}\Gamma_{\mathrm{prim}^k}\cap\mathcal{W}_{T,b}(g_0,\varepsilon,\T,\Theta)\right)
			\\
			&  \ge \#\Gamma\cap\mathcal{W}_{T,b}(g_0,\varepsilon,\T,\Theta)-\#\left(\bigcup_{k\ge 2}\Gamma\cap\mathcal{W}_{T/k,\hat{b}}(g_0,\varepsilon,\hat{\T},\sqrt[k]{\Theta})\right)
		\end{align*}
		where $\sqrt[k]{\Theta}: = \{m \in M : m^k \in \Theta\}$ and $\hat{\T}$ is the essential tube obtained from $\T$ by replacing $\mathsf{K}$ with the convex hull of $\mathsf{K}\cup \{0\}$ and $\hat{b}$ is any continuous extension of $b$ to the convex hull of $\mathsf{K} \cup \{0\}$. It suffices to show that for all sufficiently large $T$, we have
		\begin{equation}
			\label{eqn:BoundOnCountingPrimitiveHyperbolics1}
			\#\left(\bigcup_{k\ge 2}\Gamma\cap\mathcal{W}_{T/k,\hat{b}}(g_0,\varepsilon,\hat{\T},\sqrt[k]{\Theta})\right) \le \#\Gamma\cap\mathcal{W}_{2T/3,b}(g_0,\varepsilon,\T,\Theta).
		\end{equation}
		Since $\mathcal{W}_{T/k,\hat{b}}(g_0,\varepsilon,\hat{\T},\sqrt[k]{\Theta}) \subset \mathcal{V}_{T/k,\hat{b}}(g_0,\varepsilon,\hat{\T},\sqrt[k]{\Theta})$ and $\Gamma\cap\mathcal{W}_{T/k,\hat{b}}(g_0,\varepsilon,\hat{\T},\sqrt[k]{\Theta})$ is empty when $T/k$ is sufficiently small, using \cref{prop:CountingInVT} we get
		\begin{equation}
			\label{eqn:BoundOnCountingPrimitiveHyperbolics2}	
			\#\left(\bigcup_{k\ge2}\Gamma\cap\mathcal{W}_{T/k,\hat{b}}(g_0,\varepsilon,\hat{\T},\sqrt[k]{\Theta})\right)=O\left(T\frac{e^{ \delta_\v T/2}}{T^{(\rank-1)/2}}\right).
		\end{equation}
		Using \cref{lem:ComparisonLemmaForCountingInVTandWT}, \cref{prop:CountingInVT} and $\vol_M(\Theta) > 0$, we have 
		\begin{equation}
			\label{eqn:BoundOnCountingPrimitiveHyperbolics3}
			\#\Gamma\cap\mathcal{W}_{2T/3,b}(g_0,\varepsilon,\T,\Theta) \ge O\left(\frac{e^{2\delta_\v T/3}}{T^{(\rank-1)/2}}\right).
		\end{equation}
		The inequality \eqref{eqn:BoundOnCountingPrimitiveHyperbolics1} now follows from \eqref{eqn:BoundOnCountingPrimitiveHyperbolics2} and \eqref{eqn:BoundOnCountingPrimitiveHyperbolics3}.
	\end{proof}
	
	\subsection*{Joint Equidistribution}
	Our next goal is to prove \cref{thm:MuTJointEquidistribution}.
	For each $T>0$, we define a Radon measure $\eta_T$ on $\Omega \times [M]$ by the following: for $f \in \mathrm{C}_{\mathrm{c}}(\Omega)$ and $\varphi \in \mathrm{Cl}(M)$, let 
	
	\begin{equation*}
		\eta_{T}(f\otimes\varphi)=\sum_{
			[\gamma] \in [\primGamma],\, \lambda(\gamma) \in \T_{T,b}} \int_{C_\ga} f \cdot \varphi(m(\gamma)).
	\end{equation*} We will prove \cref{thm:MuTJointEquidistribution} by using the asymptotic in \cref{prop:CountingInVT}. Let $g_0 \in G$ with $\Gamma g_0M \in \supp\BMS$, $\Theta \subset M_\Gamma$ be a conjugation-invariant Borel subset with $\vol_M(\Theta) > 0$ and $\vol_M(\partial \Theta) = 0$ and let $\varepsilon>0$ be sufficiently small as in \cref{prop:CountingInVT}.

	\begin{lemma}\cite[Lemma 6.3]{CF23} \label{lem:MuTAndCountingPrimitiveHyperbolics}  
		For all sufficiently large $T>1$, we have
		\[\eta_{T}(\tilde{\mathcal{B}}(g_0,\varepsilon)\otimes\Theta)=b_{\rank}(\varepsilon)\cdot\#(\primGamma\cap\mathcal{W}_{T,b}(g_0,\varepsilon,\T,\Theta)).\]
	\end{lemma}
	
	In view of \cref{lem:MuTAndCountingPrimitiveHyperbolics}, we can now use \cref{lem:ComparisonLemmaForCountingInVTandWT} and \cref{lem:BoundOnCountingPrimitiveHyperbolics} to prove the following lemma comparing $\eta_T$ with $\mathcal{V}_{T,b}$.
	
	\begin{lemma}[Comparison Lemma]  
		\label{lem:ComparisonLemma}
		For all sufficiently large $T>1$, we have
		\begin{multline*}
			b_{\rank}(\varepsilon)\cdot\#\Gamma\cap\left(\mathcal{V}_{T,b_\varepsilon^-}(g_0,\varepsilon-O(\varepsilon e^{-2CT/3})),\T^-_\varepsilon,\Theta^-_\varepsilon)-\mathcal{V}_{2T/3,b}(g_0,\varepsilon,\T,\Theta)\right)
			\\
			\le \eta_{T}(\tilde{\mathcal{B}}(g_0,\varepsilon)\otimes\Theta)\le b_{\rank}(\varepsilon)\cdot\#\Gamma\cap\mathcal{V}_{T,b}(g_0,\varepsilon,\T,\Theta)
		\end{multline*}
		where $C$, $\Theta^-_\varepsilon$ and $\T^-_{T,\varepsilon}$ are as in \cref{lem:ComparisonLemmaForCountingInVTandWT}. 
	\end{lemma}
	
	\begin{proof}
		The upper bound follows directly from \cref{lem:MuTAndCountingPrimitiveHyperbolics} and the inclusion $\Gamma\cap\mathcal{W}_T(g_0,\varepsilon,\T,\Theta) \subset \Gamma\cap\mathcal{V}_T(g_0,\varepsilon,\T,\Theta)$. The lower bound follows by using \cref{lem:MuTAndCountingPrimitiveHyperbolics}, \cref{lem:BoundOnCountingPrimitiveHyperbolics}, and \cref{lem:ComparisonLemmaForCountingInVTandWT}:
		\begin{align*}
			& \eta_{T}(\tilde{\mathcal{B}}(g_0,\varepsilon)\otimes\Theta)
			\\
			& = b_{\rank}(\varepsilon)\#(\primGamma\cap\mathcal{W}_{T,b}(g_0,\varepsilon,\T,\Theta))
			\\
			& \ge b_{\rank}(\varepsilon)\#\Gamma\cap(\mathcal{W}_{T,b}(g_0,\varepsilon,\T,\Theta)-\mathcal{W}_{2T/3,b}(g_0,\varepsilon,\T,\Theta))
			\\
			& \ge
			b_{\rank}(\varepsilon)\#\Gamma\cap(\mathcal{V}_{T,b_\varepsilon^-}(g_0,\varepsilon -O(\varepsilon e^{-2CT/3})),\T^-_\varepsilon,\Theta^-_\varepsilon)-\mathcal{V}_{2T/3,b}(g_0,\varepsilon,\T,\Theta)).
		\end{align*}
	\end{proof}
	
	Combining \cref{prop:CountingInVT} and \cref{lem:ComparisonLemma}, we obtain the following asymptotic for $\eta_{T}(\tilde{\mathcal{B}}(g_0,\varepsilon)\otimes\Theta)$.
	
	\begin{proposition}        
		\label{prop:MuTOfFlowBox}
		We have
		\begin{multline*}
			\eta_{T}(\tilde{\mathcal{B}}(g_0,\varepsilon)\otimes\Theta)
			=c(\v, b) \left( \BMS(\tilde{\mathcal{B}}(g_0,\varepsilon))\vol_M(\Theta) (1+O(\varepsilon)) +o_T(1)\right)\frac{e^{\delta_\v T}}{T^{(\rank-1)/2}}
		\end{multline*}
		where $c(\v,b)$ is as in \eqref{cv}.
	\end{proposition}
	
	\begin{proof}
		Using the asymptotics from \cref{prop:CountingInVT} in the inequality in \cref{lem:ComparisonLemma} gives
		\begin{multline*}
			\eta_{T}(\tilde{\mathcal{B}}(g_0,\varepsilon)\otimes\Theta)
			\le c(\v,b) \left( \BMS(\tilde{\mathcal{B}}(g_0,\varepsilon))\vol_M(\Theta) (1+O(\varepsilon)) +o_T(1)\right)\frac{e^{\delta_\v T}}{ T^{(\rank-1)/2}}
		\end{multline*}
		and
		\begin{align*}
			&\eta_{T}(\tilde{\mathcal{B}}(g_0,\varepsilon)\otimes\Theta)
			\\
			& \ge b_{\rank}(\varepsilon)\cdot\#\Gamma\cap\left(\mathcal{V}_{T,b_\varepsilon^-}(g_0,\varepsilon-O(\varepsilon e^{-2CT/3})),\T^-_\varepsilon,\Theta^-_\varepsilon)-\mathcal{V}_{2T/3,b}(g_0,\varepsilon,\T,\Theta)\right)
			\\
			&= c(\v,b_\varepsilon^-)
			\\
			& \cdot\left(\frac{b_{\rank}(\varepsilon)\BMS(\tilde{\mathcal{B}}(g_0,\varepsilon-O(\varepsilon e^{-\tfrac{2CT}{3}})))}{b_{\rank}(\varepsilon-O(\varepsilon e^{-\tfrac{2CT}{3}}))}\vol_M(\Theta^-_\varepsilon)(1+O(\varepsilon))+o_T(1) \right)\frac{e^{\delta_\v T}}{T^{(\rank-1)/2}}
			\\
			& \qquad \qquad - c(\v,b) \left(\BMS(\tilde{\mathcal{B}}(g_0,\varepsilon))\vol_M(\Theta) (1+O(\varepsilon))+o_T(1) \right)\frac{e^{2\delta_\v T/3}}{(2T/3)^{(\rank-1)/2}}
			\\
			& =c(\v,b) \left( \BMS(\tilde{\mathcal{B}}(g_0,\varepsilon))\vol_M(\Theta) (1+O(\varepsilon))+o_T(1) \right)\frac{e^{\delta_\v T}}{ T^{(\rank-1)/2}}.
		\end{align*}
	\end{proof}
	\subsection*{Proofs of \cref{thm:MuTJointEquidistribution} and \cref{JointEquidistribution}}
	The left-hand side of the asymptotic formula of \cref{thm:MuTJointEquidistribution} is precisely $\eta_{T}(f\otimes\varphi)$. \cref{thm:MuTJointEquidistribution}  now follows from \cref{prop:MuTOfFlowBox} and a standard partition of unity argument (\cite[Theorem 5.17]{MMO14}, \cite[Theorem 6.12]{CF23}).
	
	We now deduce \cref{JointEquidistribution} from \cref{thm:MuTJointEquidistribution}.
	Let $f \in \mathrm{C}_{\mathrm{c}}(\Omega)$ and $\varphi \in \mathrm{Cl}(M)$. Since $\psi_\v(w) \le T + \sup b$ for all $w \in \T_{T,b}$, we have 
	$$(T+\sup b)\sum_{
		[\gamma] \in [\primGamma], \, \lambda(\gamma) \in \T_{T,b}} \frac{\int_{C_\ga} f}{\psi_\v(\lambda(\gamma))}\varphi(m(\gamma)) \ge \eta_{T}(f\otimes\varphi).$$ 
	On the other hand, for any $\varepsilon>0$, we have
	\begin{align*}
		& \tfrac{1}{e^{\delta_\v T} T^{(1-\rank)/2}}\sum_{
			[\gamma] \in [\primGamma] ,\, \lambda(\gamma) \in \T_{T,b}} \frac{\int_{C_\ga} f}{\psi_\v(\lambda(\gamma))}\varphi(m(\gamma))
		\\ 
		&=\tfrac{1}{e^{\delta_\v T} T^{(1-\rank)/2}}\sum_{
			[\gamma] \in [\primGamma] ,\, \lambda(\gamma) \in \T_{(1-\varepsilon)T,b}} \frac{\int_{C_\ga} f}{\psi_\v(\lambda(\gamma))}\varphi(m(\gamma))
		\\
		& \qquad \qquad +\tfrac{1}{e^{\delta_\v T} T^{(1-\rank)/2}}\sum_{
			[\gamma] \in [\primGamma] ,\, \lambda(\gamma) \in \T_{T,b}-\T_{(1-\varepsilon)T,b}} \frac{\int_{C_\ga} f}{\psi_\v(\lambda(\gamma))}\varphi(m(\gamma))
		\\
		&\le \tfrac{1}{e^{\delta_\v T} T^{(1-\rank)/2}}O(\eta_{(1-\varepsilon)T}(f\otimes\varphi)) 
		\\
		& \qquad \qquad +\tfrac{1}{e^{\delta_\v T} T^{(1-\rank)/2}}\cdot\tfrac{1}{(1-\varepsilon)T+\inf b}\left(\eta_{T}(f\otimes\varphi)-\eta_{(1-\varepsilon)T}(f\otimes\varphi)\right)
		\\
		&=O(Te^{-\varepsilon\delta_\v  T}) + \tfrac{1}{1-\varepsilon+(\inf b)/T}\cdot\tfrac{1}{e^{\delta_\v T}T^{(1-\rank)/2}}\eta_{T}(f\otimes\varphi).
	\end{align*}
	Using the asymptotic for $\eta_T(f\otimes\varphi)$ from \cref{thm:MuTJointEquidistribution} in both inequalities, taking $T\to\infty$ and $\varepsilon \to 0$ gives the desired asymptotic.	
	\qed

	\section{Cartan projections in tubes}
	\label{sec:Cartan}
	
	In this section, we prove an asymptotic for the number of Cartan projections of a Zariski dense Anosov subgroup in a given essential tube (\cref{Cartan}). Recall the Cartan projection $\mu:G\to \fa^+$.
	
	\subsection*{Cartan projections in tubes.}
	
	Let $\Gamma<G$ be a Zariski dense Anosov subgroup. Fix 
	$$\text{an essential tube }\T=\T(\v,\mathsf{K}) = \R\v + \mathsf{K}\quad\text{and}\quad  b \in C(\mathsf{K}).$$
	
	Recall that associated to $\v$ are the $\Gamma$-conformal measures $\nu=\nu_{\v}$ and $\nu_{\i}=\nu_{\involution(\v)}$. Fix Borel subsets $\Xi_1 ,\Xi_2\subset K$ such that $\Xi_1 M=\Xi_1$,
	$M \Xi_2 =\Xi_2$, and  $\nu(\partial\Xi_1) = \nu_{\involution}(\partial\Xi_2^{-1}) = 0$. We will abuse notation and also view $\Xi_1, \Xi_2^{-1}$ as subsets of $\Fboundary\cong K/M$. 
	
	Recall the notation:
	$$\T_{T,b}=\T_T(\mathsf{v},\mathsf{K},b)=\{t\mathsf{v}+u \in \limitcone : 0 \le t \le T+b(u), \, u \in \mathsf{K}\}.$$

	\begin{thm}[Cartan projections in tubes]
		\label{Cartan}
		We have as $T \to \infty$,
		\begin{align*}
			\#\Gamma \cap \Xi_1\exp(\T_{T,b})\Xi_2 
			\sim c(\v, b) \cdot\nu(\Xi_1) \cdot \nu_{\involution}(\Xi_2^{-1})\frac{e^{\delta_\v T}}{ T^{(\rank-1)/2}}
		\end{align*}
		where $c(\v, b)$ is as in \eqref{cv}. 
	\end{thm}
	
	An immediate corollary is:
	\begin{corollary}
		\label{CartanCount}
		As $T \to \infty$ we have
		$$\#\{\gamma\in\Gamma:\mu(\gamma) \in \T_{T,b}\} \sim c(\v, b) \, \cdot\frac{e^{\delta_\v T}}{ T^{(\rank-1)/2}}.$$
		
	\end{corollary}
	
	\cref{CartanCount} together with \cref{Counting} now implies the following:
	\begin{corollary}\label{pret}
		We have  as $T \to \infty$,
		$$\frac{\#\{\gamma \in \Gamma: 
			\mu(\gamma) \in \T_{T,b}\} }{\#\{[\gamma] \in [\Gamma]:\lambda(\gamma) \in \T_{T,b}\} }\sim 
		\frac{[M:M_\Gamma]}{|m_{\mathcal{X}_\v}|} T.$$
	\end{corollary}

	\subsection*{Proof of \cref{Cartan}}
	In \cite[Section 9]{ELO20},
	an asymptotic for the number of Cartan projections in cones  was obtained for  certain special kind of norms. In principle, it is not clear whether their result can be extended to deal with the Euclidean norm counting in cones. However for counting in tubes, the fact that the tubes contain only one direction implies that all norms are essentially the same restricted to tubes.
	Together with integral computation in \cref{lem:MainTermAsymptotic}, this enables us to use the approach of \cite[Section 9]{ELO20} to prove \cref{Cartan}. To be precise, consider the following bounded subset for each $T$:
	let $$R_{T,b} = R_T(\Xi_1,\Xi_2,\T,b) = \Xi_1\exp(\T_{T,b})\Xi_2.$$
	
	For a given bounded subset $B\subset G$,  define the counting function $F_{B}: \Gamma \backslash G \to \R$ by 
	$$F_B(\Gamma g) = \sum_{\gamma \in \Gamma}\mathbbm{1}_B(\gamma g) = \# (\Gamma \cap Bg^{-1}).$$

		We claim that   for any $\Psi \in C_{\mathrm{c}}(\Gamma\backslash G)$, we have
		$$\lim\limits_{T\to\infty}\frac{\langle F_{R_{T,b}}, \Psi\rangle }{e^{\delta_\v T} T^{(1-\rank)/2}} = c(\v, b) \cdot\nu(\Xi_1) \cdot \BR(\Psi \ast \mathbbm{1}_{\Xi_2})$$
		where $\Psi \ast \mathbbm{1}_{\Xi_2}(x) = \int_{\Xi_2}\Psi(xk)\,dk$.
	
		The Haar measure $dg$ on $G$ can be written 
		$$dg = \zeta(w) \,dk_1d\,k_2\,dw$$
		where  $g=k_1\exp(w)k_2 \in K\exp(\LieA^+)K$, 
		$\zeta(w) = \prod_{\alpha \in \Phi^+}2\sinh(\alpha(w))$; here
		$dk$  and $dw$ denote the Haar measures on $K$ and $\LieA^+$ respectively 
		(\cite[Proposition 5.28]{Kna86}, c.f. \cite[Theorem 8.1]{OS13} for the normalization).
		
		Using this formula as in the proof of \cite[Proposition 9.10]{ELO20} and the decomposition of $\T_{T,b}$ as in \cref{lem:MainTermAsymptotic}, we have
		\begin{multline*}
			\frac{\langle F_{R_{T,b}}, \Psi\rangle }{e^{\delta_\v T}T^{(1-\rank)/2}}
			= \frac{1}{e^{\delta_\v T}T^{(1-\rank)/2}}\int_{\ker\psi_\v}\int_{R_T(u)}\bigg(e^{\delta_\v t}e^{-2\rho(t\v+\sqrt{t}u)}\zeta(t\v+\sqrt{t}u)
			\\
			\cdot t^{(\rank-1)/2}e^{2\rho(t\v+\sqrt{t}u)-\delta_\v t}\int_{k \in \Xi_1}\Psi \ast \mathbbm{1}_{\Xi_2}(\Gamma ka_{t\v+\sqrt{t}u})\,d(\Gamma k)\bigg)\,dt\,du.
		\end{multline*}
		Note that 
		$$\lim_{t\to\infty}e^{-2\rho(t\v+\sqrt{t}u)}\zeta(t\v+\sqrt{t}u) = 1 .$$ By \cite[Proposition 8.11]{ELO20}, we have
		\begin{multline}
			\label{eqn:Cartan1}
			\lim_{t\to\infty}t^{(\rank-1)/2}e^{2\rho(t\v+\sqrt{t}u)-\delta_\v t}\int_{ k \in \Xi_1}\Psi \ast \mathbbm{1}_{\Xi_2}(\Gamma ka_{t\v+\sqrt{t}u})\,d(\Gamma k) 
			\\
			= \frac{\kappa_\v}{|m_{\mathcal{X}_\v}|} e^{-I(u)}\nu^K(\Xi_1)\BR(\Psi \ast \mathbbm{1}_{\Xi_2}).
		\end{multline}
		Then we can apply the Lebesgue dominated convergence theorem, \eqref{eqn:Cartan1} and \cref{lem:MainTermAsymptotic} to conclude that
		\begin{equation*}
			\begin{aligned}[b]
				& \lim\limits_{T\to\infty}\frac{\langle F_{R_{T,b}}, \Psi\rangle }{e^{\delta_\v T}T^{(1-\rank)/2}}
				\\
				&  = \lim\limits_{T\to\infty}\frac{1}{e^{\delta_\v T}T^{(1-\rank)/2}}\int_{\ker\psi_\v}\int_{R_T(u)}e^{\delta_\v t}\frac{\kappa_\v}{|m_{\mathcal{X}_\v}|} e^{-I(u)}\nu(\Xi_1)\BR(\Psi \ast \mathbbm{1}_{\Xi_2})\,dt\,du
				\\
				& = \frac{\kappa_\v}{\delta_\v|m_{\mathcal{X}_\v}|}\int_{\mathsf{K}} e^{\delta_\v b(u)}\,du\cdot\nu(\Xi_1)\BR(\Psi \ast \mathbbm{1}_{\Xi_2}).
			\end{aligned}
		\end{equation*}

		This proves the claim. \cref{Cartan} is then proved in the same way as \cite[Corollary 9.21]{ELO20}.
		
		\subsection*{Proof of \cref{main00}}
		\cref{main00} is deduced from \cref{CartanCount} in the same way that \cref{main0} was deduced from \cref{Counting}.	

		\section{Applications to correlations of spectra and the growth indicator}
		\label{sec:Correlation}
		
		In this section, we prove \cref{main2} as an application of \cref{Counting} and \cref{CartanCount}. 
		For this section, let $G_1, \dots, G_d$ be connected simple real algebraic groups of rank one. For each $i$, we use the same notations for Lie subgroups of $G_i$ as introduced in \cref{sec:Preliminaries} but with $i$ as subscript. For each $i$, let $(X_i=G_i/K_i, \mathsf d_i)$ denote the associated Riemannian symmetric space. Let $q_i \in X_i$ be the point stabilized by a $K_i$. We identify each $\LieA_i^+$ with $[0,\infty)$ using the induced norm on $\LieA_i$.

		\subsection*{Correlations of length spectra and correlations of displacement spectra for convex cocompact manifolds}
		We give an application of \cref{Counting} to the correlations of length spectra. We also give an application of \cref{Cartan} to the correlations of displacement spectra in the same setting.
		
		Let $\rho=(\rho_1,\dots,\rho_d):\Ga\to G_1 \times \dots \times G_d$ be a $d$-tuple of faithful representations of a finitely generated group $\Ga$ whose images are Zariski dense convex cocompact subgroups. For each $i=1.\dots,d$, each conjugacy class $[\rho_i(\ga)] \in [\rho_i(\Ga)]$ corresponds to a unique closed geodesic in the convex cocompact manifold $\rho_i(\Ga)\ba X_i$ whose length $\ell_{\rho_i(\gamma)}$ is equal to the Jordan projection of $\rho_i(\gamma)$. We denote by $[m_{\rho_i(\gamma)}] \in [M_i]$ the holonomy class associated to $\rho_i(\gamma)$. 
		Define the {\it{spectrum cone}} $\mathcal{L}_\rho$ of $\rho$ as the smallest closed cone in $\R^d$ containing the set
		$$ \{(\ell_{\rho_1(\gamma)},\dots,\ell_{\rho_d(\gamma)}): \gamma \in \Gamma\}.$$ 
		
		\begin{theorem}[Correlations of length spectra and holonomies]
			\label{Correlation} 
			For any $\v = (v_1,\dots,v_d) \in \interior\limitcone_\rho$, there exists $\delta_\rho(\v)>0$ such that for any $\e_1,\dots,\e_d>0$ and for any conjugation-invariant Borel sets $\Theta_i \subset M_i$ with  null boundaries, we have as $T\to\infty$,
			\begin{multline} 
				\label{eqn:Correlation}
				\#\{[\gamma]\in [\Gamma]:  v_iT \le \ell_{\rho_i(\ga)} \le v_iT+\e_i,\, m_{\rho_i(\ga)}\in \Theta_i, \; 1\le i\le d\} 
				\\
				\sim  c \frac{e^{\delta_\rho(\v)T}}{ T^{(d+1)/{2}}} \prod_{i=1}^d\vol_{M_i}(\Theta_i) \end{multline}
			for some constant $c=c(\v,\varepsilon_1,\dots,\e_d)>0$.
			Moreover, we have 	
   \be\label{ub0} \delta_\rho(\v) \le \min_i \delta_{\rho_i(\Gamma)}v_i.\ee  If $d\ge 2$, we also have
			\begin{equation}
				\label{upperbound}
				\delta_\rho(\v) < \frac{1}{d}\sum_{i=1}^d\delta_{\rho_i(\Gamma)}v_i.
			\end{equation}
		\end{theorem}

		\begin{theorem}[Correlations of displacements]
			\label{Correlation2} 
			For any $\v = (v_1,\dots,v_d) \in \interior\limitcone_\rho$ and
			for any $\e_1,\dots,\e_d>0$, as $T\to\infty$,
			\begin{equation*} 
				\#\{\gamma\in \Gamma:  v_iT\le \mathsf{d}_i(\rho_i(\gamma) q_i,q_i) \le v_iT+\e_i, \; 1 \le i \le d\} 
				\sim  c' c \frac{e^{\delta_\rho(\v)T}}{ T^{{(d-1)}/{2}}}  \end{equation*}
			where  $c=c(\v,\varepsilon_1,\dots,\e_d)$ is as in \cref{Correlation} and $c' = c'(\v,\rho) >0$.
		\end{theorem}

		\begin{remark}\label{indd}
		 We say that $\rho_1,\dots,\rho_d$ are {\it{independent from each other}} if $\rho_j\circ\rho_i^{-1}:\rho_i(\Ga)\to \rho_j(\Ga) $ does not extend to a Lie group isomorphism $G_i\to G_j$ for all $i\ne j$.  
    In the case that $\rho_1,\dots,\rho_d$ are not independent from each other, we observe in the proof of \cref{Correlation} that $\limitcone_\rho$ has empty interior and hence the above theorems are vacuous in that case. 
		\end{remark}
		
		\subsection*{Proofs of Theorems \ref{Correlation} and \ref{Correlation2}}
		
		The \emph{self-joining} of $\Gamma$ via $\rho=(\rho_1,\dots,\rho_d)$ is the discrete subgroup
		\begin{equation}
			\label{eqn:selfjoining}
			\Gamma_\rho =\{ \rho(\gamma)=(\rho_1(\gamma), \dots, \rho_d(\gamma)):\gamma \in \Gamma\} < \prod_{i=1}^dG_i.
		\end{equation}

		By Remark \ref{indd}, it suffices to consider the case where $\rho_1,\dots,\rho_d$ are independent from each other. Since $\rho_i(\Gamma)$ is a Zariski dense subgroup of $G_i$ for each $1\le i\le d$, it follows that $\Gamma_\rho$ is a Zariski dense in $\prod_{i=1}^d G_i$ (cf. \cite[Lemma 4.1]{KO22}). Since $\rho_i(\Gamma)$ is convex cocompact for each $i=1,\dots,d$, $\Gamma_\rho$ is an Anosov subgroup \cite[Theorem 4.11]{GW12}. Identifying $\fa$ with $\br^d$, the Jordan projection of $\rho(\gamma)\in \Ga_\rho$ is
		$$\lambda(\rho(\gamma)) = (\ell_{\rho_1(\gamma)},\dots,\ell_{\rho_d(\gamma)}).$$ Hence $\limitcone_{\rho} = \limitcone_{\Gamma_\rho}$ and in particular $\limitcone_\rho$ has non-empty interior. 
		We prove \cref{Correlation} first using the above setup. For $\gamma \in \Gamma$, let  
		$$m(\rho(\ga)) = (m_{\rho_1(\ga)},\dots,m_{\rho_d(\ga)} ).$$  
		
		We are interested in the asymptotic of 
		$$\#\{[\rho(\gamma)]\in[\Gamma_\rho]: \lambda(\rho(\gamma)) \in \prod_{i=1}^d[v_iT, v_iT+\e_i], \; m(\rho(\ga)) \in \prod_{i=1}^d\Theta_i\}.$$
		
		Up to re-scaling, we may assume that $\v=(v_1, \cdots, v_d)$ is a unit vector in $\interior\limitcone_\rho$. We claim that we can choose a compact subset $\mathsf{K} \subset \ker\psi_\v$ and  $b_1,b_2\in C(\mathsf{K})$ such that the truncated tubes $\T_T(\v,\mathsf{K},b_1)$ and $\T_T(\v,\mathsf{K},b_2)$ satisfy 
		$$\T_T(\v,\mathsf{K},b_1)- \T_T(\v,\mathsf{K},b_2) = \prod_{i=1}^d[v_iT,v_iT+\e_i]$$
		for all sufficiently large $T$. Consider the box $\prod_{i=1}^d[0,\e_i]$. Let $F_1 \subset \R^d$ (resp. $F_2\subset \R^d$) denote the union of the faces of the box  not containing (resp. containing) the origin in $\R^d$. Then it suffices to choose the truncated tubes so that (see \cref{fig:Correlation2} for an illustration of the $d=2$ case)
		\begin{align*}
			& \T_T(\v,\mathsf{K},b_1) = [0,T]\v + F_1; & \T_T(\v,\mathsf{K},b_2) = [0,T]\v + F_2.
		\end{align*} 
		Equivalently, we need to choose $\mathsf{K}$, $b_1$ and $b_2$ so that
		\begin{align*}
			& \{b_1(u)\v+u:u \in \mathsf{K}\} = F_1; & \{b_2(u)\v+u:u \in \mathsf{K}\} = F_2. 
		\end{align*} 
		Noting that $\R^d = \R\v\oplus\ker\psi_\v$, it suffices to check that for each $i =1,2$, the following holds:
		\begin{enumerate}
			\item 
			\label{itm:graph1}
			if $w, w' \in F_i$ and $w\ne w'$, then $w - w' \notin \R\v$;
			
			\item 
			\label{itm:graph2}
			$F_1$ and $F_2$ have the same image under the projection $\R\v\oplus\ker\psi_\v \to \ker\psi_\v$.
		\end{enumerate}
		For (\ref{itm:graph1}), suppose that $w=(w_1,\dots,w_d),w'=(w_1',\dots,w_d') \in F_i$ with $w\ne w'$. If $w$ and $w'$ are in the same face of the box, then there exists $1 \le j \le d$ such that $w_j=w_j'$. Then $w_j-w_j'=0$ and hence $w-w'$ cannot be parallel to $\v \in (0,\infty)^d$. Suppose $w$ and $w'$ are in different faces. If $w,w' \in F_1$, then there exist $1\le j,k\le d$ with $j\ne k$ such that $w_j=v_j\varepsilon_j > w_j'$ and $w_k' = v_k\varepsilon_k > w_k$. Then $w_j-w_j' > 0 > w_k-w_k'$ so $w-w'$ cannot be parallel to $\v$. If $w,w' \in F_2$, then there exist $1\le j,k\le d$ with $j\ne k$ such that $w_j=0 < w_j'$ and $w_k' = 0 < w_k$. Then $w_j-w_j' < 0 < w_k-w_k'$ so $w-w'$ cannot be parallel to $\v$. This establishes (\ref{itm:graph1}). For (\ref{itm:graph2}), note that since $\v$ is not parallel to the faces of the box, any affine line parallel to $\v$ that intersects $F_1 \cup F_2 - F_1 \cap F_2$, must do so at exactly two points which cannot lie in the same $F_i$ by (\ref{itm:graph1}). This establishes (\ref{itm:graph2}).

		Let $\delta_\rho=\psi_{\Ga_\rho}$. Recall from \cref{sec:Preliminaries} that $M_{\rho_i(\Gamma)}=M_i$ since $G_i$ is rank one for each $i=1,\dots,d$.
		Applying \cref{Counting} to each tube $\T_T(\v,\mathsf{K},b_i)$, we have as $T\to\infty$
		\begin{multline*}
			\#\{[\gamma] \in [\Gamma]: \lambda(\rho(\gamma)) \in \T_T(\v,\mathsf{K},b_i), \; m(\rho(\ga)) \in \prod_{i=1}^d\Theta_i\} \\ \sim  C_i\frac{e^{\delta_{\rho}(\v)T}}{ T^{{(d-1)}/{2}}}\prod_{i=1}^d\vol_{M_i}(\Theta_i)
		\end{multline*}
		where  $\delta_\rho(\v)=\psi_{\Ga_\rho}(\v)$ and 
		$		C_i=\tfrac{\kappa_{\mathsf{v}}}{\psi_{\Ga_\rho}(\v) [M:M_\Gamma]}\int_{\mathsf{K}} e^{\delta_{\rho}(\v) b_i(u)} \, du$ for $i=1,2$.
		Note that by construction, $b_1 > b_2$ and hence $C_1 > C_2$. Then taking the difference gives
		\begin{multline*}
			\#\{\rho(\gamma)\in [\Gamma_\rho]: \lambda(\rho(\gamma)) \in \prod_{i=1}^d[v_iT, v_iT+\e_i], \; m(\rho(\ga)) \in \prod_{i=1}^d\Theta_i\}
			\\
			\sim  (C_1-C_2) \frac{e^{\delta_\rho(\v)T}}{ T^{{(d-1)}/{2}}}\prod_{i=1}^d\vol_{M_i}(\Theta_i).
		\end{multline*}
  The upper bounds \eqref{ub0} and \eqref{upperbound} are direct consequences of  \cite[Theorem 1.4 and Corollary 1.6]{KMO21} since
  $\delta_\rho=\psi_{\Ga_\rho}$. This completes the proof of \cref{Correlation} with $c = C_1-C_2$.
		
		Now we prove \cref{Correlation2}. For each $i=1,\dots,d$, we have $\mathsf{d}_i(\rho_i(\gamma) q_i,q_i)=\|\mu(\rho_i(\gamma))\|$. Let
		$$\mu(\rho(\gamma))=(\mathsf{d}_1(\rho_1(\gamma) q_1,q_1),\dots,\mathsf{d}_d(\rho_i(\gamma) q_d,q_d)).$$
		We are now interested in the asymptotic of
		$$\#\{\rho(\gamma)\in\Gamma_\rho: \mu(\rho(\gamma)) \in \prod_{i=1}^d[v_iT, v_iT+\e_i]\}.$$  
		
		\begin{figure}[H]
			\centering
			\scalebox{0.75}{\begin{tikzpicture}
					\begin{scope}[yscale=-0.8,xscale=1,rotate=270]
						\draw[thick,-latex] (0,0) -- (5,0);
						\draw (5.5,0) node{\large $\ell_{\rho_2(\gamma)}$};
						\draw[thick,-latex] (0,0) -- (0,4);
						\draw (0,4.5) node{\large $\ell_{\rho_1(\gamma)}$};
						\draw[dashed,-latex] (0,0) -- (0.8,4);
						\draw (1.25,4.25) node{\large slope $=\mathsf{m}_\rho$};
						\draw[dashed,-latex] (0,0) -- (5.5,1.1);
						\draw (6,1.7) node{\large slope $=\mathsf{M}_\rho$};
						\draw[dashed,-latex] (0,0) -- (5,2.5);
						\draw (5.3,3) node{\large slope $=v_2/v_1$};
						\draw[dashed] (3,1.5) -- (3,0);
						\draw (3,-0.5) node{\large $v_2T$};
						\draw[dashed] (4,1.5) -- (4,0);
						\draw (4,-1) node{\large $v_2T+\varepsilon_2$};
						\draw[dashed] (3,1.5) -- (0,1.5);
						\draw (-0.5,1.2) node{\large $v_1T$};
						\draw[dashed] (3,2) -- (0,2);
						\draw (-0.52,2.7) node{\large $v_1T+\varepsilon_1$};
						\draw[teal, thin, fill = teal, fill opacity=0.2] (4,2) -- (4,1.5) -- (5/3,1/3) -- (0,0) -- (1/9,5/9) -- (3,2)-- (4,2);
						\draw (3.2,3.2) node{\large $\T_T(\v,\mathsf{K},b_1)$};
					\end{scope}
					
					\begin{scope}[yscale=-0.8,xscale=1,rotate=270,shift={(0,8)}]
						\draw[thick,-latex] (0,0) -- (5,0);
						\draw (5.5,0) node{\large $\ell_{\rho_2(\gamma)}$};
						\draw[thick,-latex] (0,0) -- (0,4);
						\draw (0,4.5) node{\large $\ell_{\rho_1(\gamma)}$};
						\draw[dashed,-latex] (0,0) -- (0.8,4);
						\draw (1.25,4.25) node{\large slope $=\mathsf{m}_\rho$};
						\draw[dashed,-latex] (0,0) -- (5.5,1.1);
						\draw (6,1.7) node{\large slope $=\mathsf{M}_\rho$};
						\draw[dashed,-latex] (0,0) -- (5,2.5);
						\draw (5.3,3) node{\large slope $=v_2/v_1$};
						\draw[dashed] (3,1.5) -- (3,0);
						\draw (3,-0.5) node{\large $v_2T$};
						\draw[dashed] (4,1.5) -- (4,0);
						\draw (4,-1) node{\large $v_2T+\varepsilon_2$};
						\draw[dashed] (3,1.5) -- (0,1.5);
						\draw (-0.5,1.2) node{\large $v_1T$};
						\draw[dashed] (3,2) -- (0,2);
						\draw (-0.52,2.7) node{\large $v_1T+\varepsilon_1$};
						\draw[teal, thin, fill = teal, fill opacity=0.2] (3,1.5) -- (4,1.5) -- (5/3,1/3) -- (0,0) -- (1/9,5/9) -- (3,2)-- (3,1.5);
						\draw (3.2,3.2) node{\large $\T_T(\v,\mathsf{K},b_2)$};
					\end{scope}
			\end{tikzpicture}}
			\caption{}
   \label{fig:Correlation2}
		\end{figure}

		Applying \cref{CartanCount} to the truncated tubes $\T_T(\v,\mathsf{K},b_1)$ and $\T_T(\v,\mathsf{K},b_2)$ as above, we obtain
		\begin{multline*}
			\#\{\rho(\gamma)\in\Gamma_\rho: \mu(\rho(\gamma)) \in \prod_{i=1}^d[v_iT, v_iT+\e_i]\}
		\sim  \tfrac{c}{|m_{\mathcal{X}_\v}|} \tfrac{e^{\delta_\rho(\v)T}}{ T^{{(d-1)}/{2}}}\prod_{i=1}^d\vol_{M_i}(\Theta_i).
		\end{multline*}
		This finishes the proof of Theorem \ref{Correlation2} with $c'= \tfrac{1}{|m_{\mathcal{X}_\v}|}$.
		
		\subsection*{Correlation entropy rigidity}
		
		In view of \cref{Correlation} on the correlations of length spectra, we define a correlation entropy function for a general $d$-tuple of faithful representations $\rho=(\rho_1,\dots,\rho_d):\Ga\to G_1 \times \dots \times G_d$ of a group $\Ga$.
		
		\begin{Def}[Correlation entropy function]
			For any $d$-tuple of discrete faithful representations $\rho=(\rho_1,\dots,\rho_d):\Ga\to G_1 \times \dots \times G_d$ of a group $\Ga$, define the correlation entropy function $\delta_\rho:\R^d \to [-\infty, \infty]$ as follows. Given $\v = (v_1,\dots,v_d) \in \R^d$, define
			$$\delta_{\rho}(\v)=\liminf_{\max_i\varepsilon_i\to 0^+}\delta_{\rho}(\v,\e_1,\dots,\e_d)	$$
			where $\delta_{\rho}(\v,\e_1,\dots,\e_d)$ is given by 
			\begin{multline*}
				\liminf_{T\to \infty}\tfrac{1}{T}
				\log\#\{[\gamma]\in [\Gamma]:  {\text{$\rho_i(\ga)$ loxodromic}}, \;  v_iT\le \ell_{\rho_i(\ga)} \le v_iT+\e_i\; 
				\\ \text{for all } i=1,\dots,d \}.
			\end{multline*}
		\end{Def}
		
		In the previous section, we have seen that for a $d$-tuple $\rho=(\rho_1,\dots,\rho_d):\Ga\to G_1 \times \dots \times G_d$ of faithful representations of a finitely generated group $\Ga$ whose images are Zariski dense convex cocompact subgroups, the correlation entropy $\delta_\rho(\v)$ is positive on $\interior\limitcone_\rho$ when $\rho_1,\dots,\rho_d$ are independent from each other \cref{Correlation}. When $\rho_1,\dots,\rho_d$ are not independent from each other, we observed that $\limitcone_\rho$ has empty interior and moreover, $\delta_\rho = -\infty$ on ${\R^d-\limitcone_\rho}$, so the set $\{\v \in \R^d: \delta_\rho(\v) > 0\}$ has empty interior.
		
		This phenomenon in fact holds without the convex cocompact assumption:
		
		\begin{corollary}[Correlation entropy rigidity]
			\label{rigidity}
			Suppose $\rho = (\rho_1,\dots,\rho_d):\Gamma \to G_1\times\cdots\times G_d$ is a $d$-tuple of discrete faithful representations of a countable group $\Ga$ with Zariski dense image. If $\rho_1,\dots,\rho_d$ are independent from each other, then
			the interior of the set $\{\v \in \br^d: \delta_\rho(\v)>0\}$ is 
			a non-empty convex cone. Otherwise, the set $\{\v \in \br^d: \delta_\rho(\v)>0\}$ has empty interior.   
		\end{corollary}
		
		\begin{proof}
			\cref{rigidity} can be deduced from  \cref{Correlation} as follows. Using the same notation as in the previous subsection but without the convex cocompact assumption on the representations $\rho_1,\dots,\rho_d$, we can still consider the self-joining $\Gamma_\rho$. 
 Note that if $\v \in \R^d - \limitcone_{\Gamma_\rho}$, then clearly, $\delta_\rho(\v,\varepsilon_1,\dots,\varepsilon_d) = - \infty$ for all $\varepsilon_1,\dots,\varepsilon_d>0$.
			
			Suppose $\rho_1,\dots,\rho_d$ are independent from each other. Then $\Gamma_\rho$ is Zariski dense and the limit cone $\limitcone_{\Gamma_\rho}$ has non-empty interior. Let $\v=(v_1,\dots,v_d) \in \interior\limitcone_{\Gamma_\rho}$. By  \cite[Proposition 4.3]{Ben97}, there exists a Zariski dense Schottky subgroup $\Gamma' < \Gamma_\rho$ such that $\v \in \interior\limitcone_{\Gamma'}$. Moreover, $\Gamma'$ is Anosov (cf. \cite[Lemma 7.2]{ELO20}). Then as in the proof of \cref{Correlation}, we have
			\begin{align*}
				0 & < \psi_{\Gamma'}(\v) 
				= \lim_{T \to \infty} \tfrac{1}{T}\log\#\{[\gamma' ]\in [\Gamma']: \lambda(\gamma') \in \R\v+\prod_{i=1}^d[0,\varepsilon_i], \; \|\lambda(\gamma')\| \le T\}
				\\
				& \le \lim_{T \to \infty} \tfrac{1}{T}\log\#\{[\rho(\gamma)] \in [\Gamma_\rho]: \lambda(\rho(\gamma)) \in \R\v+\prod_{i=1}^d[0,\varepsilon_i], \; \|\lambda(\rho(\gamma))\| \le T\}
			\end{align*}
			Taking the infimum over $\varepsilon_i \to 0$, we get $0 < \psi_{\Gamma'}(\v) \le \delta_\rho(\v)$. Hence we conclude that
			$$\interior\{\v \in \R^d:\delta_\rho(\v) > 0 \} = \interior\limitcone_{\Gamma_\rho}.$$
			
			If $\rho_1,\dots,\rho_d$ are not independent from each other, then $\limitcone_{\Gamma_\rho}$ is contained in a strictly lower dimensional subspace of $\R^d$. This implies that $\{\v \in \R^d:\delta_\rho(\v) > 0 \} \subset \limitcone_{\Gamma_\rho}$ has empty interior.
		\end{proof}

		\section{Growth indicators using Jordan projections and tubes} \label{sec:last}
		
		Let $\Gamma < G$ be a Zariski dense Anosov subgroup. In this section, we prove \cref{growthrates} which gives equivalent definitions of the growth indicator $\psi_\Ga$, except possibly on the boundary of $\L$, using Jordan projections or Cartan projections and cones or tubes. 
		Recall the definitions of $\psi_\Gamma^{\op{tubes}},\h_\Gamma^{\op{tubes}}$ and $\h_\Gamma^{\op{cones}}$ in \eqref{growthfunctions}.  
		
		\begin{theorem}
			\label{growthrates}
			For any Zariski dense Anosov subgroup $\Ga<G$, 
			$$\psi_\Gamma = \psi_\Gamma^{\op{tubes}} =\h_\Gamma^{\op{tubes}}= \h_\Gamma^{\op{cones}}\quad\text{on $\LieA^+ - \partial\limitcone$}$$
   and 
 \be\label{idd} \psi_\Ga =\psi_\Gamma^{\op{tubes}} \le \h_\Gamma^{\op{tubes}}\le  \h_\Gamma^{\op{cones}} \quad\text{ on $\partial \cal L$}. \ee 
		\end{theorem}
		
	\begin{proof}
		Recall that $\psi_\Gamma, \psi_\Gamma^{\op{tubes}} ,\h_\Gamma^{\op{tubes}}, \h_\Gamma^{\op{cones}}$ are equal to $-\infty$ on $\LieA^+-\limitcone$ and they are degree one homogeneous functions. Then it suffices to consider a unit vector $\v \in \interior\limitcone$. By \cite[Theorem 3.1.1]{Qui02a}, since $\psi_\Ga(\v)>0$,
		$$\psi_\Ga(\v)= \inf_{\text{open cones $\cal C\ni \v$}} \limsup_{T\to \infty} \tfrac{1}{T}\log \#\{\gamma \in \Gamma:\mu(\gamma)\in \cal C, \|\mu(\gamma) \|\le T\} .$$

		By \cref{main0} and \cref{main00}, for any essential tube $\T$ of direction $\v$, we have
		\begin{align*}
			&  \lim_{T\to \infty} \tfrac{1}{T}\log \#\{\gamma \in \Gamma:\mu(\gamma)\in \T, \|\mu(\gamma) \|\le T\} 
			\\
			= \;& \lim_{T\to \infty} \tfrac{1}{T}\log \#\{[\gamma] \in [\Gamma]:\lambda(\gamma)\in \T, \|\lambda(\gamma) \|\le T\}
			\\
			= \;&  \psi_\Ga(\v)>0.
		\end{align*}
		
		This implies that the first two quantities are equal to $\psi_\Gamma^{\op{tubes}}(\v) $ and $\h_\Gamma^{\op{tubes}}(\v)$ respectively, and that 
  $\psi_\Gamma(\v) = \psi_\Gamma^{\op{tubes}}(\v) =\h_\Gamma^{\op{tubes}}(\v)$.

		It remains to show that $\psi_\Gamma(\v) = \h_\Gamma^{\op{cones}}(\v)$. We will use the fact that the definition of $\h_\Gamma^{\op{cones}}$ is independent of the choice of norm $\|\cdot\|$. We consider the more convenient norm $\mathsf{N}:\LieA\to[0,\infty)$ which is induced by the inner product for which $\v$ and $\ker\psi_\v$ are orthogonal to each other and satisfies $\mathsf{N}(\v) = 1$. 
  This norm $\mathsf N$ has the property that $\v$ is the maximal growth direction in the sense that
$\psi_\Ga(\v)=\max_{\mathsf{N}(w)=1}\growthindicator(w)$. To see this, suppose $w \in \LieA^+$ such that $\mathsf{N}(w)=1$. Then we can write $w = t\v + u$ with $t \in \T$, $u \in \ker\psi_\v$ such that $t^2+\mathsf{N}(u)^2=1$. Then $\psi_\Gamma(w) = \psi_\Gamma(t\v+u) \le \psi_\v(t\v+u) = t\psi_\v(\v)$ and the maximum is achieved when $t=1$ and $u=0$, proving the claim.
  Therefore \cite[Corollary 7.8]{CF23} implies that for any open cone $\cal C$ with $\v\in \C$,
			$$\lim_{T\to\infty}\tfrac{1}{T}\log\#\{[\gamma]\in[\Gamma]:\mathsf{N}(\lambda(\gamma)) \le T, \lambda(\gamma) \in \mathcal{C}\} =\psi_\Ga(\mathsf v).$$
     Hence it follows that $\h_\Gamma^{\op{cones}}(\v)  = \growthindicator(\v)$. 
     
     Recall that $\psi_\Ga$ is upper-semicontinuous, and concave.
It can be checked directly using the definitions that the other functions  $\psi_\Gamma^{\op{tubes}}, \h_\Gamma^{\op{cones}}$ and $\h_\Gamma^{\op{tubes}}$ are all upper-semicontinuous.
We can then deduce from Theorem \ref{growthrates} that $\psi_\Ga (\v)\le \min \{\psi_\Gamma^{\op{tubes}}(\v), \h_\Gamma^{\op{cones}}(\v), \h_\Gamma^{\op{tubes}}(\v)\}$ for all $\v\in \partial \L$. 
To see this, let $\v\in \partial \L$. Choose any $\v_0 \in \inte \L$.
Let $\varphi$ denote any of $\psi_\Gamma^{\op{tubes}}$, $\h_\Gamma^{\op{cones}}$, and $\h_\Gamma^{\op{tubes}}$.
By the concavity of $\psi_\Ga$ and Theorem \ref{growthrates},
we have for all $0<t<1$,
$$t\psi_\Ga (\v_0)+(1-t) \psi_\Ga (\v) \le 
\psi_\Ga (t\v_0+ (1-t)\v) = \varphi (t\v_0 +(1-t)\v) .$$

By taking $t\to 0^+$ and the upper-semicontinuity of $\varphi$,
we get $\psi_\Ga(\v)\le \varphi (\v)$, proving claim.
Since $\psi_\Ga^{\op{tubes}}\le \psi_\Ga$, it follows that
$\psi_\Ga^{\op{tubes}}= \psi_\Ga$. Finally, since $\h_\Gamma^{\op{cones}}\ge \h_\Gamma^{\op{tubes}}$ by definition, 
this finishes the proof. \end{proof}

	\end{document}